\documentclass[11pt]{article}

\usepackage[margin=1in]{geometry}
\usepackage{amsmath,amssymb,amsfonts,amsthm}
\usepackage{graphicx}
\usepackage{array}
\usepackage{booktabs}
\usepackage{tabularx}
\usepackage{url}
\usepackage[sort&compress]{natbib}
\usepackage[expansion=false]{microtype}
\usepackage{hyperref}
\hypersetup{colorlinks,citecolor=blue,urlcolor=blue,linkcolor=blue}
\usepackage{setspace}
\theoremstyle{plain}
\newtheorem{theorem}{Theorem}
\newtheorem{lemma}{Lemma}
\newtheorem{proposition}{Proposition}
\newtheorem{corollary}{Corollary}
\theoremstyle{definition}
\newtheorem{definition}{Definition}
\newtheorem{assumption}{Assumption}
\theoremstyle{remark}
\newtheorem{remark}{Remark}

\makeatletter
\g@addto@macro\th@plain{\thm@notefont{\normalfont\bfseries}}
\g@addto@macro\th@definition{\thm@notefont{\normalfont\bfseries}}
\makeatother

\bibpunct[, ]{[}{]}{,}{n}{}{,}

\def\BIBand{and}

\begin{document}

\title{A Polyak--Ruppert Central Limit Theorem for SA-Adam
  with Momentum and Non-Convergent Adaptive Preconditioning}

\author{%
  Sunyoung An and Xiaoming Huo\\
  H.~Milton Stewart School of Industrial and Systems Engineering\\
  Georgia Institute of Technology, Atlanta, GA 30332, USA\\
  \texttt{san49@gatech.edu}, \texttt{huo@gatech.edu}}

\date{}

\maketitle

\begin{abstract}
Adaptive optimizers combining preconditioning, momentum, and weight
decay---Adam and AdamW---are, under Polyak--Ruppert averaging, candidate
engines for one-pass inference.  Does the averaged iterate keep the
classical Polyak--Ruppert central limit theorem (CLT), with sandwich
covariance $H^{-1}SH^{-1}$ (Hessian $H$, gradient covariance $S$), under momentum and non-convergent
preconditioning?  The preconditioner-only analysis does not carry over: with
momentum the canonical decomposition collapses to a tautology.  Treating
the augmented state (iterate, momentum buffer) as a time-varying linear
stochastic approximation (SA), we prove---under local stabilization---positive
drift stability, a non-autonomous Polyak--Ruppert CLT, and a projection
identity.  The upshot: the iterate-marginal covariance is \emph{exactly}
the plain stochastic gradient descent (SGD) sandwich $H^{-1}SH^{-1}$---the adaptivity is asymptotically
invisible.  This holds for
SA-Adam (sub-linearly vanishing momentum gain, $\gamma\in(\alpha,1)$; the
sub-linear regime is essential), not constant-$\beta$ deployed Adam.
Coupled $L_2$ weight decay yields the ridge-penalized sandwich,
extending one-pass inference to regularized problems.
\end{abstract}



\section{Introduction}
\label{sec:intro}

Stochastic gradient descent (SGD) with Polyak--Ruppert iterate averaging
is a standard tool for one-pass stochastic optimization and statistical
inference.  Starting from an initial $x_1$, the recursion
$x_{t+1} = x_t - \eta_t\,\nabla f(x_t, \zeta_t)$---with step size $\eta_t$
and $\zeta_t$ the data sample drawn at step $t$---generates iterates
whose running average $\overline x_n = n^{-1}\sum_{t=1}^n x_t$ satisfies
the classical central limit theorem (CLT)
\begin{equation}
  \sqrt n\,(\overline x_n - x^*) \;\xrightarrow{d}\;
  \mathcal{N}\!\left(0,\; H^{-1}SH^{-1}\right),
  \label{eq:classical-pr}
\end{equation}
under appropriate conditions on the step-size schedule
$\eta_t = \eta_0 t^{-\alpha}$, $\alpha \in (1/2, 1)$
\citep{polyak1992acceleration, ruppert1988efficient}.  The sandwich
limit, with $H$ the population Hessian and $S$ the gradient covariance at
the minimizer $x^*$, is independent of the schedule's tuning constants
$\eta_0$ and $\alpha$ within this regime and is the
Godambe covariance of $M$-estimation, reducing to the Cram\'er--Rao bound
under correct specification \citep{vaart1998asymptotic}.

Practical large-scale optimization, however, rarely reduces to plain SGD or
a single recursion of the type above.  Modern adaptive methods introduce a data-driven preconditioner $P_t$
\citep{duchi2011adaptive, hazan2007logarithmic, tieleman2012rmsprop},
and Adam-type methods further couple it with a momentum buffer $m_t$
from exponential moving averaging (EMA) of past gradients.  The deployed form of Adam \citep{kingma2015adam} is the
canonical example, and its weight-decay variant AdamW
\citep{loshchilov2017decoupled} appends a decoupled decay term
$-\eta_t\lambda x_t$ (decay coefficient $\lambda$).  These optimizers are central to large-scale
training, and Polyak--Ruppert averaging is the standard
device for turning a stochastic-approximation (SA) trajectory into an
asymptotically efficient estimator; their combination is therefore a
natural candidate for one-pass statistical inference.  Weight decay
sharpens this goal statistically: coupled into the gradient ($L_2$
regularization), it is the optimization counterpart of
the ridge penalty, so an inference theory for such coupled updates is an
inference theory for \emph{regularized}---penalized
$M$-estimation---problems, whose efficient limit is the ridge sandwich
$H_\lambda^{-1}S_\lambda H_\lambda^{-1}$ at the penalized minimizer
$x^*_\lambda$.

From an operations-research standpoint, the question is whether
data-adaptive stochastic first-order methods can serve as reliable
one-pass engines for streaming estimation and large-scale stochastic
optimization---processing each datum once while still delivering a
calibrated uncertainty statement online.  Polyak--Ruppert averaging
supplies the classical efficiency benchmark, and adaptive preconditioning
and momentum are the algorithmic mechanisms that make these methods
effective at scale.  The foundational question is then when the adaptive
mechanisms are \emph{asymptotically invisible} to the averaged-iterate
covariance, so that the one-pass estimator retains the sandwich
limit~\eqref{eq:classical-pr} that underlies Wald-type inference.  When it
holds, the optimizer already run at scale doubles as an asymptotically
efficient estimator: confidence intervals follow from its averaged
iterates and a consistent covariance estimate, in a single streaming
pass, with no change to the update.

Whether this sandwich-covariance CLT~\eqref{eq:classical-pr} survives in
the richer setting---a data-driven preconditioner that need not converge,
coupled with a momentum buffer---has not been established.  Prior
asymptotic-normality results treat the two complications only in
isolation: momentum without preconditioning \citep{liu2023acceleration},
or an adaptive preconditioner required to converge to a fixed limit
\citep{leluc2023asymptotic, boyer2023stochastic}, which the constant
exponential-moving-average preconditioners deployed in practice do not.
\citet{anhuo2026} recently removed the convergence requirement on $P_t$
through a rate-only stabilization condition---the one-step variation of
the effective inverse drift $M_t = (P_tH)^{-1}$ decaying faster than
$t^{-(\alpha+1)/2}$---but only for preconditioned SGD \emph{without}
momentum; whether that framework extends to a non-convergent
preconditioner coupled with a momentum buffer was open
(Section~\ref{sec:related}).  This paper answers it affirmatively for
SA-Adam: an augmented-state analysis recovers the \emph{same} sandwich
limit $H^{-1}SH^{-1}$, so an adaptive preconditioner not assumed to
converge and time-varying momentum are asymptotically invisible to the averaged
iterate---though, as we show next, momentum makes the problem structurally
harder than the momentum-free case.

\paragraph{The structural role of momentum.}  Relative to the preconditioner-only
theory of \citet{anhuo2026}, the momentum buffer is a \emph{structural}
complication that forces a different route.  First,
it leaves the iterate-error recursion no longer closed on its own, collapsing
the preconditioner-only decomposition to a tautology
(Proposition~\ref{prop:tautology}) and forcing the augmented state
$(\Delta_t, m_{t-1})$, $\Delta_t := x_t - x^*$.  Second, the augmented
drift $L_t$ is non-normal, with a \emph{vanishing} two-time-scale gap
$\tau_t = \rho_t/\eta_t \to 0$ (momentum gain $\rho_t$ over step size),
unlike the uniformly stable preconditioner-only drift, and controlling it
requires the exact symmetrizer of Appendix~\ref{app:lyapunov}.  Third, the
momentum-decay exponent ($\rho_t \propto t^{-\gamma}$) must be sub-linear,
$\gamma < 1$: a $\gamma = 1$ buffer breaks the sandwich limit (Remark~\ref{rem:Rnz-gamma-necessary}).  The only essential ingredient imported
from \citet{anhuo2026} is its rate-only stabilization condition, entering
only as Assumption~\ref{ass:stab} and re-verified for SA-Adam
(Proposition~\ref{prop:sa-adam-stab}); every other result is proved here.

\subsection{Contributions}
\label{sec:contributions}

The paper makes five contributions; a detailed comparison with prior
work is given in Section~\ref{sec:related}.

\begin{enumerate}
\renewcommand{\labelenumi}{(\roman{enumi})}
\renewcommand{\theenumi}{(\roman{enumi})}
\item \label{contrib:augmented}
\emph{An augmented-state framework with a positive-stable spectrum}
(Section~\ref{sec:augmented}).  Treating the joint state
$z_t = (\Delta_t, m_{t-1})$ as a $2d$-dimensional linear stochastic
approximation, we derive its joint drift matrix $L_t$ and prove that
under uniform two-sided ellipticity of $P_tH$ (bounded
$\|M_t\|_{\mathrm{op}}$ and $\|P_tH\|_{\mathrm{op}}$), $L_t$ is positive
stable.  Each eigenvalue of $L_t$ has positive real part of order
$\tau_t := \rho_t/\eta_t \sim (c_1/\eta_0)\,t^{\alpha-\gamma}$, where
$\rho_t = c_1/t^\gamma$ is the momentum-decay rate with
$\gamma\in(\alpha,1)$; the augmented chain contracts at the one-step rate
$\rho_t$ and exhibits two-time-scale structure (slow buffer, fast
iterate) in this regime.

\item \label{contrib:obstruction}
\emph{Why the augmented state is necessary: a methodological obstruction}
(Section~\ref{sec:tautology}, Proposition~\ref{prop:tautology}).  We justify
the augmented route just introduced by showing the natural alternative fails:
any decomposition of the averaged Adam-iterate error $\overline\Delta_n$ that
fixes the canonical Polyak--Ruppert martingale and Taylor leading terms
$\Xi_n, T_n$ and reaches them by the Abel-summation/buffer-telescoping route
collapses to the tautological rearrangement $\overline\Delta_n = \Xi_n + T_n +
H^{-1}\overline g_n$ of the averaged gradient equation.  In particular,
the rate-only decomposition technique of \citet{anhuo2026} does
\emph{not} extend to momentum-based optimizers: there is no
preconditioner-isolating remainder term whose size is controlled by a
rate condition on a matrix increment.  The obstruction is
structural---it stems from the dependence of the error recursion on past
gradients through the buffer $m_t$, which makes the recursion
non-Markovian in the iterate error $\Delta_t$ alone (the augmented state
$(\Delta_t, m_{t-1})$ closes it).  (This is a negative
result about one proof technique, not a new limit theorem; it serves only to
justify the augmented-state route.)

\item \label{contrib:clt}
\emph{A non-autonomous Polyak--Ruppert CLT for the joint chain}
(Section~\ref{sec:clt}, Theorem~\ref{thm:augmented-clt}).  Building on
the linear-SA Polyak--Ruppert framework of \citet{mou2020linear}, we
allow the drift matrix $L_t$ to vary in time---its variation controlled
by the stabilization condition of \citet{anhuo2026} on $M_t = (P_tH)^{-1}$
together with the $O(t^{-1})$ one-step preconditioner increment of the standing
conditions, which the remainder bound consumes---and establish joint
asymptotic normality $\sqrt n\,\overline z_n \xrightarrow{d}
\mathcal{N}(0, \Sigma_z)$.  Here $\Sigma_z(t) := L_t^{-1}\Sigma_w(t)
L_t^{-\top}$ is the frozen-chain Polyak--Ruppert covariance at time $t$
(the congruence solution of $L_t\,\Sigma_z(t)\,L_t^\top = \Sigma_w(t)$,
not the stationary Lyapunov covariance); Theorem~\ref{thm:projection}
and an exact buffer-elimination identity show it takes the same value
$\Sigma_z$ for every $t$---so no single limiting $L$ is required.

\item \label{contrib:projection}
\emph{An iterate-marginal projection identity and SA-Adam
Polyak--Ruppert efficiency} (Section~\ref{sec:projection},
Theorem~\ref{thm:projection}; Section~\ref{sec:sa-adam},
Theorem~\ref{thm:main}).  We establish
$\Sigma_z^{(1,1)} = H^{-1}\,S\,H^{-1}$ regardless of $\rho, \tau, P$:
the iterate-marginal asymptotic covariance is the canonical
Polyak--Ruppert sandwich, \emph{independent} of momentum-decay rate
and preconditioner.  This extends the ``asymptotic equivalence
to averaged SGD'' observation of \citet{liu2023acceleration}---proved
for plain heavy-ball with constant momentum---to Adam-type methods with
adaptive preconditioning and time-varying momentum.  For an
SA-reparametrized Adam algorithm (SA-Adam) with schedule
$\beta_{1,t} = 1 - c_1/t^\gamma$ ($\gamma\in(\alpha,1)$) and
$\beta_{2,t} = 1 - c_2/t$, the averaged
iterate attains Polyak--Ruppert efficiency (Theorem~\ref{thm:main}),
thereby yielding the asymptotic-normality guarantee that underlies
Wald-type online inference.  The
sub-linear momentum-decay exponent $\gamma$ is essential: at the
canonical Polyak--Ruppert value $\gamma=1$, the endpoint buffer $m_n$
can contribute a non-vanishing $\Theta(n^{-1/2})$ term to the
iterate-marginal average, so the sandwich limit can fail: already in a scalar model the
iterate-marginal limit exceeds $H^{-1}SH^{-1}$
(Lemma~\ref{lem:Rnz-bound} and Remark~\ref{rem:Rnz-gamma-necessary} make
this precise).

\item \label{contrib:adamw}
\emph{Regularized one-pass inference via coupled ($L_2$) weight decay}
(Section~\ref{sec:variants}, Corollary~\ref{cor:adamw-coupled}).
Coupling weight decay into the gradient places the coupled-decay
variant within the paired-drift framework: the averaged iterate is $\sqrt n$-asymptotically
normal at the ridge-penalized minimizer $x^*_\lambda$ with the regularized
sandwich covariance $H_\lambda^{-1}S_\lambda H_\lambda^{-1}$
(Corollary~\ref{cor:adamw-coupled}), the sandwich form for penalized
$M$-estimation, so one-pass Wald-type inference for ridge-regularized
problems is immediate.  This also delineates the boundary of the
framework: the \emph{decoupled} weight decay of genuine AdamW
\citep{loshchilov2017decoupled} breaks the $P_tH$ structure behind the
projection identity, gives a preconditioner-dependent limit, and is left
as an open problem (Remark~\ref{rem:adamw}).
\end{enumerate}

The constant-EMA deployed form of Adam, with
$(\beta_1, \beta_2) = (0.9, 0.999)$, by contrast, lies outside our
framework on both counts: its constant second-moment EMA keeps the
one-step fluctuations of $P_t$ from decaying, so it is not expected to
satisfy the stabilization-rate condition of \citet{anhuo2026} (who treat
constant-EMA preconditioners heuristically, as a candidate
threshold-violation case, rather than proving non-stabilization), while
its constant momentum weight fails the joint spectral-rate condition of
Section~\ref{sec:clt}; whether it nonetheless attains the Polyak--Ruppert
sandwich limit through some compensating effect of the momentum buffer is
left open by our analysis.

\subsection{Related Work}
\label{sec:related}

\emph{Averaged SGD with momentum.}  \citet{liu2023acceleration} establish
a Polyak--Ruppert CLT for averaged SGD with momentum, with the same
$H^{-1}SH^{-1}$ limit as plain averaged SGD, but for a fixed momentum
weight and no preconditioning; \citet{wei2025weighted} treat weighted
averaging, and \citet{sebbouh2020convergence} give almost-sure rates for
stochastic heavy ball without an averaged-iterate CLT.  None handles
time-varying momentum and adaptive preconditioning together with the CLT
goal.

\emph{Augmented-state and two-time-scale analyses.}  Treating a momentum
method through the augmented (iterate, buffer) state has precedent in the
stochastic heavy-ball analysis of \citet{gadat2018stochastic} and in the
$2d$-dimensional contraction governing \citet{liu2023acceleration}; our
joint chain is a non-autonomous linear two-time-scale stochastic
approximation in the sense of \citet{konda2004actor,
mokkadem2006convergence, kaledin2020finite}.
\citet{barakat2021convergence} analyze Adam via an augmented-state
Lyapunov/ordinary-differential-equation (ODE) approach but target stationarity rather than the
averaged-iterate covariance.  In concurrent work,
\citet{wang2026inference} study an online Newton method whose Newton
system is solved by Nesterov-accelerated sketching, reducing the
resulting $2d$-dimensional recursion to $2\times2$ blocks
through a Cayley--Hamilton device akin to ours for
the momentum drift (Lemma~\ref{lem:Qsym-id}); their preconditioner is a
\emph{convergent} Hessian average and their CLT is for the
\emph{last} iterate, whereas ours is a $P_t$ not assumed to converge,
with a Polyak--Ruppert averaged-iterate CLT.  The new ingredients here are
this preconditioner and the closed-form symmetrizer of
Appendix~\ref{app:lyapunov}.

\emph{Adaptive preconditioning, convergent and not.}  Asymptotic
normality has been established for conditioned SGD with conditioning
matrices $C_t \to C$ \citep{leluc2023asymptotic} and for stochastic
Newton and weighted-averaged variants \citep{boyer2023stochastic}, in
each case under a \emph{convergent} conditioning matrix---excluding the
constant-EMA preconditioners deployed in practice; related
rate analyses include \citet{duchi2011adaptive,
defossez2022simple, godichon2024adagrad}.  \citet{anhuo2026} removed the
convergence requirement through a rate-only stabilization condition, but
only for preconditioned SGD \emph{without} momentum, and via an exact
pathwise decomposition that does not carry over to momentum
(Proposition~\ref{prop:tautology}).  The present paper is the
momentum-augmented counterpart: it replaces that decomposition with the
augmented-state framework, establishes the joint CLT and a projection
identity now joint in $P$ and the momentum schedule, and additionally
covers coupled weight decay; only the stabilization condition
(Assumption~\ref{ass:stab}) is imported.

\emph{Adam-family algorithms and weight decay.}
\citet{adaptiveprecond2025} give finite-sample convergence for
Nesterov-accelerated adaptive methods; \citet{reddi2018convergence} and
\citet{loshchilov2017decoupled} introduce AMSGrad and AdamW, respectively; we bring
AMSGrad within the Polyak--Ruppert framework as a certified variant, and
the \emph{coupled} ($L_2$) form of weight decay via
Corollary~\ref{cor:adamw-coupled} (genuine decoupled AdamW is left open,
Remark~\ref{rem:adamw}).
Weight decay is the optimization counterpart of ridge penalization; no
asymptotic-normality theory for the averaged iterate of a
\emph{regularized} adaptive-momentum method is known, and the
coupled form supplies one, with the penalized sandwich
$H_\lambda^{-1}S_\lambda H_\lambda^{-1}$.  The \emph{decoupled} form of
genuine AdamW bypasses the preconditioner, gives a
preconditioner-dependent limit outside the framework, and remains open
(Remark~\ref{rem:adamw}).

\emph{Why deployed Adam remains open.}  Taken together, the results above
leave one case untouched: the deployed, constant-EMA form of
Adam---fixed $(\beta_1,\beta_2)$, hence a non-vanishing gain---has no
known Polyak--Ruppert central limit theorem for its averaged iterate.
Each existing normality result forgoes a defining feature of Adam or
removes the obstacle: momentum without preconditioning
\citep{liu2023acceleration}, a \emph{convergent}
preconditioner $P_t \to P$ \citep{leluc2023asymptotic,
boyer2023stochastic}, or a decaying-gain reparametrization
\citep{anhuo2026}---SA-Adam being its momentum-augmented instance.
The remaining Adam literature is algorithmic or convergence-oriented
\citep{kingma2015adam, reddi2018convergence, defossez2022simple,
adaptiveprecond2025} and says nothing about the limiting law of the
average.  The obstruction is structural: a constant exponential moving
average keeps the one-step fluctuations of $P_t$ from vanishing, so the
preconditioner is not expected to stabilize and the chain leaves the Robbins--Monro
regime on which the Polyak--Ruppert machinery---ours included---rests.
A bona fide constant-EMA central limit theorem would instead demand an
ergodic or mixing-rate analysis of the joint chain---of the kind
developed for constant-step-size SGD viewed as a Markov chain
\citep{dieuleveut2020bridging}---which has not been carried out in the
adaptive-momentum setting.  The one apparent exception, \citet{barakat2021convergence},
is a conditional \emph{fluctuation} CLT around stationary points in the
nonconvex regime---a different object from the averaged-iterate sandwich
limit at issue.  Deployed Adam thus sits just outside the present
framework, and its averaged-iterate central limit theorem remains open.

\paragraph{Outline.}
Section~\ref{sec:setup} sets up the model and assumptions.
Section~\ref{sec:augmented} first records a methodological obstruction
(Section~\ref{sec:tautology}) that motivates augmenting the state, then
develops the augmented-state framework: the joint linear recursion, the
spectral structure of $L_t$, and the iterate mean-square bound.
Section~\ref{sec:asymptotics} proves the non-autonomous Polyak--Ruppert
CLT for the joint chain---outlined there, with the full proof in
Appendix~\ref{app:clt-proof}---and establishes the iterate-marginal
projection identity.
Section~\ref{sec:sa-adam} verifies the stabilization condition for
SA-Adam, proves the main theorem, and treats the bias-correction
reduction, downstream online inference, and side variants (SA-AMSGrad,
coupled $L_2$ weight decay, and SA-Adam-full).
Section~\ref{sec:simulation} describes the simulation study.
Section~\ref{sec:discussion} concludes.

\section{Setup and Assumptions}
\label{sec:setup}

We adopt the probabilistic setup of \citet{anhuo2026} and extend it to
the momentum-augmented setting.  For parameter $x \in \mathbb{R}^d$ and
data observation $\zeta$, let
$f(x,\zeta)$ denote the sample loss and
$F(x) := \mathbb{E}_\zeta[f(x,\zeta)]$ the population risk; we assume
$F$ has a unique minimizer $x^* := \arg\min F$, and write
$H := \nabla^2 F(x^*)$ for the population Hessian and
$S := \mathrm{Cov}(\nabla f(x^*,\zeta))$ for the gradient covariance at
$x^*$ (both well-defined under
Assumptions~\ref{ass:var}--\ref{ass:quadratic} below).  We write
$A \preceq B$ when $B-A$ is positive semidefinite, and
$\|A\|_{\mathrm{op}}$ for the operator norm of a matrix~$A$.

Let $\{\zeta_t\}_{t \ge 1}$ be a sequence of i.i.d.\ random variables.
$\{x_t\}_{t\ge1}$ denotes the iterate sequence generated by the
SA-Adam update introduced below, with initialization $x_1$ assumed
deterministic.  Define the natural filtration
$\mathcal{F}_t := \sigma(x_1,\zeta_1,\ldots,\zeta_t)$, $t \ge 0$.
Let $\xi_t(x) := \nabla f(x,\zeta_t) - \nabla F(x)$ denote the stochastic
gradient noise as a function of the parameter, and write
$\xi_t := \xi_t(x_t) = \nabla f(x_t,\zeta_t) - \nabla F(x_t)$ for its
evaluation at the current iterate.  Because $\zeta_t$ is independent of
$\mathcal F_{t-1}$, $x_t$ is $\mathcal F_{t-1}$-measurable, and $\xi_t$ is a
martingale difference (Assumption~\ref{ass:mg}), the
conditional noise covariance is the deterministic map
$S(x) := \mathrm{Cov}_\zeta(\nabla f(x,\zeta))$ evaluated at the iterate,
$\mathbb E[\xi_t\xi_t^\top \mid \mathcal F_{t-1}] = S(x_t)$, $S = S(x^*)$,
a representation used in the continuity hypothesis of
Theorem~\ref{thm:augmented-clt}.

\subsection{Assumptions}

Assumptions~\ref{ass:mg}--\ref{ass:iterate} below are the model
assumptions of \citet{anhuo2026}, restated here for completeness.

\begin{assumption}[Martingale Difference]
\label{ass:mg}
$\mathbb{E}[\xi_t \mid \mathcal{F}_{t-1}] = 0$ for all $t \geq 1$.
\end{assumption}

\begin{assumption}[Uniform Conditional Covariance Bound]
\label{ass:var}
There exists a deterministic positive semidefinite matrix $\overline{S}$ such
that $\mathbb{E}[\xi_t \xi_t^\top \mid \mathcal{F}_{t-1}] \preceq \overline{S}$
for all $t \geq 1$, with $S \preceq \overline S$.
\end{assumption}

\begin{assumption}[Strong Convexity]
\label{ass:convex}
$F$ is $\mu$-strongly convex for some $\mu > 0$, hence $H \succeq \mu I$.
\end{assumption}

\begin{assumption}[Local Second-Order Expansion with Trajectory Confinement]
\label{ass:quadratic}
There exist a neighborhood $\mathcal{N}$ of $x^*$ and a constant $L_R > 0$
such that $F$ is twice continuously differentiable on $\mathcal{N}$,
$\nabla^2 F(x^*) = H$, and $\nabla F(x) = H(x - x^*) + r(x)$,
$\|r(x)\| \leq L_R \|x - x^*\|^2$
for all $x \in \mathcal{N}$; moreover, $x_t \in \mathcal{N}$ almost surely
for all $t \ge 1$.  Define $u_t := r(x_t)$, so that
$\nabla f(x_t,\zeta_t) = H \Delta_t + u_t + \xi_t$ with
$\Delta_t := x_t - x^*$.
\end{assumption}

\begin{assumption}[Iterate Bound]
\label{ass:iterate}
There exist constants $C_\Delta > 0$ and $\alpha \in (1/2, 1)$ such that
$\mathbb{E}\|\Delta_t\|^2 \leq C_\Delta\, t^{-\alpha}$ for all
$t \geq 1$.
\end{assumption}

Assumption~\ref{ass:iterate} is the iterate mean-square bound assumed by \citet{anhuo2026}, the
standard mean-squared error (MSE) rate achieved by step size $\eta_t = \eta_0 t^{-\alpha}$;
as in that paper, it is stated here but derived below,
conditional on Assumption~\ref{ass:fourth}.  The general augmented mean-square bound is
Proposition~\ref{prop:augmented-mse} (Section~\ref{sec:augmented}); it is
specialized to SA-Adam in Section~\ref{sec:sa-adam}
(Theorem~\ref{thm:main}).

\begin{assumption}[Fourth-moment trajectory stability]
\label{ass:fourth}
There exists $C_4 < \infty$ such that $\mathbb E\|\Delta_t\|^4 \le
C_4\, t^{-2\alpha}$ for all $t \ge 1$.
\end{assumption}

Assumption~\ref{ass:fourth} is the fourth-moment analogue of
Assumption~\ref{ass:iterate}, used to control the nonlinear (Taylor)
terms in the mean-square analysis of
Proposition~\ref{prop:augmented-mse} and, through the resulting stability
bound, in the central-limit Taylor remainder (Lemma~\ref{lem:taylor}).  We take it as a hypothesis, so the
main results (Proposition~\ref{prop:augmented-mse},
Theorems~\ref{thm:augmented-clt} and~\ref{thm:main}) are \emph{conditional}
on it; two observations make it mild.  First, the noise's higher moments
need no extra assumption: under the main results' bounded-gradient hypothesis
$\|g_t\| \le G$ a.s.\ (and $\nabla F$ bounded on the
confinement neighborhood $\mathcal N$), the centered noise $\xi_t = g_t -
\nabla F(x_t)$ is almost surely bounded, so all of its conditional
moments---in particular the fourth---are finite and uniformly bounded.
Second, given this, the rate $\mathbb E\|\Delta_t\|^4 = O(t^{-2\alpha})$
could be obtained by the same mechanism as the second-moment bound, namely a
joint Lyapunov bootstrap on $(\mathbb E V_t, \mathbb E V_t^2)$: squaring
the exact one-step contraction~\eqref{eq:Qsym-onestep} gives an
$\mathbb E V_t^2$-recursion with effective contraction $2\rho_t$ whose
forcing---built from the bounded noise and the Taylor increments, with
the martingale cross-term vanishing in conditional expectation and the
remaining cross-term controlled by $\mathbb E V_t = O(\eta_t)$---closes
at order $O(\eta_t^2)$, in exact parallel to
Proposition~\ref{prop:augmented-mse}.  This is the fourth-moment
counterpart of the primitive verification of
Assumption~\ref{ass:iterate} (the iterate-bound assumption of \citet{anhuo2026}); we
record it as a hypothesis rather than reproduce the parallel
$2p$-th-moment bootstrap, which introduces no new ideas but would have to
be run jointly with Proposition~\ref{prop:augmented-mse}.

\subsection{The SA-Adam Recursion}
\label{sec:sa-adam-defn}

Let $\eta_t = \eta_0 t^{-\alpha}$ with $\eta_0 > 0$ and
$\alpha \in (1/2, 1)$, and let $\rho_t = c_1/t^\gamma$ with $c_1 > 0$ and
\emph{momentum exponent} $\gamma \in (\alpha, 1)$.  So that every
momentum weight $1 - \rho_t$ is a genuine convex weight, we take
$\rho_t \in (0,1)$ for all $t \ge 1$: for $c_1 \ge 1$ this is enforced by
the shifted schedule $\rho_t = c_1/(t + t_0)^\gamma$ with
$t_0 := \lceil c_1^{1/\gamma}\rceil$ (asymptotically identical,
$\rho_t \sim c_1 t^{-\gamma}$, and changing no statement below), while for
$c_1 < 1$ the plain schedule $\rho_t = c_1/t^\gamma$ already satisfies it.
Consequently every momentum weight $1 - \rho_s \in (0,1)$, so the
bias-correction products below are well defined.  The two-sided
constraint
$\gamma > \alpha$ (two-time-scale separation, $\tau_t = \rho_t/\eta_t \sim
(c_1/\eta_0)t^{\alpha-\gamma}\to 0$) and $\gamma < 1$ (vanishing buffer
residual at the $\sqrt n$-scale, Appendix~\ref{app:clt-proof}) is
essential: the canonical Polyak--Ruppert choice $\gamma = 1$ leaves a
persistent endpoint-buffer term in the iterate-marginal limit, so a
strictly sub-linear momentum decay is required.  No additional
lower-bound condition on $c_1$ beyond $c_1 > 0$ is needed: under
$\gamma < 1$, the contraction $\rho_t = c_1/t^\gamma$ asymptotically
dominates the $O(1/t)$ time-variation feedback in
Appendix~\ref{app:lyapunov}.  The SA-Adam update is
\begin{align}
\label{eq:sa-adam}
  g_t &:= \nabla f(x_t, \zeta_t), \\
  m_t &= (1 - \rho_t)\,m_{t-1} + \rho_t\,g_t, \\
  v_t &= (1 - \rho^v_t)\,v_{t-1} + \rho^v_t\,(g_t \odot g_t + \epsilon\mathbf{1}),
    \qquad \rho^v_t = c_2/t, \\
  P_t &= \mathrm{Diag}(\hat v_{t-1})^{-1/2}, \\
  x_{t+1} &= x_t - \eta_t\,P_t\,\hat m_t,
\end{align}
with $m_0 = 0$, $v_0 = \epsilon\mathbf 1$, $\epsilon > 0$,
$c_2 \in (0,1]$, and bias-corrected moments $\hat m_t = m_t/\kappa^m_t$,
$\hat v_t = v_t/\kappa^v_t$, the bias-correction factors being formed
from the realized gains,
\[
  \kappa^m_t = 1 - \prod_{s=1}^t (1 - \rho_s),
  \qquad
  \kappa^v_t = 1 - \prod_{s=1}^t (1 - \rho^v_s),
  \qquad \rho^v_s = c_2/s .
\]
By the schedule fixed above, $\rho_s \in (0,1)$ and $\rho^v_s \in (0,1]$
for every $s \ge 1$, so each factor $1-\rho_s \in (0,1)$,
$1-\rho^v_s \in [0,1)$, and hence $\kappa^m_t, \kappa^v_t \in (0,1]$ for
all $t \ge 1$ (with $\kappa^v_t \equiv 1$ when $c_2 = 1$).  We adopt the
convention $\hat v_0 := v_0 = \epsilon\mathbf 1$, so that $P_1 =
\mathrm{Diag}(v_0)^{-1/2} = \epsilon^{-1/2} I$ is well defined; the first
momentum factor $\kappa^m_1 = \rho_1 > 0$ (equal to $c_1$ for the plain
schedule), so $\hat m_1$ needs no convention.

The bias correction satisfies $1 - \kappa^m_t = O(e^{-a t^{1-\gamma}})$
for any $a < c_1/(1-\gamma)$ (Lemma~\ref{lem:bias-effective-step}), so
$\kappa^m_t \to 1$ super-polynomially fast (since $1-\gamma > 0$), even
better behaved than the $\gamma = 1$ case.  In the analysis below
we drop the bias correction for notational simplicity; the full
algorithm with bias correction is recovered by an effective-step-size
reparametrization (see Section~\ref{sec:bias-correction}).

Define the predictable preconditioner sequence $\{P_t\}$ via the recursion
above, and note that $P_t$ is $\mathcal F_{t-1}$-measurable.  The
\emph{effective inverse drift} of \citet{anhuo2026} is $M_t := (P_t H)^{-1}$.

\begin{assumption}[Preconditioner Stabilization Condition]
\label{ass:stab}
There exists $\beta > (\alpha+1)/2$ and $C_M > 0$ such that, almost
surely, $\|M_t - M_{t-1}\|_{\mathrm{op}} \le C_M\,t^{-\beta}$ for all
$t \ge 2$ and $\sup_t \|M_t\|_{\mathrm{op}} < \infty$.
\end{assumption}

This is precisely the stabilization condition of
\citet{anhuo2026} (with their rate-optimal threshold
$\beta > (\alpha+1)/2$).  For SA-Adam with bounded
gradients it holds with $\beta = 1$: our preconditioner
$P_t = \mathrm{Diag}(\hat v_{t-1})^{-1/2}$ uses the \emph{bias-corrected}
second moment, so it parallels---rather than directly inherits---the
SA-RMSProp stabilization verification of \citet{anhuo2026}, which treats
the uncorrected recursion; the
bias-correction step is verified in
Proposition~\ref{prop:sa-adam-stab} below.

Subtracting $x^*$ from both sides of the update and using
Assumption~\ref{ass:quadratic} (so that $g_t = H\Delta_t + u_t + \xi_t$
with $u_t = r(x_t)$), we obtain the centered error recursion.  As flagged
above, the analysis is carried out for the \emph{un-bias-corrected}
update, in which $\hat m_t$ is replaced by the raw buffer $m_t$; the
bias-correction factor $\kappa^m_t$ is restored as an effective step size
$\tilde\eta_t = \eta_t/\kappa^m_t$ in
Section~\ref{sec:bias-correction}:
\begin{equation}
\label{eq:error-recursion}
  \Delta_{t+1} = \Delta_t - \eta_t\,P_t\,m_t,\qquad
  m_t = (1 - \rho_t)\,m_{t-1} + \rho_t\,(H \Delta_t + u_t + \xi_t).
\end{equation}
The averaged iterate is $\overline x_n = n^{-1}\sum_{t=1}^n x_t$ and the
averaged error is $\overline\Delta_n = n^{-1}\sum_{t=1}^n \Delta_t$.

The assumptions enter the analysis modularly.
Assumptions~\ref{ass:mg}--\ref{ass:quadratic} (martingale-difference noise,
a conditional-covariance bound, strong convexity, and a local quadratic
expansion with trajectory confinement) define the local linear
stochastic-approximation model.
Assumption~\ref{ass:iterate} (iterate bound) is not an extra hypothesis but
the conclusion of Proposition~\ref{prop:augmented-mse}.
Assumption~\ref{ass:fourth} (fourth-moment trajectory stability) is used only
to control the nonlinear Taylor terms---in the mean-square bound of
Proposition~\ref{prop:augmented-mse} and, through the resulting stability
bound, in the central-limit Taylor remainder (Lemma~\ref{lem:taylor}).
Assumption~\ref{ass:stab} (preconditioner stabilization), with the standing
$O(t^{-1})$ one-step preconditioner increment, controls the time variation of
the joint drift $L_t$.
These feed the three steps of the argument:
Proposition~\ref{prop:augmented-mse} establishes the augmented mean-square
rates; Theorem~\ref{thm:augmented-clt} converts them into the joint
Polyak--Ruppert CLT, whose covariance the projection identity
(Theorem~\ref{thm:projection}) evaluates as $H^{-1}SH^{-1}$; and
Proposition~\ref{prop:sa-adam-stab} verifies Assumption~\ref{ass:stab} and the
standing conditions for SA-Adam, giving the main theorem
(Theorem~\ref{thm:main}).

\section{The Augmented-State Framework}
\label{sec:augmented}

We analyze SA-Adam by treating the momentum buffer as an additional
state variable.  This section first records a structural obstruction
showing why a direct transcription of the pathwise-decomposition
technique of \citet{anhuo2026}---one that retains the Polyak--Ruppert
martingale and Taylor terms as a fixed leading part and seeks to isolate
the momentum contribution in a controllable remainder---collapses to a
tautology, and then constructs the joint linear recursion, establishes
its spectrum, and proves the iterate mean-square bound.

\subsection{A Pathwise-Decomposition Obstruction}
\label{sec:tautology}

\citet{anhuo2026} prove the Polyak--Ruppert CLT for plain preconditioned
SGD via an exact pathwise decomposition of the
averaged error into a martingale term, a Taylor remainder, and a dynamic
remainder that isolates the preconditioner.  A natural strategy for the
SA-Adam setting is to seek an analogous decomposition---one that keeps
the martingale leading term in the noise $\xi_t$ and the Taylor term at
their canonical Polyak--Ruppert forms, leaving the momentum and
preconditioner contributions to a controllable remainder.  The following
proposition shows that any decomposition of this canonical form collapses
to a tautology: its dynamic remainder is forced to carry the full
averaged error with coefficient one (it equals
$\overline\Delta_n - \Xi_n - T_n$), so the decomposition cannot deliver
the CLT on its own.

\begin{proposition}[Tautological collapse of the canonical pathwise decomposition]
\label{prop:tautology}
Under the SA-Adam recursion~\eqref{eq:sa-adam} with $c_1 > 0$,
$\gamma \in (\alpha, 1)$, $\eta_t = \eta_0 t^{-\alpha}$,
$\alpha \in (1/2, 1)$, and the standing
Assumptions~\ref{ass:mg}--\ref{ass:iterate}, consider any exact
decomposition of the averaged error
\begin{equation}
\label{eq:purported-decomp}
  \overline\Delta_n = \Xi_n + T_n + R_n,\qquad
  \Xi_n := -\frac{1}{n}H^{-1}\sum_{t=1}^n \xi_t,\quad
  T_n := -\frac{1}{n}H^{-1}\sum_{t=1}^n u_t,
\end{equation}
whose leading martingale term $\Xi_n$ and Taylor term $T_n$ are fixed to
the canonical Polyak--Ruppert forms.  Then the dynamic remainder is
forced to equal, \emph{exactly},
\begin{equation}
\label{eq:tautological-Rn}
  R_n = H^{-1}\,\overline g_n,\qquad
  \overline g_n := \frac{1}{n}\sum_{t=1}^n g_t,
\end{equation}
the averaged gradient---with no surviving boundary term.  This identity
is purely algebraic: it uses only the local expansion of
Assumption~\ref{ass:quadratic} and the invertibility of $H$, not the
momentum schedule or the martingale structure.  Consequently $R_n$ is
not an independently controllable preconditioner remainder: it is
algebraically identical to $\overline\Delta_n - \Xi_n - T_n$, carrying
the full averaged error with coefficient one.  Using in addition the
iterate bound of Assumption~\ref{ass:iterate} to render the Taylor term
negligible ($\sqrt n\,T_n \to 0$), we find that $\sqrt n\,R_n \to 0$
holds if and only if
$H\overline\Delta_n = -\overline\xi_n + o_p(n^{-1/2})$---the asymptotic
linear representation underlying the Polyak--Ruppert CLT (from which the
CLT follows by the martingale CLT for $\sqrt n\,\overline\xi_n$ under the
usual conditions).
The buffer-telescoping/Abel-summation route that would generalize the
pathwise decomposition of \citet{anhuo2026} produces exactly this
remainder, and so cannot deliver the CLT without presupposing it.
\end{proposition}

\begin{proof}
By the local expansion of Assumption~\ref{ass:quadratic}, each gradient
obeys the exact identity $g_t = H\Delta_t + u_t + \xi_t$.  Averaging over
$t = 1, \ldots, n$ gives
$\overline g_n = H\overline\Delta_n + \overline u_n + \overline\xi_n$,
and applying $H^{-1}$ and rearranging,
\begin{equation}
\label{eq:taut-rearrange}
  \overline\Delta_n
    = H^{-1}\overline g_n - H^{-1}\overline u_n - H^{-1}\overline\xi_n
    = \Xi_n + T_n + H^{-1}\overline g_n .
\end{equation}
Comparing~\eqref{eq:taut-rearrange} with the posited
decomposition~\eqref{eq:purported-decomp}, whose leading terms
$\Xi_n, T_n$ are fixed, identifies the remainder \emph{exactly} as
$R_n = H^{-1}\overline g_n$, which is~\eqref{eq:tautological-Rn}.  This
is a pure algebraic consequence of the averaged gradient equation: no
buffer-endpoint or boundary term arises, and the identity is independent
of the momentum exponent $\gamma$ and of whatever Abel-summation or
buffer-telescoping manipulations one uses to reach it.  Any such
route---summing the un-bias-corrected increment
$\Delta_t - \Delta_{t+1} = \eta_t P_t m_t$ of~\eqref{eq:error-recursion}
and closing the $m_t$-dependence through the buffer telescoping $g_t =
\rho_t^{-1}(m_t - m_{t-1}) + m_{t-1}$---produces a decomposition of the
form~\eqref{eq:purported-decomp} and therefore returns this same $R_n$
(the full update merely replaces $\eta_t$ by the effective step
$\tilde\eta_t = \eta_t/\kappa^m_t$, which changes nothing in the algebra,
since the identity rests on the gradient expansion, not on the update).

The decomposition is therefore a \emph{tautology}.  Substituting
$\overline g_n = H\overline\Delta_n + \overline u_n + \overline\xi_n$
back into $R_n = H^{-1}\overline g_n$ returns
\begin{equation}
\label{eq:taut-Rn-expand}
  R_n = \overline\Delta_n + H^{-1}\overline u_n + H^{-1}\overline\xi_n
      = \overline\Delta_n - \Xi_n - T_n,
\end{equation}
the defining rearrangement of~\eqref{eq:purported-decomp}: the
``remainder'' carries the full averaged error $\overline\Delta_n$ with
coefficient one, not a decaying preconditioner increment.

It remains to record the CLT-scale equivalence.  The Taylor term is
negligible at the $\sqrt n$ scale: by Assumption~\ref{ass:quadratic}
$\|u_t\| \le L_R\|\Delta_t\|^2$, so the iterate bound of
Assumption~\ref{ass:iterate} gives
$\mathbb E\|u_t\| \le L_R C_\Delta t^{-\alpha}$ and hence
\[
  \sqrt n\,\mathbb E\|T_n\|
    \le \frac{\|H^{-1}\|_{\mathrm{op}}\,L_R C_\Delta}{\sqrt n}\sum_{t=1}^n t^{-\alpha}
    = O\!\big(n^{1/2-\alpha}\big) \xrightarrow[n\to\infty]{} 0
    \qquad (\alpha > 1/2),
\]
so $\sqrt n\,T_n \to 0$ in $L^1$, hence in probability.  By the
expansion~\eqref{eq:taut-Rn-expand}, $\sqrt n\,R_n =
\sqrt n\,(\overline\Delta_n + H^{-1}\overline\xi_n) - \sqrt n\,T_n$, so
$\sqrt n\,R_n \to 0$ in probability if and only if
$\sqrt n\,(\overline\Delta_n + H^{-1}\overline\xi_n) \to 0$, i.e.\ if and
only if $H\overline\Delta_n = -\overline\xi_n + o_p(n^{-1/2})$.  This is
exactly the asymptotic linear representation underlying the
Polyak--Ruppert CLT---from which the CLT follows by the martingale CLT
for $\sqrt n\,\overline\xi_n$ under the usual conditions: establishing the
remainder negligible is
logically equivalent to that representation, so the pathwise route cannot
deliver the CLT on its own.

In particular this canonical-leading-term framework admits no
increment-sum representation
$R_n = \frac{1}{n}\sum_t \eta_{t-1}^{-1}(M_t - M_{t-1})\Delta_t + (\mathrm{boundary})$ with
decaying increments $\|M_t - M_{t-1}\|_{\mathrm{op}} \lesssim t^{-\beta}$,
$\beta > (\alpha+1)/2$---the form that would force $\sqrt n\,R_n \to 0$ via the
stabilization bounds of \citet{anhuo2026}---without circularity, since by the
equivalence just shown any such representation already encodes the CLT-scale
cancellation (that the remainder \emph{is} negligible for $\gamma < 1$ is
established later by the augmented-state analysis,
Appendix~\ref{app:clt-proof}).  At the boundary $\gamma = 1$ the obstruction is
starker still: $\sqrt n\,\overline g_n$ can be asymptotically
non-degenerate---already in a scalar model
(Remark~\ref{rem:Rnz-gamma-necessary})---so $\sqrt n\,R_n$ need not vanish and
no canonical-leading-term decomposition~\eqref{eq:purported-decomp} can hold
with a negligible $o_p(n^{-1/2})$ remainder.
\end{proof}

\emph{Why the obstruction is structural.}
The collapse reflects a structural feature of momentum, not a particular
manipulation.  Because the error recursion is non-Markovian in
$\Delta_t$ alone, any exact identity that closes it through the buffer's
telescoping carries the entire history as the single object
$H^{-1}\overline g_n$, which by the gradient identity equals
$\overline\Delta_n$ plus the very terms one seeks to isolate; the
putative remainder is then the averaged error itself---not an
increment-sum remainder with decaying increments $M_t - M_{t-1}$---and
its negligibility is equivalent to the leading-order cancellation
$\sqrt n\,(H\overline\Delta_n + \overline\xi_n)\to 0$ underlying the
Polyak--Ruppert CLT.  This---independent of $\gamma$---motivates the
augmented-state route below, which treats the buffer $m_t$ as an
additional state variable.

\subsection{Joint Linear Recursion and Its Spectrum}
\label{sec:joint-spectrum}

We work throughout with the un-bias-corrected centered
recursion~\eqref{eq:error-recursion}; as noted in
Section~\ref{sec:sa-adam-defn}, the first-moment bias correction is
reinstated as an effective step size $\tilde\eta_t = \eta_t/\kappa^m_t$
(Section~\ref{sec:bias-correction}) and leaves the linear-algebraic
structure below unchanged.

Define the joint state
\[
  z_t := \begin{pmatrix} \Delta_t \\ m_{t-1} \end{pmatrix} \in \mathbb R^{2d},
\]
where by convention $m_0 := 0$.  Both components are
$\mathcal F_{t-1}$-measurable.

Substituting the recursion~\eqref{eq:error-recursion}:
\begin{align*}
  \Delta_{t+1} &= (I - \eta_t \rho_t P_t H)\,\Delta_t
    - \eta_t (1-\rho_t) P_t m_{t-1} - \eta_t \rho_t P_t(u_t + \xi_t),\\
  m_t &= \rho_t H \Delta_t + (1-\rho_t) m_{t-1} + \rho_t (u_t + \xi_t).
\end{align*}
In matrix form,
\begin{equation}
\label{eq:joint-recursion}
  z_{t+1} = \Phi_t\, z_t + B_t\,(u_t + \xi_t),
\end{equation}
with
\[
  \Phi_t = \begin{pmatrix}
    I - \eta_t \rho_t P_t H & -\eta_t(1-\rho_t) P_t \\
    \rho_t H & (1-\rho_t) I
  \end{pmatrix},
  \qquad
  B_t = \begin{pmatrix} -\eta_t \rho_t P_t \\ \rho_t I \end{pmatrix}.
\]

The drift matrix in the standard linear-SA form
$z_{t+1} = (I - \eta_t L_t) z_t + B_t(u_t + \xi_t)$ is
$L_t := (I - \Phi_t)/\eta_t$:
\begin{equation}
\label{eq:Lt}
  L_t = \begin{pmatrix}
    \rho_t P_t H & (1-\rho_t) P_t \\
    -\tau_t H & \tau_t I
  \end{pmatrix},
\end{equation}
where we introduce the abbreviation
$\tau_t := \rho_t / \eta_t \sim (c_1/\eta_0)\, t^{\alpha-\gamma}$
(with equality for the unshifted schedule), which is $\to 0$ under
$\gamma > \alpha$.

Having set up the joint drift, we turn to its spectrum.  The central
spectral fact about $L_t$ is the following.

\begin{lemma}[Spectrum of $L_t$]
\label{lem:spectrum}
For each $t \ge 1$, the eigenvalues of $L_t$ (counted with multiplicity)
are the $2d$ roots of the $d$ scalar quadratics
\begin{equation}
\label{eq:char-poly}
  \lambda^2 - (\tau_t + \rho_t \mu_i^{(t)})\,\lambda + \tau_t \mu_i^{(t)} = 0,
  \qquad i = 1, \ldots, d,
\end{equation}
where $\mu_1^{(t)}, \ldots, \mu_d^{(t)}$ are the eigenvalues of $P_t H$
(real and positive, since $P_t H$ is similar to the symmetric positive
definite $P_t^{1/2} H P_t^{1/2}$).
\end{lemma}

\begin{proof}
The reduction of the $2d \times 2d$ spectrum to $d$ scalar quadratics in
the eigenvalues of $P_t H$ is the standard reduction for momentum-type
and two-time-scale block drifts, obtained by diagonalizing the single
matrix $P_t H$ (similar to the symmetric positive-definite
$P_t^{1/2} H P_t^{1/2}$)---not by simultaneously diagonalizing $P_t$ and
$H$, which do not commute in general (cf.\ the heavy-ball eigenvalue
analysis of \citet{polyak1964methods, lessard2016analysis} and the
two-time-scale stochastic-approximation analyses of
\citet{konda2004actor, mokkadem2006convergence}); we record the
time-varying, preconditioned form needed here.

We compute the characteristic polynomial
$\chi_t(\lambda) := \det(\lambda I_{2d} - L_t)$.  Writing $L_t$ in
$d \times d$ blocks as in~\eqref{eq:Lt},
\[
  \lambda I_{2d} - L_t =
  \begin{pmatrix}
    \lambda I - \rho_t P_t H & -(1-\rho_t) P_t \\
    \tau_t H & (\lambda - \tau_t) I
  \end{pmatrix}.
\]
The lower-right block $(\lambda - \tau_t) I$ is invertible for
$\lambda \ne \tau_t$, so the Schur-complement determinant formula gives
\[
  \chi_t(\lambda)
  = \det\!\big((\lambda - \tau_t) I\big)\,
    \det\!\Big(\lambda I - \rho_t P_t H
      + \tfrac{(1-\rho_t)\tau_t}{\lambda - \tau_t}\,P_t H\Big),
\]
where the Schur correction
$-\big[-(1-\rho_t)P_t\big]\big[(\lambda-\tau_t)I\big]^{-1}\big[\tau_t H\big]
= \tfrac{(1-\rho_t)\tau_t}{\lambda-\tau_t}\,P_t H$ uses that the
off-diagonal blocks are scalar multiples of $P_t$ and $H$.  Collecting the
$P_t H$ terms through
$\rho_t - \tfrac{(1-\rho_t)\tau_t}{\lambda-\tau_t}
 = \tfrac{\rho_t\lambda - \tau_t}{\lambda-\tau_t}$ and absorbing one factor
of $(\lambda - \tau_t)$ into the determinant along each of the $d$
eigendirections of $P_t H$ yields the polynomial identity
\begin{equation}
\label{eq:char-det-identity}
  \chi_t(\lambda)
  = \det\!\big(\lambda(\lambda - \tau_t) I - (\rho_t \lambda - \tau_t)\,P_t H\big),
\end{equation}
derived for $\lambda \ne \tau_t$ and hence, both sides being polynomials
in $\lambda$ of degree $2d$ that agree at all but finitely many points,
for every $\lambda \in \mathbb C$.  No $\lambda$-dependent denominator
survives in~\eqref{eq:char-det-identity}, so it is valid at
$\lambda = \tau_t$ as well.

Because $P_t H$ is similar to the symmetric positive-definite matrix
$P_t^{1/2} H P_t^{1/2}$, it is diagonalizable with real positive
eigenvalues $\mu_1^{(t)}, \ldots, \mu_d^{(t)}$.  For any scalars $a, b$,
$\det(a I - b\,P_t H) = \prod_{i=1}^d (a - b\,\mu_i^{(t)})$; applying this
to~\eqref{eq:char-det-identity} with $a = \lambda(\lambda - \tau_t)$ and
$b = \rho_t\lambda - \tau_t$ gives
\[
  \chi_t(\lambda)
  = \prod_{i=1}^d \Big[\lambda(\lambda - \tau_t)
      - (\rho_t\lambda - \tau_t)\,\mu_i^{(t)}\Big]
  = \prod_{i=1}^d \Big[\lambda^2
      - (\tau_t + \rho_t\mu_i^{(t)})\,\lambda + \tau_t\mu_i^{(t)}\Big].
\]
Thus the $2d$ eigenvalues of $L_t$, counted with multiplicity, are exactly
the roots of the $d$ quadratics~\eqref{eq:char-poly}.  The identity is
exact for every $\lambda$ and requires no genericity hypothesis; in
particular it captures the boundary case $\rho_t = 1$, where each
quadratic factors as $(\lambda - \tau_t)(\lambda - \mu_i^{(t)})$, so that
$\lambda = \tau_t$ is then an eigenvalue of multiplicity $d$---consistent
with $L_t$ becoming block lower-triangular at $\rho_t = 1$.
\end{proof}

From this scalar reduction we read off the stability of the joint drift:
the next lemma shows every eigenvalue has positive real part of order
$\tau_t$, so that $L_t$ is positive-stable---the property the Lyapunov
mean-square bound of Section~\ref{sec:iterate-mse} requires.

\begin{lemma}[Positive stability of $L_t$ with rate $\tau_t$]
\label{lem:hurwitz}
Assume the uniform spectral bound
$0 < \mu_- \le \mu_i^{(t)} \le \mu_+ < \infty$ on the eigenvalues of
$P_t H$, uniformly in $i$ and $t$.  (The lower bound is the boundedness
$\sup_t \|M_t\|_{\mathrm{op}} < \infty$ of Assumption~\ref{ass:stab},
with $M_t = (P_t H)^{-1}$; the upper bound is the uniform ellipticity of
$P_t$.  For SA-Adam both are verified in
Proposition~\ref{prop:sa-adam-stab}.)  Then for all $t$ sufficiently
large, the discriminant of~\eqref{eq:char-poly} is negative, and every
eigenvalue $\lambda$ of $L_t$ belongs to a complex-conjugate pair
satisfying
\[
  \mathrm{Re}(\lambda) = \frac{\tau_t}{2}\,(1 + o(1)),\qquad
  |\mathrm{Im}(\lambda)| \asymp \sqrt{\tau_t},
\]
uniformly in $i$.  In particular $L_t$ is \emph{positive stable}---every
eigenvalue lies in the open right half-plane, equivalently $-L_t$ is
Hurwitz---with a spectral gap $\mathrm{Re}(\lambda) \asymp \tau_t$ that
\emph{vanishes} as $t \to \infty$ (it is not uniform in $t$).  Moreover
the corresponding discrete eigenvalue $1 - \eta_t\lambda$ of $(I -
\eta_t L_t)$ has modulus
\begin{equation}
\label{eq:onestep-modulus}
  |1 - \eta_t\lambda| = \sqrt{1 - \rho_t}
  = 1 - \tfrac{\rho_t}{2} + O(\rho_t^2)
\end{equation}
\emph{exactly}, for every eigenvalue (both members of each conjugate
pair) and independently of $i$.
\end{lemma}

\begin{proof}
The qualitative picture---a complex-conjugate spectrum whose real part
is of the slow order $\tau_t$ and whose imaginary part scales as
$\sqrt{\tau_t}$---is the classical heavy-ball phenomenon
\citep{polyak1964methods, lessard2016analysis}; we derive its precise
two-time-scale, preconditioned form from the
quadratics~\eqref{eq:char-poly}.  Fix $t$ and write
$\mu := \mu_i^{(t)} \in [\mu_-, \mu_+]$.  The two roots of
$\lambda^2 - (\tau_t + \rho_t\mu)\lambda + \tau_t\mu = 0$ have discriminant
\begin{equation}
\label{eq:hurwitz-disc}
  \Delta = (\tau_t + \rho_t\mu)^2 - 4\tau_t\mu
  = -4\tau_t\mu\Bigl(1 - \tfrac{\tau_t}{4\mu} - \tfrac{\rho_t}{2}
      - \tfrac{\rho_t^2\mu}{4\tau_t}\Bigr).
\end{equation}
Each bracketed correction is $o(1)$ uniformly in $i$: $\tau_t \to 0$,
$\rho_t \to 0$, and $\rho_t^2/\tau_t = \rho_t\eta_t \to 0$ (using
$\rho_t/\tau_t = \eta_t$), while $\mu \in [\mu_-, \mu_+]$.  Hence for all
$t$ large enough $\Delta = -4\tau_t\mu\,(1+o(1)) < 0$, and the two roots
form a complex-conjugate pair $\lambda = \tfrac12(\tau_t + \rho_t\mu)
\pm \tfrac{i}{2}\sqrt{-\Delta}$.

\emph{Real part.}  $\mathrm{Re}(\lambda) = \tfrac12(\tau_t + \rho_t\mu)$;
since $\rho_t\mu \le \mu_+\rho_t$ and $\rho_t/\tau_t = \eta_t \to 0$,
$\mathrm{Re}(\lambda)
  = \tfrac{\tau_t}{2}\Bigl(1 + \tfrac{\rho_t\mu}{\tau_t}\Bigr)
  = \tfrac{\tau_t}{2}\bigl(1 + O(\eta_t)\bigr)$
uniformly in $i$, so $\mathrm{Re}(\lambda) \ge \tfrac{\tau_t}{2}(1-o(1))
> 0$ for all large $t$.  Thus every one of the $2d$ eigenvalues of $L_t$
has strictly positive real part, bounded below by
$\tfrac{\tau_t}{2}(1-o(1))$, so $L_t$ is positive stable (equivalently
$-L_t$ is Hurwitz); the spectral gap is of order $\tau_t$ and vanishes
as $t\to\infty$.

\emph{Imaginary part.}  From~\eqref{eq:hurwitz-disc},
$|\mathrm{Im}(\lambda)|^2 = -\Delta/4 = \tau_t\mu\,(1+o(1))$, so
$|\mathrm{Im}(\lambda)| = \sqrt{\tau_t\mu}\,(1+o(1))$ and, with $\mu \in
[\mu_-, \mu_+]$, $|\mathrm{Im}(\lambda)| \asymp \sqrt{\tau_t}$ uniformly
in $i$.

\emph{Discrete eigenvalue modulus.}  The sum and product of the two roots are,
\emph{exactly}, $2\mathrm{Re}(\lambda) = \tau_t + \rho_t\mu$ and
$|\lambda|^2 = \lambda\bar\lambda = \tau_t\mu$.  Hence, using
$\eta_t\tau_t = \rho_t$ and $\eta_t^2\tau_t\mu = \eta_t\rho_t\mu$,
\[
  |1 - \eta_t\lambda|^2
  = 1 - 2\eta_t\mathrm{Re}(\lambda) + \eta_t^2|\lambda|^2
  = 1 - \eta_t(\tau_t + \rho_t\mu) + \eta_t^2\tau_t\mu
  = 1 - \rho_t - \eta_t\rho_t\mu + \eta_t\rho_t\mu
  = 1 - \rho_t,
\]
the two $\mu$-dependent terms cancelling identically.  Therefore
$|1 - \eta_t\lambda| = \sqrt{1 - \rho_t} = 1 - \tfrac{\rho_t}{2} +
O(\rho_t^2)$ for every eigenvalue, independently of $i$---this
is~\eqref{eq:onestep-modulus}.
\end{proof}

\emph{Two-time-scale structure.}
For $\alpha \in (1/2, 1)$ and $\gamma \in (\alpha, 1)$ the three rates
are strictly ordered, $\tau_t \asymp t^{-(\gamma-\alpha)} \gg \eta_t
\asymp t^{-\alpha} \gg \rho_t \asymp t^{-\gamma}$ (from $0 <
\gamma-\alpha < \alpha < \gamma$, the middle inequality since
$\gamma < 1 < 2\alpha$).  The buffer self-drift block of $L_t$ has
spectral norm $\asymp \tau_t$ and the iterate self-drift block
$\asymp \rho_t$, the signature of a
two-time-scale system \citep{borkar2008stochastic,
mokkadem2006convergence}: the iterate (step $\eta_t$) is \emph{fast} and
the buffer (step $\rho_t$, smaller since $\gamma > \alpha$) \emph{slow},
with separation factor $\rho_t/\eta_t = \tau_t \to 0$.  The joint
contraction rate $\mathrm{Re}(\lambda) \asymp \tau_t$, however, governs
the asymptotic analysis.

\subsection{Augmented Mean-Square Bound}
\label{sec:iterate-mse}

The positive-stable, two-time-scale spectrum now underpins a mean-square
bound.  The following proposition builds an explicit Lyapunov function for
the joint drift and establishes the anisotropic two-time-scale split,
supplying hypothesis~(b) of the CLT (Theorem~\ref{thm:augmented-clt})
through the buffer rate $\mathbb E\|m_{t-1}\|^2 = O(t^{-\gamma})$.

\begin{proposition}[Augmented Lyapunov bound]
\label{prop:augmented-mse}
Adopt Assumptions~\ref{ass:mg}--\ref{ass:quadratic} and
Assumption~\ref{ass:stab} (the iterate bound of
Assumption~\ref{ass:iterate} is \emph{not} assumed---it is the
conclusion), bounded gradients $\|g_t\| \le G$ a.s., and the
fourth-moment stability condition
$\mathbb E\|\Delta_t\|^4 = O(t^{-2\alpha})$
(Assumption~\ref{ass:fourth}).  Assume further that the preconditioner
$P_t$ is \emph{predictable} ($\mathcal F_{t-1}$-measurable) and satisfies
the standing conditions
\begin{equation}
\label{eq:prop-precond}
  p_- I \preceq P_t \preceq p_+ I,\quad
  \mu_- I \preceq P_t^{1/2} H P_t^{1/2} \preceq \mu_+ I,\quad
  \|P_{t+1} - P_t\|_{\mathrm{op}} \le C_P\,t^{-1},
\end{equation}
for $t$-independent constants $0 < p_- \le p_+ < \infty$,
$0 < \mu_- \le \mu_+ < \infty$, $C_P < \infty$ (these are the conditions
restated as~\eqref{eq:appA-standing} in Appendix~\ref{app:lyapunov}; for
SA-Adam $P_t = \mathrm{Diag}(\hat v_{t-1})^{-1/2}$ is predictable and they
follow from bounded gradients and $\epsilon > 0$), and the momentum
schedule $\rho_t = c_1/t^\gamma$ (shifted as in
Section~\ref{sec:sa-adam-defn} when $c_1 \ge 1$, so that
$\rho_t \in (0,1)$) with $c_1 > 0$ and $\gamma \in (\alpha, 1)$.  Then,
with $V_t = z_t^\top Q_t z_t$ the
Lyapunov function of Appendix~\ref{app:lyapunov}, $\mathbb E V_t =
O(\eta_t)$, and consequently---via the anisotropic bound
$V_t \ge c_0(\|\Delta_t\|^2 + \tau_t^{-1}\|m_{t-1}\|^2)$ of
Lemma~\ref{lem:Qsym-bounds}---the two-time-scale split
\begin{equation}
\label{eq:prop-mse-split}
  \mathbb E\|z_t\|^2 = O(t^{-\alpha}),
  \qquad
  \mathbb E\|\Delta_t\|^2 = O(t^{-\alpha}),
  \qquad
  \mathbb E\|m_{t-1}\|^2 = O(\tau_t\eta_t) = O(t^{-\gamma}).
\end{equation}
In particular, Assumption~\ref{ass:iterate} is a conclusion here rather
than a separate hypothesis---indeed, given Assumption~\ref{ass:fourth},
the second-moment rate already follows by Jensen's inequality, so the
proposition's distinctive content is instead the buffer rate
$\mathbb E\|m_{t-1}\|^2 = O(t^{-\gamma})$, which supplies
hypothesis~(b),~\eqref{eq:clt-stability}, of
Theorem~\ref{thm:augmented-clt}.
\end{proposition}

\begin{proof}[Proof sketch; full proof in Appendix~\ref{app:lyapunov}]
Standard Lyapunov argument for linear SA with a positive-stable drift
(Lemma~\ref{lem:hurwitz}).  The mechanism is an explicit \emph{symmetrizer}
$Q_t$ (Definition~\ref{def:Qsym}), which obeys the \emph{exact} one-step
identity $(I-\eta_t L_t)^\top Q_t (I-\eta_t L_t) = (1-\rho_t)\,Q_t$
(Lemma~\ref{lem:Qsym-decrease}) together with the $t$-uniform two-sided
bound $c_0\,\mathrm{diag}(I, \tau_t^{-1} I) \preceq Q_t \preceq
C_0\,\tau_t^{-1} I_{2d}$ (anisotropic in the lower bound, weighting the
buffer block by $\tau_t^{-1}$; Lemma~\ref{lem:Qsym-bounds}).  Setting
$V_t(z) = z^\top Q_t z$ and conditioning with the \emph{predictable} $Q_t$
($Q_{t+1}$ is only $\mathcal F_t$-measurable), the one-step identity gives
$\mathbb{E}[z_{t+1}^\top Q_t z_{t+1} \mid \mathcal F_{t-1}]
\le (1 - \rho_t)\, V_t(z_t) + \mathrm{tr}(Q_t B_t \overline S B_t^\top)
+ (\text{Taylor})$, the noise term being
$\mathrm{tr}(Q_t B_t \overline S B_t^\top) = O(\eta_t\rho_t)$ (since
$\rho_t/\tau_t = \eta_t$).  Accounting for the time variation
$Q_{t+1}-Q_t$ blockwise (Appendix~\ref{app:lyapunov}) gives, for any
$c_1 > 0$ under $\gamma\in(\alpha,1)$, the recursion $\mathbb E V_t =
O(t^{-\alpha})$; together with the $t$-uniform lower bound $Q_t \succeq
c_0 I_{2d}$, this yields $\mathbb E\|z_t\|^2 \le c_0^{-1}\,\mathbb E V_t
= O(t^{-\alpha})$.  The contraction $\rho_t = c_1/t^\gamma$ dominates
the $O(1/t)$ time-variation feedback because $\gamma<1$, so no lower
bound on $c_1$ beyond $c_1>0$ is required (in contrast to the boundary
case $\gamma = 1$, where a finite threshold $c_1 > c_1^\ast$ would be
needed).
\end{proof}

\section{Asymptotic Normality and the Projection Identity}
\label{sec:asymptotics}

This section establishes the two halves of the asymptotic result: a
non-autonomous Polyak--Ruppert CLT for the joint
chain~\eqref{eq:joint-recursion}, and the projection identity that
evaluates its iterate-marginal covariance as the canonical sandwich
$H^{-1}SH^{-1}$, independent of the \emph{frozen} momentum-decay rate
$\rho$, time-scale ratio $\tau$, and preconditioner $P$.  This
frozen-parameter invariance is the content of the projection identity;
the central limit theorem itself still requires the sub-linear momentum
schedule $\gamma\in(\alpha,1)$.

\subsection{The Non-autonomous Polyak--Ruppert CLT}
\label{sec:clt}

We first state the main result of the asymptotic analysis: a
Polyak--Ruppert CLT for the joint chain~\eqref{eq:joint-recursion}.

Let $\Sigma_w(t) := B_t S B_t^\top / \eta_t^2$ denote the
$\eta_t$-normalized one-step noise covariance---formed with the
\emph{limiting} noise level $S$ at $x^*$, not the conditional
$S(x_t)$---which after substituting
$B_t = (-\eta_t \rho_t P_t, \rho_t I)^\top$ equals
\[
  \Sigma_w(t) =
  \begin{pmatrix}
    \rho_t^2 P_t S P_t & -\rho_t \tau_t P_t S \\
    -\rho_t \tau_t S P_t & \tau_t^2 S
  \end{pmatrix}.
\]
Let $\Sigma_z(t)$ denote the leading-order (classical) Polyak--Ruppert
covariance associated with the frozen coefficients $(L_t, \Sigma_w(t))$;
in the decaying-step regime this is the relevant limit, the $O(\eta_t)$
fixed-step correction carried by the exact covariance of
\citet{mou2020linear} being lower order.  As for any linear
stochastic-approximation recursion with iterate averaging
\citep{polyak1992acceleration, mou2020linear}, this leading covariance is
\emph{not} the stationary (Ornstein--Uhlenbeck) covariance, which solves
the additive Lyapunov equation $L_t\,\Sigma + \Sigma\,L_t^\top =
\Sigma_w(t)$.  It is instead given by the \emph{congruence}
\begin{equation}
\label{eq:freeze-lyap}
  L_t\,\Sigma_z(t)\,L_t^\top = \Sigma_w(t),
  \qquad\text{equivalently}\qquad
  \Sigma_z(t) = L_t^{-1}\,\Sigma_w(t)\,L_t^{-\top}.
\end{equation}
The two notions are genuinely different matrices---already in the scalar
case, averaging gives variance $q/a^2$ whereas the stationary diffusion
gives $q/(2a)$ for drift $a$ and noise level $q$---and conflating them
changes the iterate-marginal block by an $O(1)$ matrix factor.  The
closed-form right-hand side of~\eqref{eq:freeze-lyap} is the object
Section~\ref{sec:projection} below evaluates explicitly.

\begin{theorem}[Augmented Polyak--Ruppert CLT]
\label{thm:augmented-clt}
For the joint chain~\eqref{eq:joint-recursion}---the un-bias-corrected
centered recursion of Section~\ref{sec:sa-adam-defn}, the first-moment
bias correction being restored as an effective step size in
Theorem~\ref{thm:main}---suppose Assumptions~\ref{ass:mg}--\ref{ass:quadratic}
and~\ref{ass:stab} hold, gradients are bounded ($\|g_t\| \le G$ a.s.), the
preconditioner is predictable ($P_t$ is $\mathcal F_{t-1}$-measurable, so
$L_t$ and $B_t$ are too) and obeys the standing
conditions~\eqref{eq:appA-standing} (in particular
$\|P_{t+1} - P_t\|_{\mathrm{op}} = O(t^{-1})$), and the momentum schedule
is $\rho_t = c_1/t^\gamma$ with $\gamma\in(\alpha,1)$,
$\alpha\in(1/2,1)$ (the lower bound $\gamma>\alpha$ gives the
two-time-scale separation $\tau_t = \rho_t/\eta_t \to 0$; the upper bound
$\gamma<1$ controls the Lyapunov time-variation in
Proposition~\ref{prop:augmented-mse} and the endpoint-buffer remainder in
Appendix~\ref{app:clt-proof}).  Suppose moreover:
\begin{enumerate}
\item[\textnormal{(a)}] (\emph{Conditional-covariance continuity at
  $x^*$})  $\mathbb E[\xi_t\xi_t^\top \mid \mathcal F_{t-1}] = S(x_t)$
  for a deterministic map $S(\cdot)$ with $S(x) \to S$ as $x \to x^*$
  (as in the conditional-covariance continuity assumed by \citet{anhuo2026});
\item[\textnormal{(b)}] (\emph{Augmented mean-square bounds})
\begin{equation}
\label{eq:clt-stability}
  \mathbb E\|\Delta_t\|^2 = O(t^{-\alpha}),
  \qquad
  \mathbb E\|m_{t-1}\|^2 = O(t^{-\gamma}),
\end{equation}
  which hold under the hypotheses of Proposition~\ref{prop:augmented-mse}
  (in particular Assumption~\ref{ass:fourth}).
\end{enumerate}
Then
\begin{equation}
\label{eq:joint-clt}
  \sqrt n\,\overline z_n \xrightarrow{d}
  \mathcal{N}\!\left(0,\ \begin{pmatrix} H^{-1}SH^{-1} & 0 \\ 0 & 0
  \end{pmatrix}\right).
\end{equation}
The iterate block is the projection identity of
Theorem~\ref{thm:projection}; the buffer row and column vanish because
$A_t B_t = (-H^{-1}, 0)^\top$ eliminates the leading buffer term and the
remainder bound gives $\sqrt n\,\overline m_{n-1} \to 0$ (both in
Appendix~\ref{app:clt-proof}).  In particular the iterate marginal
satisfies $\sqrt n\,\overline\Delta_n \xrightarrow{d}
\mathcal N(0, H^{-1}SH^{-1})$.  The complete proof is given in
Appendix~\ref{app:clt-proof}.
\end{theorem}

\begin{proof}[Proof sketch; full proof in Appendix~\ref{app:clt-proof}]
The argument adapts the linear-SA Polyak--Ruppert framework of
\citet{mou2020linear} to the time-varying drift $L_t$.  Left-multiplying
the recursion by $A_t := L_t^{-1}/\eta_t$ and Abel-summing yields the
decomposition $\overline z_n = \overline{A_t B_t\,\xi_t} +
\overline{A_t B_t\,u_t} + R_n^z$, whose leading coefficient is the
\emph{exact} constant $A_t B_t = (-H^{-1},0)^\top$ (Eq.~\eqref{eq:AtBt});
the martingale term then obeys a Hall--Heyde CLT with covariance
$\mathrm{diag}(H^{-1}SH^{-1},0)$, the Taylor term vanishes in $L^1$ under
the iterate bound of~(b), and the boundary/increment remainder
$R_n^z = n^{-1}[A_1 z_1 - A_n z_{n+1} + \sum_{t=2}^n (A_t - A_{t-1})z_t]$
is $o(n^{-1/2})$ in $L^2$ (Lemma~\ref{lem:Rnz-bound}).  The one genuinely
new step is this time-varying Abel summation, whose error consumes the
$O(t^{-1})$ preconditioner increment of~\eqref{eq:appA-standing}---a rate
strictly stronger than the bare stabilization threshold of
Assumption~\ref{ass:stab}.  Slutsky's lemma then assembles the joint
CLT~\eqref{eq:joint-clt}.
\end{proof}

\emph{No convergence hypothesis on $P_t$.}
The limit in~\eqref{eq:joint-clt} exists without assuming $P_t \to P$.
The frozen-chain covariance $\Sigma_z(t) = L_t^{-1}\Sigma_w(t)
L_t^{-\top}$ equals $\mathrm{diag}(H^{-1}SH^{-1}, 0)$ for \emph{every}
$t$---the $(1,1)$ block by Theorem~\ref{thm:projection}, the buffer row
and column by the exact identity~\eqref{eq:AtBt}---so the limiting
covariance \emph{value} needs only that the projection identity hold
per-$t$, which it does for any $P_t \succ 0$, $\rho_t \in (0,1]$, and
$\tau_t > 0$.  The remainder bound additionally uses the boundedness and
$O(t^{-1})$ one-step increment of $P_t$ from~\eqref{eq:appA-standing}
(supplied by SA-Adam, Proposition~\ref{prop:sa-adam-stab}); what neither
needs is a limit $P_t \to P$.

\subsection{The Iterate-Marginal Projection Identity}
\label{sec:projection}

We first record the block inverse of the frozen drift, the computational
ingredient of the projection identity.  Let $L = L(P, \rho, \tau)$ denote
the frozen drift matrix~\eqref{eq:Lt} with fixed $P, \rho, \tau$, and let
$M := (PH)^{-1}$.

\begin{lemma}[Inverse of $L$]
\label{lem:Linv}
Suppose $P \succ 0$, $H \succ 0$, and $\tau > 0$ (so that
$M := (PH)^{-1}$ exists); $\rho \in \mathbb R$ is arbitrary.  Then $L$ is
invertible, with
\begin{equation}
\label{eq:Linv}
  L^{-1} = \begin{pmatrix}
    M & -(1-\rho)\tau^{-1} H^{-1} \\
    P^{-1} & \rho\tau^{-1} I
  \end{pmatrix}.
\end{equation}
\end{lemma}

\begin{proof}
Write $L$ in $d\times d$ blocks as $L = \bigl(\begin{smallmatrix} A & B
\\ C & D \end{smallmatrix}\bigr)$ with $A = \rho PH$, $B = (1-\rho)P$,
$C = -\tau H$, $D = \tau I$.  The lower-right block $D = \tau I$ is invertible
($\tau > 0$), and its Schur complement
$A - BD^{-1}C = \rho PH + (1-\rho)PH = PH$ is invertible ($P, H \succ 0$) with
$(A-BD^{-1}C)^{-1} = (PH)^{-1} = M$; so $L$ is invertible and~\eqref{eq:Linv}
is its standard Schur-complement block inverse \citep[\S0.7.3]{hornjohnson2013}.
Direct multiplication confirms $L L^{-1} = I_{2d}$, using $M = H^{-1}P^{-1}$,
$HM = P^{-1}$, and $(PH)M = I$ (e.g.\ the off-diagonal blocks vanish:
$(LL^{-1})_{21} = -\tau HM + \tau P^{-1} = 0$ and
$(LL^{-1})_{12} = -\rho(1-\rho)\tau^{-1}P + \rho(1-\rho)\tau^{-1}P = 0$).
\end{proof}

With the block inverse in hand, we evaluate the iterate-marginal
covariance and find it independent of the preconditioner and the
momentum scales.

\begin{theorem}[Iterate-marginal projection identity]
\label{thm:projection}
Let $H = H^\top \succ 0$, $P = P^\top \succ 0$, $S = S^\top \succeq 0$,
and $\rho \in (0,1]$, $\tau > 0$, and let $L = L(P, \rho, \tau)$ be the frozen
drift~\eqref{eq:Lt}.  Let $\Sigma_w$ denote the frozen normalized
one-step noise covariance
\[
  \Sigma_w = \begin{pmatrix}
    \rho^2 P S P & -\rho\tau P S \\
    -\rho\tau S P & \tau^2 S
  \end{pmatrix}.
\]
Then the Polyak--Ruppert covariance $\Sigma_z = L^{-1}\Sigma_w
L^{-\top}$ satisfies
\begin{equation}
\label{eq:projection}
  \Sigma_z^{(1,1)} = H^{-1} S H^{-1},
\end{equation}
\emph{independently} of $\rho$, $\tau$, and $P$.
\end{theorem}

\begin{proof}
Direct computation from the block inverse $L^{-1}$ of
Lemma~\ref{lem:Linv}, whose top row has blocks $(L^{-1})_{11} = M =
(PH)^{-1} = H^{-1}P^{-1}$ and $(L^{-1})_{12} = -(1-\rho)\tau^{-1}H^{-1}$.
Compute the top row of $L^{-1}\Sigma_w$:
\begin{align*}
  (L^{-1}\Sigma_w)_{11}
  &= M \cdot \rho^2 PSP + \bigl(-(1-\rho)\tau^{-1}H^{-1}\bigr)\cdot
     (-\rho\tau SP) \\
  &= \rho^2\, H^{-1}P^{-1}\cdot PSP + (1-\rho)\rho\, H^{-1}SP \\
  &= \rho^2 H^{-1}SP + (1-\rho)\rho H^{-1}SP \\
  &= \rho\, H^{-1}SP.
\end{align*}
\begin{align*}
  (L^{-1}\Sigma_w)_{12}
  &= M \cdot (-\rho\tau PS) + \bigl(-(1-\rho)\tau^{-1}H^{-1}\bigr)\cdot \tau^2 S \\
  &= -\rho\tau\, H^{-1}P^{-1}\cdot PS - (1-\rho)\tau\, H^{-1}S \\
  &= -\rho\tau\, H^{-1}S - (1-\rho)\tau\, H^{-1}S \\
  &= -\tau\, H^{-1}S.
\end{align*}

Now compute the $(1,1)$ block of $L^{-1}\Sigma_w L^{-\top}$.  Its first
block-column draws on the blocks $(L^{-\top})_{11} = (L^{-1})_{11}^\top =
M^\top = P^{-1}H^{-1}$ and $(L^{-\top})_{21} = (L^{-1})_{12}^\top =
-(1-\rho)\tau^{-1}H^{-1}$ (using $H = H^\top$, $P = P^\top$):
\begin{align*}
  \Sigma_z^{(1,1)}
  &= (L^{-1}\Sigma_w)_{11}\,(L^{-\top})_{11}
   + (L^{-1}\Sigma_w)_{12}\,(L^{-\top})_{21} \\
  &= (\rho H^{-1}SP)(P^{-1}H^{-1})
   + (-\tau H^{-1}S)\bigl(-(1-\rho)\tau^{-1}H^{-1}\bigr) \\
  &= \rho\, H^{-1}S\, P P^{-1}\, H^{-1}
   + (1-\rho)\, H^{-1}S H^{-1} \\
  &= \rho\, H^{-1}SH^{-1} + (1-\rho)\, H^{-1}SH^{-1} \\
  &= H^{-1} S H^{-1}.
\end{align*}
Three exact cancellations produce the $P$-, $\rho$-, and
$\tau$-independence: the preconditioner cancels through $PP^{-1} = I$ in
the first term, the timescale prefactors cancel through
$\tau\cdot\tau^{-1} = 1$ in the second, and the weights $\rho$ and
$1-\rho$ sum to $1$.  No assumption on $P, \rho, \tau$ beyond
positivity---$P \succ 0$ and $\rho \in (0,1]$, $\tau > 0$, which ensure
$L$ is invertible by Lemma~\ref{lem:Linv}---is used; in particular no
commutativity of $P$ and $H$ is required.
\end{proof}

\begin{remark}[Unification of two prior preservation results]
\label{rem:unification}
Theorem~\ref{thm:projection} unifies two prior preservation results:
\citet{liu2023acceleration} showed averaged SGD with momentum (SGDM) is asymptotically
equivalent to averaged SGD with sandwich covariance $H^{-1}SH^{-1}$
(momentum, no preconditioning; \citet{wei2025weighted} reach the same
form up to a weight-dependent scalar), and \citet{anhuo2026} showed the
iterate-marginal covariance equals $H^{-1}SH^{-1}$ independently of a
non-convergent $P_t$ (preconditioning, no momentum).  The present theorem
establishes the identity for both effects simultaneously, hence for
time-varying momentum with a non-convergent preconditioner; the
$P$-independence is the algebraic reason preconditioning changes
finite-sample behavior but not the first-order asymptotic covariance.
\end{remark}

\section{The SA-Adam Main Theorem and Variants}
\label{sec:sa-adam}

The preceding theory is stated for an abstract preconditioner satisfying
the stabilization and standing conditions.  We now instantiate it for the
concrete SA-Adam algorithm of Section~\ref{sec:sa-adam-defn}:
Section~\ref{sec:sa-adam-verify} verifies that SA-Adam's data-driven
preconditioner meets those conditions,
Section~\ref{sec:bias-correction} reduces the first-moment bias correction
to an effective step size, Section~\ref{sec:main-theorem} assembles the
resulting Polyak--Ruppert CLT (Theorem~\ref{thm:main}), and
Section~\ref{sec:variants} treats the SA-AMSGrad, coupled ($L_2$)
weight-decay, and full-matrix SA-Adam-full variants.

\subsection{Preconditioner Verification}
\label{sec:sa-adam-verify}

To apply the general theory to SA-Adam we must check that its data-driven
preconditioner $P_t = \mathrm{Diag}(\hat v_{t-1})^{-1/2}$ satisfies the
conditions assumed abstractly above.  The following proposition discharges
both the rate-only stabilization condition (Assumption~\ref{ass:stab},
with $\beta = 1$) and the standing conditions~\eqref{eq:prop-precond}; the
second-moment argument parallels the SA-RMSProp verification of
\citet{anhuo2026}, the $\hat v$ bias-correction step being SA-Adam-specific.

\begin{proposition}[Preconditioner verification for SA-Adam]
\label{prop:sa-adam-stab}
Consider the SA-Adam recursion~\eqref{eq:sa-adam} with $c_2 \in (0, 1]$,
bounded stochastic gradients $\|g_t\| \le G$ a.s., initialization
$v_0 = \epsilon\mathbf 1$ ($\epsilon > 0$), and the convention
$\hat v_0 := v_0$.  (The momentum parameters $c_1, \gamma$ and the
initialization $m_0$ do not enter this verification beyond the
bounded-gradient assumption: momentum shapes the trajectory $x_t$ and
hence $g_t$, but the preconditioner $P_t = \mathrm{Diag}(\hat
v_{t-1})^{-1/2}$ uses $g_t$ only through the second-moment buffer $v_t$,
which the analysis controls via $\|g_t\| \le G$ alone.)  Then, writing
$M_t := (P_t H)^{-1}$:
\begin{enumerate}
\item[\textnormal{(i)}] (\emph{Stabilization}, Assumption~\ref{ass:stab}
  with $\beta = 1$) $\|M_t - M_{t-1}\|_{\mathrm{op}} \le C\, t^{-1}$
  a.s.\ for all $t \ge 2$, and $\sup_t \|M_t\|_{\mathrm{op}} < \infty$;
\item[\textnormal{(ii)}] (\emph{Standing conditions}
  \eqref{eq:prop-precond}, equivalently~\eqref{eq:appA-standing}) there
  exist constants $0 < p_- \le p_+$, $0 < \mu_- \le \mu_+$,
  $C_P < \infty$ such that, for all $t \ge 1$,
  $p_- I \preceq P_t \preceq p_+ I$,
  $\mu_- I \preceq P_t^{1/2} H P_t^{1/2} \preceq \mu_+ I$, and
  $\|P_{t+1} - P_t\|_{\mathrm{op}} \le C_P\, t^{-1}$.
\end{enumerate}
\end{proposition}

\begin{proof}
\emph{The $v_t$-preconditioner (parallel to SA-RMSProp).}  SA-Adam's
preconditioner is built from the second-moment buffer exactly as
SA-RMSProp's, so the boundedness $v_t \in [\epsilon, G^2+\epsilon]^d$, the
increment $\|v_t - v_{t-1}\|_\infty \le c_2 G^2/t = O(t^{-1})$, and the
resulting stabilization $\|M_t - M_{t-1}\|_{L^2(\mathrm{op})} = O(t^{-1})$
upgrading to the pathwise rate $\beta = 1$ under bounded gradients are the
SA-RMSProp case of \citet{anhuo2026}, whose argument we recall here before
verifying the SA-Adam-specific bias correction in full.  The inputs are elementary: bounded gradients give $g_t \odot g_t +
\epsilon\mathbf 1 \in [\epsilon, G^2+\epsilon]$ coordinatewise, so
$v_t$---a convex combination started at $v_0 = \epsilon\mathbf 1$---stays
in $[\epsilon, G^2+\epsilon]$ with $\|v_t - v_{t-1}\|_\infty =
\rho^v_t\|g_t \odot g_t + \epsilon\mathbf 1 - v_{t-1}\|_\infty \le c_2
G^2/t$.

\emph{Bias correction (the one SA-Adam-specific step).}  SA-Adam divides
by $\kappa^v_t = 1 - \prod_{s=1}^t(1 - c_2/s)$, which RMSProp does not.
As $\kappa^v_t$ increases from $\kappa^v_1 = c_2$ to $1$ it is bounded
away from $0$, with increment $\kappa^v_t - \kappa^v_{t-1} =
(c_2/t)\prod_{s=1}^{t-1}(1 - c_2/s) = O(t^{-c_2-1})$ (the product is
$O(t^{-c_2})$).  For $t \ge 2$ (so that $\kappa^v_{t-1}$ is defined), the
exact bias-corrected increment
\[
  \hat v_t - \hat v_{t-1}
  = \frac{v_t - v_{t-1}}{\kappa^v_t}
    + v_{t-1}\Bigl(\frac{1}{\kappa^v_t} - \frac{1}{\kappa^v_{t-1}}\Bigr)
\]
has first term $O(t^{-1})$ (numerator $O(t^{-1})$, denominator
$\ge c_2$) and second term $O(t^{-c_2-1})$ (since $\tfrac{1}{\kappa^v_t}
- \tfrac{1}{\kappa^v_{t-1}} = -\tfrac{\kappa^v_t - \kappa^v_{t-1}}
{\kappa^v_t \kappa^v_{t-1}} = O(t^{-c_2-1})$ and $v_{t-1}$ is bounded), so
$\|\hat v_t - \hat v_{t-1}\|_\infty = O(t^{-1})$.  The initial step
$t = 1$, where $\hat v_0 := v_0$ by convention, is a single finite
exception absorbed into the constants, so $\|\hat v_t - \hat
v_{t-1}\|_\infty = O(t^{-1})$ for all $t \ge 1$.  Moreover $\hat v_{t-1}
\in [\epsilon, c_G]$ coordinatewise for all $t \ge 1$ ($\hat v_{t-1} =
v_{t-1}/\kappa^v_{t-1}$ for $t \ge 2$, and $\hat v_0 = v_0 =
\epsilon\mathbf 1$), with $c_G := (G^2+\epsilon)/c_2$.

\emph{Conclusion (both claims).}  Both $x \mapsto x^{-1/2}$ and
$x \mapsto x^{1/2}$ are Lipschitz on $[\epsilon, \infty)$, and
$P_t = \mathrm{Diag}(\hat v_{t-1})^{-1/2}$, $P_t^{-1} =
\mathrm{Diag}(\hat v_{t-1})^{1/2}$ are diagonal, so $\hat v_{t-1} \in
[\epsilon, c_G]$ gives the two-sided bounds $c_G^{-1/2} I \preceq P_t
\preceq \epsilon^{-1/2} I$ (whence the $\mu_\pm$-bounds on
$P_t^{1/2}HP_t^{1/2}$, whose eigenvalues equal those of $P_t H$ and lie in
$[p_-\lambda_{\min}(H), p_+\lambda_{\max}(H)]$), and the $O(t^{-1})$
increment of $\hat v_{t-1}$ gives $\|P_{t+1} - P_t\|_{\mathrm{op}} =
O(t^{-1})$ and, via $M_t = H^{-1}P_t^{-1}$, $\|M_t - M_{t-1}\|_{\mathrm
{op}} \le \|H^{-1}\|_{\mathrm{op}}\|P_t^{-1} - P_{t-1}^{-1}\|_{\mathrm{op}}
= O(t^{-1})$ with $\sup_t\|M_t\|_{\mathrm{op}} \le \|H^{-1}\|_{\mathrm{op}}
c_G^{1/2} < \infty$.  This is (i) and~(ii).
\end{proof}

\subsection{Bias-Correction Reduction}
\label{sec:bias-correction}

The analysis in Sections~\ref{sec:augmented} and~\ref{sec:asymptotics} and
Appendices~\ref{app:lyapunov}--\ref{app:clt-proof} replaces $\hat m_t
= m_t/\kappa^m_t$ by $m_t$ in the iterate update---the ``drop the bias
correction'' simplification flagged in Section~\ref{sec:sa-adam-defn}.
The bias-corrected preconditioner $P_t = \mathrm{Diag}(\hat
v_{t-1})^{-1/2}$ is already used as-is in the analysis, and
Proposition~\ref{prop:sa-adam-stab} verifies its stabilization condition
(Assumption~\ref{ass:stab}, cf.\ Appendix~\ref{app:lyapunov}); the only
piece to restore is therefore the first-moment correction $\hat m_t$ in
the iterate update.

Since the first-moment bias correction
is a deterministic scalar factor, $\hat m_t = m_t/\kappa^m_t$ folds into
the step size---the standard view of Adam's \emph{first-moment} bias
correction as an effective learning-rate factor \citep{kingma2015adam}
(the second-moment correction is not folded in here, being already
carried by the preconditioner $P_t$).  Writing this out in the
SA-Adam update~\eqref{eq:sa-adam} yields the \emph{exact} substitution
\begin{equation}
\label{eq:effective-stepsize}
  x_{t+1} = x_t - \eta_t\, P_t\, \hat m_t
          = x_t - \tilde\eta_t\, P_t\, m_t,
  \qquad
  \tilde\eta_t := \eta_t/\kappa^m_t .
\end{equation}
The buffer recursion for $m_t$ in~\eqref{eq:sa-adam} is untouched by the
first-moment correction.  Hence the bias-corrected algorithm is
\emph{identical} to the un-bias-corrected centered
recursion~\eqref{eq:error-recursion} under the deterministic scalar
step-size substitution $\eta_t \mapsto \tilde\eta_t$; the only task is to
record that $\tilde\eta_t$ retains the rate properties the analysis uses.
We argue this structurally, as a reparametrization of the same analyzed
recursion---not as a comparison of two distinct trajectories.

\begin{lemma}[First-moment bias correction as an effective step size]
\label{lem:bias-effective-step}
Let $\rho_t = c_1/(t + t_0)^\gamma$ with $c_1 > 0$, $\gamma \in (\alpha, 1)$,
and a burn-in offset $t_0 \ge 0$ fixed (per the convex-weight convention
of Section~\ref{sec:sa-adam-defn}) so that $\rho_s \in (0,1)$ for all
$s \ge 1$; the plain schedule $t_0 = 0$ is admissible whenever $c_1 < 1$.
Let $\kappa^m_t = 1 - \prod_{s=1}^t(1 - \rho_s)$ be the first-moment
bias-correction factor of~\eqref{eq:sa-adam}.  Every factor
$1 - \rho_s \in (0,1)$, so $\kappa^m_t \in (0,1)$ and increases to $1$.
Then the bias-corrected update
$x_{t+1} = x_t - \eta_t P_t \hat m_t$ coincides with the
un-bias-corrected recursion~\eqref{eq:error-recursion} at step size
$\tilde\eta_t = \eta_t/\kappa^m_t$, and for some $a > 0$,
\begin{equation}
\label{eq:tilde-eta-rate}
  \kappa^m_t = 1 - O\!\bigl(e^{-a t^{1-\gamma}}\bigr),
  \qquad
  \tilde\eta_t = \eta_0\, t^{-\alpha}\bigl(1 + O(e^{-a t^{1-\gamma}})\bigr).
\end{equation}
Consequently $\tilde\eta_t$ inherits the three rate facts on which the
augmented-state analysis of
Sections~\ref{sec:augmented}--\ref{sec:asymptotics} relies:
\begin{equation}
\label{eq:tilde-eta-facts}
  \tilde\eta_t \sim \eta_0 t^{-\alpha},
  \qquad
  |\tilde\eta_{t+1} - \tilde\eta_t| = O(t^{-\alpha-1}),
  \qquad
  \tilde\tau_t := \rho_t/\tilde\eta_t = \kappa^m_t\,\tau_t
  = \tau_t\bigl(1 + o(1)\bigr) \asymp t^{\alpha-\gamma} \to 0 .
\end{equation}
\end{lemma}

\begin{proof}
The substitution~\eqref{eq:effective-stepsize} is an exact algebraic
identity, so only the rate
claims~\eqref{eq:tilde-eta-rate}--\eqref{eq:tilde-eta-facts} require
proof.

\emph{Rate of $\kappa^m_t$.}  By the convex-weight convention,
$\rho_s \in (0,1)$ for all $s \ge 1$, so
$1 - \kappa^m_t = \prod_{s=1}^t(1 - \rho_s) \in (0,1)$ with every factor
positive---no burn-in prefactor is needed.  On
$[0, \sup_{s\ge1}\rho_s] \subset [0,1)$ the expansion
$\log(1-x) = -x + O(x^2)$ holds with a uniform constant, so
$\log\bigl(1 - \kappa^m_t\bigr) = \sum_{s=1}^t \log(1 - \rho_s)
  = -\sum_{s=1}^t \rho_s + O\Bigl(\sum_{s=1}^t \rho_s^2\Bigr)$.
Because $\gamma > 1/2$ (indeed $\gamma > \alpha > 1/2$),
$\sum_s \rho_s^2 = O\bigl(\sum_s s^{-2\gamma}\bigr) < \infty$, so the
remainder is $O(1)$; and by integral comparison, $\sum_{s=1}^t \rho_s =
\frac{c_1}{1-\gamma}\,(t+t_0)^{1-\gamma} + O(1) =
\frac{c_1}{1-\gamma}\,t^{1-\gamma} + O(1)$ (the offset $t_0$ enters only
the $O(1)$ term).  Hence
$\log(1 - \kappa^m_t) = -\frac{c_1}{1-\gamma}\,t^{1-\gamma} + O(1)$, and
therefore, for every $0 < a < c_1/(1-\gamma)$,
$1 - \kappa^m_t = \prod_{s=1}^t(1 - \rho_s)
  = O\bigl(e^{-a t^{1-\gamma}}\bigr)$,
the first identity in~\eqref{eq:tilde-eta-rate}.  In particular
$\kappa^m_t \to 1$, so $\kappa^m_t \ge \tfrac12$ for all $t$ large and
$(\kappa^m_t)^{-1} = 1 + O(e^{-a t^{1-\gamma}})$, giving
$\tilde\eta_t = \eta_t/\kappa^m_t = \eta_0 t^{-\alpha}
\bigl(1 + O(e^{-a t^{1-\gamma}})\bigr)$---the second identity
in~\eqref{eq:tilde-eta-rate} and the leading power
$\tilde\eta_t \sim \eta_0 t^{-\alpha}$ in~\eqref{eq:tilde-eta-facts}.

\emph{Smoothness.}  Decompose the increment at $t$ as
$\tilde\eta_{t+1} - \tilde\eta_t
  = \frac{\eta_{t+1} - \eta_t}{\kappa^m_{t+1}}
  - \eta_t\,\frac{\kappa^m_{t+1} - \kappa^m_t}
                 {\kappa^m_{t+1}\,\kappa^m_t}$.
The first term is $O(t^{-\alpha-1})$, since $|\eta_{t+1} - \eta_t| =
\eta_0\bigl|(t+1)^{-\alpha} - t^{-\alpha}\bigr| = O(t^{-\alpha-1})$ and
$\kappa^m_{t+1} \ge \tfrac12$.  For the second, the telescoping
$\kappa^m_{t+1} - \kappa^m_t = \prod_{s=1}^{t}(1-\rho_s)
- \prod_{s=1}^{t+1}(1-\rho_s) = \rho_{t+1}\prod_{s=1}^{t}(1-\rho_s)
= O\bigl(t^{-\gamma} e^{-a t^{1-\gamma}}\bigr)$, multiplied by
$\eta_t/(\kappa^m_{t+1}\kappa^m_t) = O(t^{-\alpha})$, is
$O\bigl(t^{-\alpha-\gamma} e^{-a t^{1-\gamma}}\bigr)$: super-polynomially
small, hence $o(t^{-\alpha-1})$.  So $|\tilde\eta_{t+1} - \tilde\eta_t| =
O(t^{-\alpha-1})$, the second fact in~\eqref{eq:tilde-eta-facts}.

\emph{Two-time-scale gap.}  Finally
$\tilde\tau_t = \rho_t/\tilde\eta_t = \kappa^m_t\,(\rho_t/\eta_t)
= \kappa^m_t\,\tau_t = \tau_t\bigl(1 + O(e^{-a t^{1-\gamma}})\bigr)$, and
$\tau_t \sim (c_1/\eta_0)\,t^{\alpha-\gamma} \to 0$ because
$\gamma > \alpha$ (with equality for the plain schedule $t_0 = 0$);
this is the third fact in~\eqref{eq:tilde-eta-facts}.
\end{proof}

The bias correction is thus \emph{better behaved} than the canonical
$\gamma = 1$ case, where $1 - \kappa^m_t$ decays only polynomially.
With bias correction restored, the corrected recursion has the same
algebraic form.  The bias-corrected analogue of the centered
augmented recursion~\eqref{eq:joint-recursion} reads $z_{t+1} =
(I - \tilde\eta_t \tilde L_t)z_t + \tilde B_t(u_t + \xi_t)$ with
\[
  \tilde L_t = \begin{pmatrix}
    \rho_t P_t H & (1-\rho_t) P_t \\
    -\tilde\tau_t H & \tilde\tau_t I
  \end{pmatrix}
  = L(P_t, \rho_t, \tilde\tau_t),
  \qquad
  \tilde B_t = \begin{pmatrix} -\tilde\eta_t \rho_t P_t \\ \rho_t I
  \end{pmatrix},
  \qquad
  \tilde\tau_t = \rho_t/\tilde\eta_t .
\]
This is exactly the drift~\eqref{eq:Lt} and noise coupling $B_t$
of~\eqref{eq:joint-recursion} with $(\eta_t, \tau_t, L_t, B_t)$ replaced
by $(\tilde\eta_t, \tilde\tau_t, \tilde L_t, \tilde B_t)$; the
substitution acts only through the deterministic step size and leaves
the buffer rate $\rho_t$, the preconditioner $P_t$, and the noise
covariance structure unchanged.

Every ingredient of the analysis depends on the step size only through
this structural form and the rate facts~\eqref{eq:tilde-eta-facts}, both
preserved.  The symmetrizer $\tilde Q_t$ of Definition~\ref{def:Qsym} is
obtained by the substitution $(\eta_t, \tau_t) \mapsto (\tilde\eta_t,
\tilde\tau_t)$, and the exact one-step
identity~\eqref{eq:Qsym-onestep}---algebraic in $L(P,\rho,\tau)$ with
$\tau = \rho/\eta$---holds verbatim, $(I - \tilde\eta_t \tilde L_t)^\top
\tilde Q_t (I - \tilde\eta_t \tilde L_t) = (1 - \rho_t)\tilde Q_t$.  Since
$\tilde\eta_t = \eta_t(1+o(1))$ and $\tilde\tau_t = \tau_t(1+o(1))$
(super-polynomially), the symmetrizer bounds and blockwise time-variation
estimates (Lemmas~\ref{lem:Qsym-bounds}, \ref{lem:Qsym-timevar}) keep
their rates, so the MSE bound of Proposition~\ref{prop:augmented-mse} is
intact.

The exact leading coefficient is likewise unchanged: applying
Lemma~\ref{lem:Linv} to $\tilde L_t = L(P_t, \rho_t, \tilde\tau_t)$ gives
$\tilde A_t \tilde B_t = \tilde L_t^{-1}\tilde B_t/\tilde\eta_t =
(-H^{-1}, 0)^\top$, exactly as in~\eqref{eq:AtBt} and by the same
cancellation $\rho_t\tilde\tau_t^{-1} = \tilde\eta_t$.  The remainder
bound (Lemma~\ref{lem:Rnz-bound}) uses the augmented mean-square
rates, the buffer recursion for $m_t$ (untouched by the first-moment
correction), and this exact identity---all preserved---so the CLT proof
of Appendix~\ref{app:clt-proof} and the projection identity
(Theorem~\ref{thm:projection}) carry over.  Hence the effective-step
reduction $\eta_t \mapsto \tilde\eta_t$ leaves every ingredient of the
analysis intact, so the bias-corrected SA-Adam
recursion~\eqref{eq:sa-adam} may be analyzed as the un-bias-corrected
recursion with effective step $\tilde\eta_t$---this is Step~0 of the main
theorem's proof (Section~\ref{sec:main-theorem} below).

\subsection{Main Theorem}
\label{sec:main-theorem}

The main theorem is where the pieces assemble.  Combining the preconditioner
verification (Proposition~\ref{prop:sa-adam-stab}) and the bias-correction
reduction (Section~\ref{sec:bias-correction}) with the general augmented-state
results of Sections~\ref{sec:augmented}--\ref{sec:asymptotics} yields the
paper's main result: the Polyak--Ruppert average of SA-Adam is
$\sqrt n$-asymptotically normal at the efficient sandwich covariance
$H^{-1}SH^{-1}$, underpinning one-pass Wald inference for SA-Adam.  Its
proof carries out no new analysis---it assembles these three
ingredients: Sections~\ref{sec:augmented}--\ref{sec:asymptotics} supply the
Gaussian limit, while Sections~\ref{sec:sa-adam-verify}
and~\ref{sec:bias-correction} discharge its two algorithm-specific hypotheses.

\begin{theorem}[SA-Adam Polyak--Ruppert CLT]
\label{thm:main}
Assume the standing Assumptions~\ref{ass:mg}--\ref{ass:quadratic},
bounded stochastic gradients $\|g_t\| \le G$ a.s., the fourth-moment
stability of Assumption~\ref{ass:fourth}, the conditional-covariance
continuity at $x^*$ of Theorem~\ref{thm:augmented-clt}(a)---$\mathbb
E[\xi_t\xi_t^\top \mid \mathcal F_{t-1}] = S(x_t)$ for a deterministic map
$S(\cdot)$ with $S(x) \to S$ as $x \to x^*$---and the SA-Adam
recursion~\eqref{eq:sa-adam} with momentum schedule
$\rho_t = c_1/t^\gamma$, $c_1 > 0$, $\gamma \in (\alpha, 1)$ (under the
convex-weight convention of Section~\ref{sec:sa-adam-defn}, i.e.\ the
shifted schedule when $c_1 \ge 1$), $c_2 \in (0, 1]$, and
$\eta_t = \eta_0 t^{-\alpha}$, $\alpha \in (1/2, 1)$ (the full
bias-corrected algorithm, with both $\hat m_t$ and $\hat v_t$
corrections).  Then:
\begin{enumerate}
\item[\textnormal{(i)}] (\emph{Augmented mean-square stability})
  $\mathbb E\|\Delta_t\|^2 = O(t^{-\alpha})$ and
  $\mathbb E\|m_{t-1}\|^2 = O(t^{-\gamma})$; in particular
  Assumption~\ref{ass:iterate} holds.
\item[\textnormal{(ii)}] (\emph{Asymptotic normality})
$\sqrt n\,(\overline x_n - x^*) \xrightarrow{d}
  \mathcal N(0, H^{-1}SH^{-1})$.
\end{enumerate}
\end{theorem}

\begin{proof}
\emph{Step 0: bias-correction reduction.}  By the effective-step-size
reduction of Section~\ref{sec:bias-correction}
(Lemma~\ref{lem:bias-effective-step}), the bias-corrected
recursion~\eqref{eq:sa-adam} is the un-bias-corrected centered
recursion~\eqref{eq:error-recursion} of
Sections~\ref{sec:augmented}--\ref{sec:asymptotics} with $\eta_t$
replaced by the effective step $\tilde\eta_t = \eta_t/\kappa^m_t$.  Both
Proposition~\ref{prop:augmented-mse} and
Theorem~\ref{thm:augmented-clt} use the step size only through the rate
properties
$\tilde\eta_t = \eta_0 t^{-\alpha}(1 + o(1))$,
$|\tilde\eta_{t+1} - \tilde\eta_t| = O(t^{-\alpha-1})$,
$\tilde\tau_t := \rho_t/\tilde\eta_t \asymp t^{\alpha-\gamma}$,
the exact identity $A_t B_t = (-H^{-1},0)^\top$, and---for the Lyapunov
bound of Proposition~\ref{prop:augmented-mse}---the symmetrizer one-step
identity $(I - \tilde\eta_t \tilde L_t)^\top \tilde Q_t (I - \tilde\eta_t
\tilde L_t) = (1-\rho_t)\tilde Q_t$ of Appendix~\ref{app:lyapunov}.  By
Lemma~\ref{lem:bias-effective-step}, $\tilde\eta_t$ shares all three rate
properties with $\eta_t$ (the corrections being $O(e^{-a t^{1-\gamma}})$);
and since both the $A_t B_t$ identity and the symmetrizer identity depend
on the step size only through the relation $\tilde\tau_t =
\rho_t/\tilde\eta_t$, they are preserved under $\eta_t \mapsto
\tilde\eta_t$ (verified blockwise in Section~\ref{sec:bias-correction}).
Hence every result invoked below holds verbatim for the bias-corrected
algorithm, and we argue on the analyzed recursion.

\emph{Step 1: preconditioner verification.}  The preconditioner $P_t =
\mathrm{Diag}(\hat v_{t-1})^{-1/2}$ is predictable
($\mathcal F_{t-1}$-measurable, via the lagged $\hat v_{t-1}$), and by
Proposition~\ref{prop:sa-adam-stab} satisfies the stabilization
condition (Assumption~\ref{ass:stab}, rate $\beta = 1 > (\alpha+1)/2$)
and the standing preconditioner conditions~\eqref{eq:appA-standing}.
These are precisely the preconditioner hypotheses required by
Proposition~\ref{prop:augmented-mse} and
Theorem~\ref{thm:augmented-clt}.

\emph{Part (i): iterate stability.}  With Step~1 in hand,
Proposition~\ref{prop:augmented-mse} applies and yields the
two-time-scale split~\eqref{eq:prop-mse-split}; in particular
$\mathbb E\|\Delta_t\|^2 = O(t^{-\alpha})$ (so
Assumption~\ref{ass:iterate} holds) and $\mathbb E\|m_{t-1}\|^2 =
O(t^{-\gamma})$.

\emph{Part (ii): asymptotic normality.}  The remaining hypotheses of
Theorem~\ref{thm:augmented-clt} are now in place: the augmented
mean-square bounds (hypothesis~(b),~\eqref{eq:clt-stability}) are
Part~(i); the conditional-covariance continuity (hypothesis~(a)) and the
schedule $\rho_t = c_1/t^\gamma$, $\gamma\in(\alpha,1)$, hold by
assumption; and Assumption~\ref{ass:stab} was verified in Step~1.
Therefore Theorem~\ref{thm:augmented-clt} yields $\sqrt n\,\overline z_n
\xrightarrow{d} \mathcal N\bigl(0,\, \mathrm{diag}(H^{-1}SH^{-1},
0)\bigr)$, whose $(1,1)$ block is the projection identity of
Theorem~\ref{thm:projection}.  Since the first $d$ coordinates of
$\overline z_n$ are $\overline\Delta_n = \overline x_n - x^*$, the
iterate marginal gives $\sqrt n\,(\overline x_n - x^*) \xrightarrow{d}
\mathcal N(0, H^{-1}SH^{-1})$.
\end{proof}

\begin{remark}[Scope: a local, conditional theorem]
\label{rem:scope}
Theorem~\ref{thm:main} is a \emph{local} asymptotic-normality result,
conditional on the trajectory confinement of
Assumption~\ref{ass:quadratic} and the fourth-moment stability of
Assumption~\ref{ass:fourth}---the standard hypothesis stack for
Polyak--Ruppert-type CLTs.  It is not a global convergence theorem for
SA-Adam: we do not establish that it reaches the neighborhood
$\mathcal N$ of $x^*$ from an arbitrary start, nor do we relax the
bounded-gradient and confinement requirements.  The contribution is the
paired-drift mechanism---an adaptive preconditioner not assumed to
converge, coupled to a time-varying momentum buffer, leaves the
iterate-marginal sandwich $H^{-1}SH^{-1}$ intact, \emph{granted} that the
iterates stabilize.
\end{remark}

As a downstream consequence, the CLT of Theorem~\ref{thm:main} supplies
the asymptotic distribution needed for Wald-type online inference on
$x^*$.  By Slutsky's theorem, valid marginal confidence intervals follow
once the limiting covariance $H^{-1}SH^{-1}$ is consistently estimated,
or a pivotal statistic is formed: the online covariance estimators of
\citet{chen2020statistical} (plug-in / batch-means) and
\citet{zhu2023online} (fully online), and the random-scaling procedure of
\citet{lee2021fast} (which avoids estimating the asymptotic variance),
all developed for unpreconditioned averaged SGD, can be combined with
Theorem~\ref{thm:main} once their estimator-specific consistency or
pivotal conditions are verified in the SA-Adam setting---the same
downstream step taken for the SA-type estimators of \citet{anhuo2026}.

\subsection{Side Variants}
\label{sec:variants}

The augmented-state framework accommodates several Adam-family variants
with minor modifications; we treat three in turn---SA-AMSGrad, coupled
($L_2$) weight decay, and the full-matrix SA-Adam-full.  For SA-AMSGrad
and SA-Adam-full only the preconditioner verification of
Section~\ref{sec:sa-adam-verify} need be re-established; coupled weight
decay is SA-Adam on a ridge objective.  The main theorem then applies,
with the preconditioner-independent projection identity
(Theorem~\ref{thm:projection}) fixing the sandwich form of the limit.

\citet{reddi2018convergence} introduced AMSGrad as a fix for a
non-convergence pathology in Adam, replacing $\hat v_t$ in the update
by $\bar v_t := \max(\bar v_{t-1}, \hat v_t)$, taken coordinatewise.
In the SA-Adam framework the preconditioner uses the
\emph{one-step-lagged} running maximum, $P_t = \mathrm{Diag}(\bar
v_{t-1})^{-1/2}$ with $\bar v_{t-1} = \max(\bar v_{t-2}, \hat v_{t-1})$,
so that $P_t$ remains $\mathcal F_{t-1}$-measurable (whereas canonical
AMSGrad uses the current $\bar v_t$); the resulting SA-AMSGrad update
preserves the $O(t^{-1})$ stabilization rate on $M_t$:

\begin{proposition}[SA-AMSGrad stabilization]
\label{prop:amsgrad-stab}
Consider the SA-AMSGrad update---SA-Adam~\eqref{eq:sa-adam} with
$\hat v_{t-1}$ replaced by $\bar v_{t-1} := \max(\bar v_{t-2},
\hat v_{t-1})$ coordinatewise (base case $\bar v_0 := \hat v_0 = v_0$)
---under bounded stochastic gradients $\|g_t\| \le G$ a.s.,
$c_2 \in (0,1]$, and $v_0 = \epsilon\mathbf 1$ with $\epsilon > 0$.  Then
$\|M_t - M_{t-1}\|_{\mathrm{op}} \le C\, t^{-1}$ a.s.\ for all $t \ge 2$,
and SA-AMSGrad satisfies the stabilization condition
(Assumption~\ref{ass:stab}, $\beta = 1$) and the standing
conditions~\eqref{eq:appA-standing}.  Consequently, under the full
hypotheses of Theorem~\ref{thm:main}---the same step-size and momentum
schedules, bounded gradients, and conditional-covariance continuity at
$x^*$, with the trajectory confinement (Assumption~\ref{ass:quadratic})
and fourth-moment stability (Assumption~\ref{ass:fourth}) evaluated along
the SA-AMSGrad trajectory---the conclusion of Theorem~\ref{thm:main}
holds with the same sandwich limit $H^{-1}SH^{-1}$.
\end{proposition}

\begin{proof}
Take the max coordinatewise with base case $\bar v_0 := \hat v_0 =
\epsilon\mathbf 1$, so that $\bar v_{t-1,i} = \max_{0 \le s \le t-1}
\hat v_{s,i}$ (the base term $\hat v_0 = \epsilon\mathbf 1$ is included).
We show $\bar v_{t-1}$ inherits the two properties of $\hat v_{t-1}$ used
in Proposition~\ref{prop:sa-adam-stab}: a uniform two-sided bound and an
$O(t^{-1})$ increment.

\emph{Uniform bounds.}  Each $\hat v_{s,i} \in [\epsilon, c_G]$
(Proposition~\ref{prop:sa-adam-stab}), so the coordinatewise maximum
satisfies $\bar v_{t-1,i} \in [\epsilon, c_G]$ as well; hence the
preconditioner $P_t = \mathrm{Diag}(\bar v_{t-1})^{-1/2}$ obeys the same
uniform ellipticity, $c_G^{-1/2} I \preceq P_t \preceq \epsilon^{-1/2}
I$.

\emph{Increment.}  Fix a coordinate $i$.  Since $\bar v_{t-2,i} \ge
\hat v_{t-2,i}$ and $x \mapsto (x)_+$ is monotone,
$0 \le \bar v_{t-1,i} - \bar v_{t-2,i}
  = \bigl(\hat v_{t-1,i} - \bar v_{t-2,i}\bigr)_+
  \le \bigl(\hat v_{t-1,i} - \hat v_{t-2,i}\bigr)_+
  \le |\hat v_{t-1,i} - \hat v_{t-2,i}|$,
so $\|\bar v_{t-1} - \bar v_{t-2}\|_\infty \le \|\hat v_{t-1} - \hat
v_{t-2}\|_\infty = O(t^{-1})$ for all $t \ge 2$, by
Proposition~\ref{prop:sa-adam-stab}.  Both $x \mapsto x^{-1/2}$ and
$x \mapsto x^{1/2}$ are Lipschitz on $[\epsilon, \infty)$ (constants
$\tfrac12\epsilon^{-3/2}$ and $\tfrac12\epsilon^{-1/2}$), so the diagonal
maps $P_t = \mathrm{Diag}(\bar v_{t-1})^{-1/2}$ and $M_t = H^{-1}P_t^{-1}
= H^{-1}\mathrm{Diag}(\bar v_{t-1})^{1/2}$ carry this rate to
$\|P_t - P_{t-1}\|_{\mathrm{op}} = O(t^{-1})$ and
$\|M_t - M_{t-1}\|_{\mathrm{op}} = O(t^{-1})$ for $t \ge 2$,
with $\sup_t\|M_t\|_{\mathrm{op}} \le \|H^{-1}\|_{\mathrm{op}}\,c_G^{1/2}
< \infty$.  Together with the uniform ellipticity above, this is exactly
Assumption~\ref{ass:stab} ($\beta = 1$) and the standing
conditions~\eqref{eq:appA-standing}.  The conclusion of
Theorem~\ref{thm:main} then follows from its proof applied to the
SA-AMSGrad trajectory---the sandwich limit $H^{-1}SH^{-1}$ being
unchanged because the projection identity (Theorem~\ref{thm:projection})
is preconditioner-independent.

Finally, since $\bar v_{t-1}$ is coordinatewise nondecreasing and
bounded above by $c_G$, it converges; SA-AMSGrad thus has a
\emph{convergent} preconditioner, so the rate-only stabilization
established here is more than strictly needed---but it places SA-AMSGrad
in the same framework as the non-convergent variants.
\end{proof}

Weight decay for SA-Adam comes in two forms, only
one of which the paired-drift framework covers.
\citet{loshchilov2017decoupled} introduced AdamW with \emph{decoupled}
weight decay---the update appends $-\eta_t \lambda x_t$ for $\lambda > 0$
\emph{outside} the preconditioner---and this genuine AdamW is \emph{not}
covered here, because the unpreconditioned shrinkage breaks the $P_t H$
paired-drift structure (Remark~\ref{rem:adamw}).  What \emph{is} covered
is the \emph{coupled} ($L_2$-regularized) form, which folds the decay
into the gradient so that the preconditioned buffer sees the regularized
gradient: by Corollary~\ref{cor:adamw-coupled} this is SA-Adam applied to
the ridge objective $F_\lambda$, delivering $\sqrt n$-asymptotic
normality at the penalized minimizer $x^*_\lambda$ with the regularized
sandwich covariance $H_\lambda^{-1}S_\lambda H_\lambda^{-1}$
($H_\lambda = \nabla^2 F(x^*_\lambda) + \lambda I$,
$S_\lambda = S(x^*_\lambda)$), the canonical object for penalized
$M$-estimation.

\begin{corollary}[SA-Adam with coupled weight decay]
\label{cor:adamw-coupled}
Consider SA-Adam with \emph{coupled} weight decay at rate $\lambda > 0$,
in which the regularized gradient $g_t + \lambda x_t$ replaces $g_t$ in
the moment recursions~\eqref{eq:sa-adam} (so both the momentum buffer
$m_t$ and the second-moment buffer $v_t$ see $g_t + \lambda x_t$).  Let
$F_\lambda(x) := F(x) + \tfrac\lambda2\|x\|^2$ be the ridge-penalized
objective, with minimizer $x^*_\lambda := \arg\min_x F_\lambda(x)$ and
Hessian \emph{at $x^*_\lambda$} $H_\lambda := \nabla^2 F_\lambda(x^*_\lambda)
= \nabla^2 F(x^*_\lambda) + \lambda I \succ 0$ (which equals $H + \lambda I$
only when $F$ is quadratic, since then $\nabla^2 F \equiv H$); write
$S_\lambda := S(x^*_\lambda)$ for the gradient-noise covariance at
$x^*_\lambda$.  Then, under the hypotheses of Theorem~\ref{thm:main}
applied to $F_\lambda$,
\[
  \sqrt n\,(\overline x_n - x^*_\lambda) \xrightarrow{d}
  \mathcal N\!\bigl(0,\ H_\lambda^{-1} S_\lambda H_\lambda^{-1}\bigr).
\]
\end{corollary}

\begin{proof}
Coupling folds the decay into the gradient,
$g_t \mapsto g_t + \lambda x_t = \nabla f(x_t,\zeta_t) + \lambda x_t$,
which is the stochastic gradient of $F_\lambda$ at $x_t$.  Its
conditional mean is $\nabla F_\lambda(x_t) = \nabla F(x_t) + \lambda x_t$,
with Hessian $\nabla^2 F_\lambda(x^*_\lambda) = \nabla^2 F(x^*_\lambda) +
\lambda I = H_\lambda$ at $x^*_\lambda$; the decay term is
$\mathcal F_{t-1}$-measurable, so the conditional noise covariance is
unchanged,
$\mathrm{Cov}(g_t + \lambda x_t \mid \mathcal F_{t-1})
 = \mathrm{Cov}(g_t \mid \mathcal F_{t-1}) = S(x_t) \to S(x^*_\lambda) =
S_\lambda$.  The iterate-block drift accordingly becomes
$\rho_t P_t H_\lambda$, and the centered
recursion~\eqref{eq:error-recursion} holds with $(H, x^*)$ replaced by
$(H_\lambda, x^*_\lambda)$ and the same noise $\xi_t$.  The hypotheses of
Theorem~\ref{thm:main} \emph{as applied to $F_\lambda$}---trajectory
confinement, fourth-moment stability, and conditional-covariance
continuity, now at $x^*_\lambda$, together with $H_\lambda \succeq
(\mu+\lambda)I \succ 0$ (Assumption~\ref{ass:convex} plus the
$\lambda$-strong convexity of the penalty)---are assumed; the
preconditioner verification of Proposition~\ref{prop:sa-adam-stab} carries
over because it uses only boundedness of the buffered gradient
$g_t + \lambda x_t$ on the confinement region.  The theorem's conclusion
then gives the stated limit, with sandwich form $H_\lambda^{-1} S_\lambda
H_\lambda^{-1}$ because the projection identity of
Theorem~\ref{thm:projection} is preconditioner-independent.
\end{proof}

\begin{remark}[Decoupled weight decay is not covered]
\label{rem:adamw}
\emph{Decoupled} decay (genuine AdamW) is different: the unpreconditioned
term $-\eta_t\lambda x_t$ enters the iterate-block drift \emph{outside}
the preconditioner---schematically $\rho_t P_t H + \lambda I$---rather
than \emph{through} it as in the coupled drift $\rho_t P_t H_\lambda$.
Two consequences follow.  The limiting fixed point becomes
preconditioner-dependent: even if $P_t \to P_\infty$, it solves
$P_\infty \nabla F(x_\infty) + \lambda x_\infty = 0$, not in general
$x^*_\lambda$.  And the $P_tH$ structure behind the projection identity
breaks: for a frozen $P$ the leading iterate coefficient is
$-(H + \lambda P^{-1})^{-1}$, not $-H_\lambda^{-1}$ (agreeing only at
$P = I$), so the limit depends on $P$ and is not in general
$H_\lambda^{-1}SH_\lambda^{-1}$.  A correct treatment under a
non-convergent preconditioner is left to future work.
\end{remark}

Finally, \emph{SA-Adam-full}, the full-matrix variant we define here (not a
standard deployed optimizer), replaces the diagonal second-moment buffer
of SA-Adam by the symmetric positive-definite matrix
\[
  C_t = (1 - \rho^v_t)\,C_{t-1} + \rho^v_t\,(g_t g_t^\top + \epsilon I),
  \qquad \rho^v_t = c_2/t,\quad C_0 = \epsilon I,
\]
with bias correction $\hat C_t = C_t/\kappa^v_t$ and preconditioner
$P_t = \hat C_{t-1}^{-1/2}$ (the spectral map $A \mapsto A^{-1/2}$).
This mirrors the full-matrix AdaGrad variant of
\citet{anhuo2026}, now driving a momentum buffer.

\begin{proposition}[SA-Adam-full CLT]
\label{prop:full-matrix}
Consider the SA-Adam-full update---SA-Adam~\eqref{eq:sa-adam} with the
diagonal second moment replaced by the full-matrix $C_t$ above and
$P_t = \hat C_{t-1}^{-1/2}$---under bounded stochastic gradients
$\|g_t\| \le G$ a.s., $c_2 \in (0,1]$, $C_0 = \epsilon I$ with
$\epsilon > 0$, the schedules $\rho_t = c_1/t^\gamma$ ($c_1 > 0$,
$\gamma\in(\alpha,1)$; under the convex-weight convention of
Section~\ref{sec:sa-adam-defn}, shifted when $c_1 \ge 1$) and
$\eta_t = \eta_0 t^{-\alpha}$ ($\alpha\in(1/2,1)$) of
Theorem~\ref{thm:main}, and the same
Assumptions~\ref{ass:mg}--\ref{ass:quadratic} and~\ref{ass:fourth} and
conditional-covariance continuity (Theorem~\ref{thm:augmented-clt}(a))
evaluated along the SA-Adam-full trajectory.  Then the conclusion of
Theorem~\ref{thm:main} holds, with the same sandwich limit
$H^{-1}SH^{-1}$.
\end{proposition}

\begin{proof}
Two facts combine; only the second is specific to this paper.

\emph{Stabilization.}  Bounded
gradients give $g_t g_t^\top + \epsilon I \in [\epsilon I,
(G^2+\epsilon)I]$, so by induction $\epsilon I \preceq C_t \preceq
(G^2+\epsilon) I$ a.s., with one-step variation
$\|C_t - C_{t-1}\|_{\mathrm{op}} = \rho^v_t\,\|g_t g_t^\top + \epsilon I
- C_{t-1}\|_{\mathrm{op}} = O(t^{-1})$; the scalar bias factor
$\kappa^v_t \in [c_2, 1]$ keeps $\hat C_t = C_t/\kappa^v_t \in
[\epsilon I, c_G I]$ with $c_G := (G^2+\epsilon)/c_2$, and contributes
only a subdominant $O(t^{-c_2-1})$ term to the increment, exactly as in
Proposition~\ref{prop:sa-adam-stab}.  The square-root and its inverse are
\emph{operator-Lipschitz} on $[\epsilon I, c_G I]$---which the matrix case
requires in place of the scalar Lipschitz bound: for $A, B \succeq
\epsilon I$ (so $A^{1/2}, B^{1/2} \succeq \sqrt\epsilon\,I$), the Sylvester
identity $A - B = A^{1/2}(A^{1/2} - B^{1/2}) + (A^{1/2} - B^{1/2})B^{1/2}$,
together with the coercivity of the map $X \mapsto A^{1/2}X + XB^{1/2}$
(whose smallest singular value is $\ge 2\sqrt\epsilon$), gives
$\|A^{1/2} - B^{1/2}\|_{\mathrm{op}} \le \|A - B\|_{\mathrm{op}}/(2\sqrt\epsilon)$, whence
$\|A^{-1/2} - B^{-1/2}\|_{\mathrm{op}} \le \epsilon^{-1}\|A^{1/2} -
B^{1/2}\|_{\mathrm{op}} = O(\|A-B\|_{\mathrm{op}})$
(cf.\ \citet{hornjohnson2013}).  Applied to $\hat C_{t-1}$, this
carries the $O(t^{-1})$ rate to $\|P_t - P_{t-1}\|_{\mathrm{op}}$ and to
$\|M_t - M_{t-1}\|_{\mathrm{op}}$, with $\sup_t\|M_t\|_{\mathrm{op}} \le
\|H^{-1}\|_{\mathrm{op}}\,c_G^{1/2} < \infty$, for
$M_t = (P_t H)^{-1} = H^{-1}\hat C_{t-1}^{1/2}$.  This gives
Assumption~\ref{ass:stab} ($\beta = 1$) and the standing
conditions~\eqref{eq:appA-standing} (the spectral-map argument as in
\citet{anhuo2026}).

\emph{CLT extension (immediate from our framework).}  The
augmented-state analysis of
Sections~\ref{sec:augmented}--\ref{sec:asymptotics}---the spectrum of
$L_t$, the symmetrizer, the block inverse (Lemma~\ref{lem:Linv}), and the
projection identity (Theorem~\ref{thm:projection})---uses only that
$P_t$ is symmetric positive-definite and obeys the standing conditions;
diagonality is never invoked.  Hence the \emph{proof} of
Theorem~\ref{thm:main} applies verbatim once the predictable $P_t = \hat C_{t-1}^{-1/2}$
is substituted, giving the same sandwich limit $H^{-1}SH^{-1}$
(preconditioner-independent, by Theorem~\ref{thm:projection}).
\end{proof}

\section{Simulation Study}
\label{sec:simulation}

We report three numerical experiments validating the paper's main
results, with diagnostics in the spirit of the numerical study of
\citet{anhuo2026}.  Section~\ref{sec:exp-projection} verifies the
projection identity and momentum-invisibility distributionally, in the
\emph{same} streaming-regression environment used by that paper but with
SA-Adam momentum added; Section~\ref{sec:exp-gamma} establishes the
necessity of the sub-linear momentum exponent $\gamma<1$ by exact
evaluation of the scalar limiting variance; and
Section~\ref{sec:exp-coverage} confirms, on a semi-synthetic design with
real covariates, that the resulting one-pass Wald confidence statements
attain nominal coverage, with SA-Adam statistically indistinguishable from
plain SGD.  Apart from Section~\ref{sec:exp-gamma}'s comparisons outside $(\alpha,1)$, each
experiment uses a schedule in the admissible range
$\tfrac12<\alpha<\gamma<1$ suited to its diagnostic.  Covariance-based
diagnostics use the \emph{oracle} sandwich $H^{-1}SH^{-1}$, isolating the
distributional claim from plug-in covariance-estimation error.

\subsection{Projection Identity and Momentum-Invisibility}
\label{sec:exp-projection}

\begin{figure}[tb]
\centering
\includegraphics[width=\textwidth]{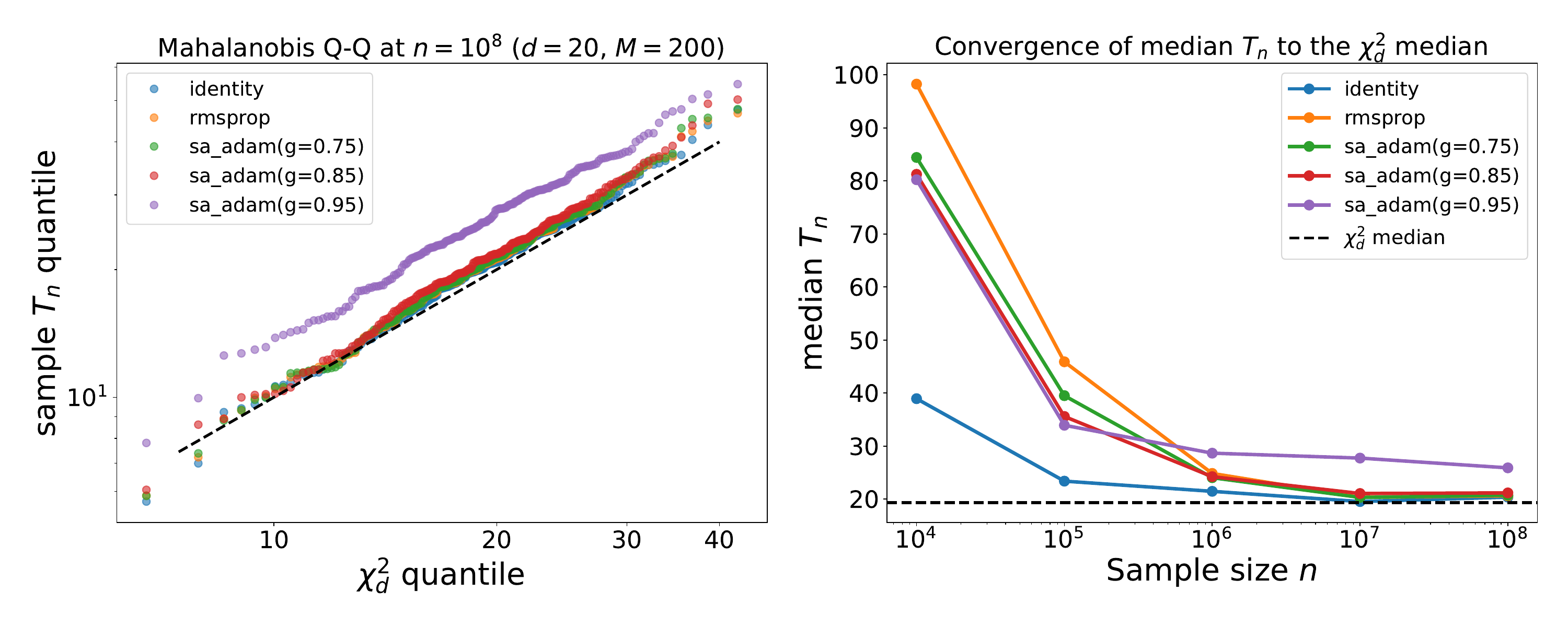}
\caption{Projection identity and momentum-invisibility (streaming Toeplitz
regression, $S\neq H$; $M=200$).  Left: Mahalanobis Q-Q against $\chi^2_d$ at
$n=10^8$ for plain SGD, SA-RMSProp, and SA-Adam across the $\gamma$-sweep;
points on $y=x$ indicate ideal distributional agreement.  Right: the median
Mahalanobis $T_n$ versus $n$ converges to the $\chi^2_d$ median (dashed) for
every arm---plain SGD, SA-RMSProp, and SA-Adam at $\gamma\in\{0.75,0.85\}$
nearly coincide, illustrating momentum-invisibility, while the near-boundary
$\gamma=0.95$ approaches more slowly.}
\label{fig:exp-projection}
\end{figure}

We first verify the projection identity and momentum-invisibility
distributionally, in the streaming linear-regression design of
\citet{anhuo2026}: Gaussian covariates $a_t\sim\mathcal N(0,H)$ on
the Toeplitz Hessian $H_{jk}=0.4^{|j-k|}$ ($d=20$, $\kappa(H)\approx5.3$),
responses $y_t=a_t^\top x^*+\varepsilon_t$ with heteroskedastic noise
$\varepsilon_t\mid a_t\sim\mathcal N(0,\sigma_0^2+\sigma_1^2(a_t^\top u)^2)$
($\sigma_0=0.35$, $\sigma_1=0.8$), so that $S\neq H$.  Each step consumes one
fresh sample; the step size is $\eta_t=0.2\,t^{-0.7}$ and the second-moment
gain is $\rho^v_t=1/(t+1)$ with stabilization ridge $0.5$.  We compare plain
Polyak--Ruppert SGD, SA-RMSProp (preconditioner, no momentum), and SA-Adam
across a sweep $\gamma\in(\alpha,1)$, all driven by a \emph{shared} gradient
stream.  The diagnostic is the Mahalanobis statistic
$T_n=n\,(\overline x_n-x^*)^\top(H^{-1}SH^{-1})^{-1}(\overline x_n-x^*)$, predicted by
Theorems~\ref{thm:projection} and~\ref{thm:main} to converge to $\chi^2_d$;
we report its median, its one-sample Kolmogorov--Smirnov distance $D_M$ from
$\chi^2_d$ relative to the Monte Carlo (MC) level $1.36/\sqrt M$, and---measuring
momentum-invisibility directly---the paired quantity
$n\,\|\overline x_n^{\mathrm{adam}}-\overline x_n^{\mathrm{sgd}}\|^2_{(H^{-1}SH^{-1})^{-1}}$.

This unbounded-covariate design is large-$n$: as in
\citet{anhuo2026}, the methods reach the $\chi^2_d$ regime only
near $n=10^7$--$10^8$, so Figure~\ref{fig:exp-projection} runs to $n=10^8$ over
$M=200$ replications.  There the median $T_n$ has descended to the
$\chi^2_{20}$ regime (median $19.3$) for plain SGD ($20.4$), SA-RMSProp
($20.6$), and SA-Adam at $\gamma=0.75$ ($20.7$) and $\gamma=0.85$ ($21.2$),
with $D_M$ near or modestly above the reference $1.36/\sqrt{200}=0.096$ (ratios
$0.86,0.99,1.04,1.44$) and normalized MSE (NMSE) in $[1.04,1.09]$; the momentum arms are
nearly indistinguishable from the no-momentum baseline.  Invisibility is
confirmed directly: the paired statistic is $0.15$--$0.51$ for
$\gamma\in\{0.75,0.85\}$, about $1\%$ of the sandwich trace $43.1$.  The
near-boundary $\gamma=0.95$ stays elevated ($T_n=25.9$ at $n=10^8$), exactly as
the necessity analysis of Section~\ref{sec:exp-gamma} predicts---the limit is
fixed for every $\gamma\in(\alpha,1)$, but the rate of approach slows as
$\gamma\to1$.  Finally, the $P$-independence at the heart of the identity is
checked exactly: over $10^3$ random triples $(P,\rho,\tau)$ with $P$ not commuting with $H$, the
iterate block of $\Sigma_z=L^{-1}\Sigma_w L^{-\top}$ matches $H^{-1}SH^{-1}$ to
relative error $1.6\times10^{-15}$.

\subsection{Necessity of Sub-Linear Momentum}
\label{sec:exp-gamma}

\begin{figure}[tb]
\centering
\includegraphics[width=\textwidth]{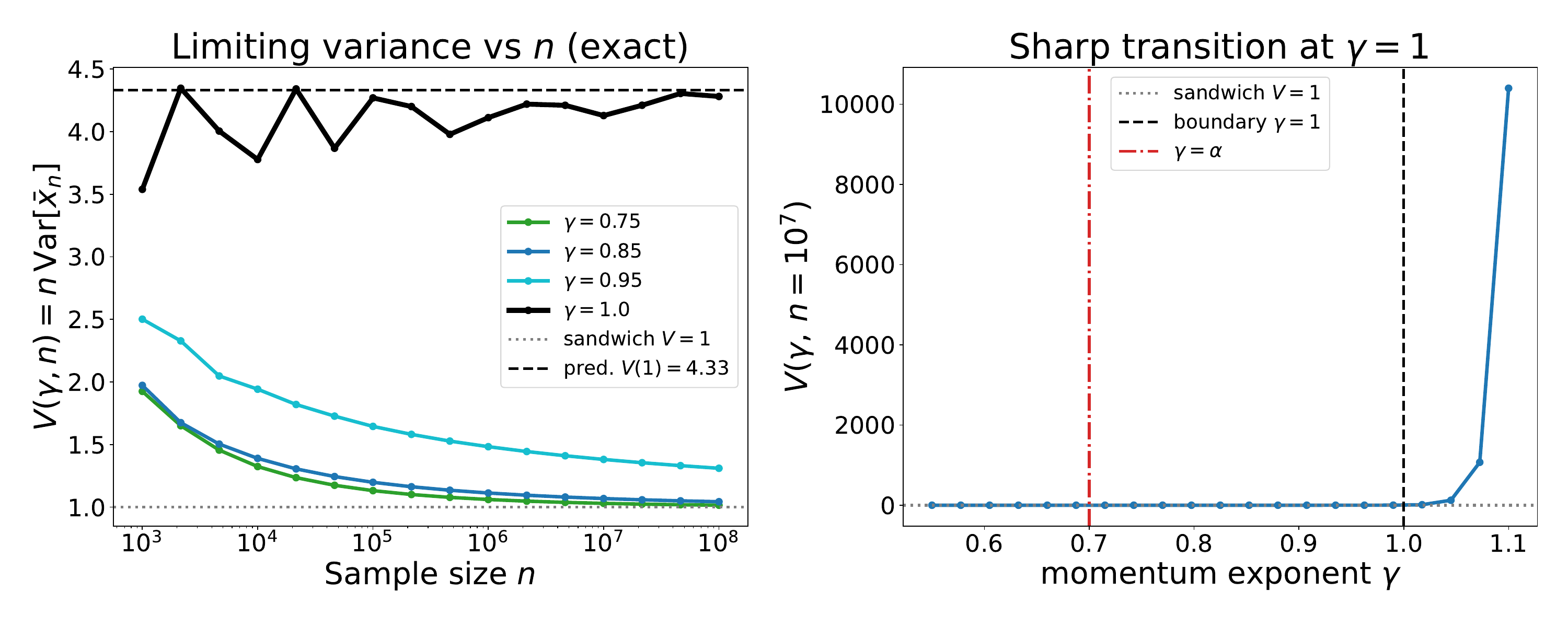}
\caption{Necessity of $\gamma<1$ (exact scalar evaluation).  Left:
$V(\gamma,n)=n\,\mathrm{Var}[\overline x_n]$ versus $n$ for several $\gamma$;
curves with $\gamma<1$ descend toward the sandwich $V=1$ (more slowly as
$\gamma\uparrow1$), while $\gamma=1$ plateaus at the predicted inflated
limit $1+1/(2c_1-1-\alpha)$.  Right: $V(\gamma,10^7)$, showing the
sharp transition at $\gamma=1$.}
\label{fig:exp-gamma}
\end{figure}

The sub-linear exponent $\gamma<1$ is essential (Lemma~\ref{lem:Rnz-bound},
Remark~\ref{rem:Rnz-gamma-necessary}): at the canonical value $\gamma=1$ the
endpoint buffer leaves a persistent $\Theta(n^{-1/2})$ residual that
inflates the iterate-marginal covariance above the sandwich.  We illustrate
this by an exact, deterministic evaluation, in the scalar model
$H=\sigma=1$, of $V(\gamma,n):=n\,\mathrm{Var}[\overline x_n]$ and its limit
$V(\gamma):=\lim_n V(\gamma,n)$, obtained by propagating the closed
$3\times3$ covariance recursion of the augmented state
$(\Delta_t,m_{t-1},\sum_{s\le t}\Delta_s)$ (no MC error).

For $\gamma\in(\alpha,1)$ the variance converges to the sandwich
$V(\gamma)=1$ at the predicted rate $n^{\gamma-1}$.  Over
the sweep $\gamma\in\{0.75,0.85,0.95\}$ ($\alpha=0.7$, $c_1=1$), the well-inside
exponents are essentially on the sandwich ($V(0.75,10^8)=1.02$, $V(0.85,10^8)=1.04$),
while the near-boundary $V(0.95,10^8)=1.31$ is still visibly descending, since the
rate $n^{\gamma-1}$ degrades as $\gamma\uparrow1$.  At the boundary $\gamma=1$
instead $V(1)=1+1/(2c_1-1-\alpha)=4.33$ (with $V(1,10^8)=4.28$),
and for $\gamma>1$ it diverges.  Figure~\ref{fig:exp-gamma} shows the
resulting sharp transition at $\gamma=1$, matching the closed form of
Remark~\ref{rem:Rnz-gamma-necessary}.

\subsection{Semi-Synthetic Coverage with Real Covariates}
\label{sec:exp-coverage}

\begin{figure}[tb]
\centering
\includegraphics[width=\textwidth]{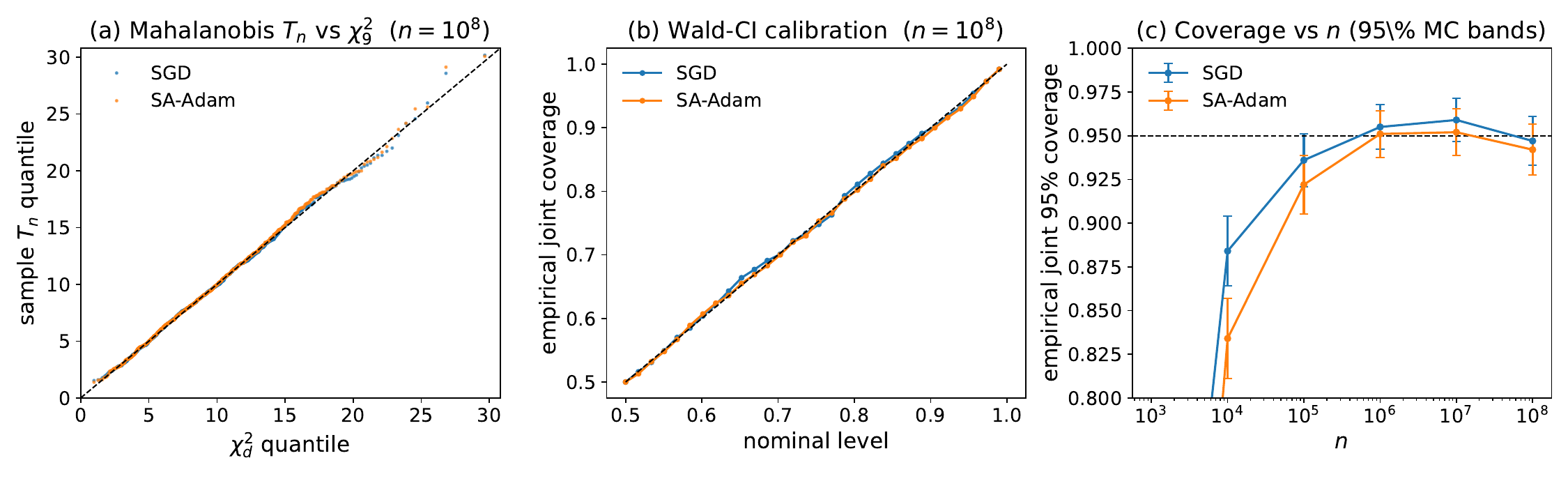}
\caption{Semi-synthetic coverage of averaged SA-Adam vs.\ averaged SGD (real
\texttt{diabetes} covariates, known heteroskedastic response; $d=9$, $M=10^3$,
oracle $\Sigma$).  (a)~The Mahalanobis statistic $T_n$ at $n=10^8$ aligns with the
$\chi^2_d$ line for both methods, whose curves nearly coincide.  (b)~Empirical joint coverage
tracks the nominal level.  (c)~Joint $95\%$ coverage rises to nominal by
$n\approx10^6$ and then fluctuates within MC error (bars: $95\%$ MC)
about it, indistinguishably for SA-Adam and SGD---exhibiting both the asymptotic
validity and the invisibility of preconditioning and momentum to the
averaged-iterate covariance.}
\label{fig:exp-coverage}
\end{figure}

To confirm that the identity is not an artifact of the Gaussian design above
and to exhibit the inferential payoff directly, we run a \emph{semi-synthetic}
coverage check on real feature geometry.  A coverage claim requires a known
ground truth, so we retain real covariates but simulate the response: the
standardized \texttt{diabetes} design (\citealp{efron2004least}; \texttt{scikit-learn},
\citealp{pedregosa2011scikit}; $442$ rows; the
two near-collinear serum features dropped and an intercept added, so $\kappa(H)\approx7.5$ and $d=9$),
with the empirical row distribution as the streaming population and
heteroskedastic responses $y=a^\top\theta^\star+\varepsilon$,
$\operatorname{sd}(\varepsilon\mid a)=0.5+0.8\,|a^\top w|$, so that
$\theta^\star,H,S,\Sigma=H^{-1}SH^{-1}$ are all known exactly with $S\neq H$.
Over $M=10^3$ independent streams of length $n=10^8$ we compare averaged
SA-Adam ($\alpha=0.6$, $\gamma=0.75$) with averaged SGD, scoring the
oracle-$\Sigma$ Mahalanobis statistic $T_n$ and the Wald confidence
ellipsoid/intervals.  At $n=10^8$ both methods are essentially unbiased
(empirical bias norm $\approx1.3\times10^{-5}$), the scaled
empirical covariance matches $\Sigma$ to $\sim8\%$ (the sampling error of a
$9\times9$ covariance from $10^3$ draws), and the median $T_n$ is $8.34$ for
both---matching the $\chi^2_9$ median.  Joint $95\%$ coverage is $0.942$
(SA-Adam) and $0.947$ (SGD) and mean marginal coverage is $0.951$ and $0.950$, all
within MC error (standard error, $\mathrm{SE}\approx0.007$) of nominal; the two
methods are statistically indistinguishable in the asymptotic regime
(Figure~\ref{fig:exp-coverage}), exactly as the projection identity predicts.

\section{Conclusion}
\label{sec:discussion}

We have extended the asymptotic-normality theory for averaged adaptive
stochastic gradient descent to the Adam family, which combines adaptive
preconditioning with momentum.  Because the pathwise decomposition of
\citet{anhuo2026} provably does not extend to momentum updates
(Proposition~\ref{prop:tautology}), we developed an augmented-state
framework resting on three results: the positive-stable spectrum of the
joint drift $L_t$ (Lemma~\ref{lem:hurwitz}); a non-autonomous
Polyak--Ruppert central limit theorem (Theorem~\ref{thm:augmented-clt}),
built on the linear-SA framework of \citet{mou2020linear} and the
stabilization condition of \citet{anhuo2026}; and the iterate-marginal
projection identity $\Sigma_z^{(1,1)} = H^{-1}SH^{-1}$
(Theorem~\ref{thm:projection}).  The projection identity is the central algebraic
contribution, generalizing the asymptotic-equivalence finding of
\citet{liu2023acceleration} to adaptive preconditioning and time-varying
momentum simultaneously.

Applied to SA-Adam---the stochastic-approximation reparametrization with
$\beta_{1,t} = 1 - c_1/t^\gamma$ ($\gamma\in(\alpha,1)$) and
$\beta_{2,t} = 1 - c_2/t$---these results yield the standard
Polyak--Ruppert efficiency (Theorem~\ref{thm:main}): the adaptive
preconditioning and momentum designed to accelerate optimization leave the
averaged-iterate limit untouched, so SA-Adam is a principled engine for
one-pass online inference, providing the limiting law needed for streaming
Wald confidence statements at no cost in asymptotic variance.  The constant-EMA deployed form of
Adam, by contrast, is not expected to satisfy the rate-only stabilization condition on $P_t$
underpinning our theory; whether it nonetheless attains the same
asymptotic variance remains open, as do sharper rate conditions, richer
state augmentations (e.g.\ the second-moment buffer $v_t$), and a
structural explanation of the projection identity~\eqref{eq:projection}.

\appendix

\renewcommand{\theHsection}{appendix.\arabic{section}}
\renewcommand{\theHsubsection}{appendix.\arabic{section}.\arabic{subsection}}

\makeatletter
\renewcommand\@seccntformat[1]{%
  \csname the#1\endcsname\ifcsname appsecdot@#1\endcsname.\fi\quad}
\@namedef{appsecdot@section}{}
\makeatother

\phantomsection
\section*{Appendix}

\section{Lyapunov Construction for Proposition~\ref{prop:augmented-mse}}
\label{app:lyapunov}

We complete the proof of Proposition~\ref{prop:augmented-mse} by
exhibiting an explicit Lyapunov matrix for the joint drift $L_t$.  In
place of the integral solution of a continuous Lyapunov equation we use
a closed-form \emph{symmetrizer} that is jointly block-diagonalized
with $L_t$ under the scalar reduction of Lemma~\ref{lem:spectrum} and
obeys two exact algebraic identities.  Its advantage is not an
improvement in condition number---that of $Q_t$ in fact grows like
$\tau_t^{-1}$---but that it is matched to the \emph{discrete} dynamics:
the identities yield the exact one-step cancellation
$(I-\eta_tL_t)^\top Q_t(I-\eta_tL_t) = (1-\rho_t)Q_t$
(Lemma~\ref{lem:Qsym-decrease}), and its $t$-independent $O(1)$ lower
eigenvalue bound is enough to convert a Lyapunov bound into an ordinary
Euclidean MSE bound.

\paragraph{Relation to prior work.}
The augmented-state route to an iterate-MSE bound for momentum methods
is not new; our contribution is its extension to a data-driven,
non-convergent preconditioner.  The closest precedent is the stochastic
heavy-ball analysis of \citet{gadat2018stochastic}, who control the
position--velocity pair by a Lyapunov energy combining the
objective with a memory-rescaled velocity term \emph{and} a
gradient--velocity cross-term, and obtain non-asymptotic $L^2$ rates for
both constant and decaying memory; their best-rate gain threshold is analogous to our
two-time-scale separation requirement $\gamma > \alpha$ (with the $c_1$-dependent threshold trivial
under $\gamma < 1$).  Our symmetrizer $Q_t$, the diagonal congruence
$T_t = \mathrm{diag}(P_t^{1/2}, P_t^{-1/2})$ with its
$\tau_t^{-1}$-rescaled buffer block, and
Proposition~\ref{prop:augmented-mse} are the preconditioned,
time-varying counterparts; the off-diagonal block
$\tfrac12(\eta_t H P_t - I)$ of $Q_t$ plays the role of their
gradient--velocity coupling.  \citet{liu2023acceleration} likewise
control the averaged-SGDM second moment through a $2d$-dimensional matrix
recursion; the linear two-time-scale second-moment analyses
\citep{konda2004actor, kaledin2020finite}, and
the nonlinear two-time-scale averaging of \citet{mokkadem2006convergence},
are the closest stochastic-approximation counterparts.  None of these admits
an adaptive, non-convergent preconditioner $P_t$; the symmetrizer~\eqref{eq:Qsym}, whose
\emph{exact} identities~\eqref{eq:Qsym-id} and one-step
congruence~\eqref{eq:Qsym-onestep} hold for arbitrary positive-definite
$P_t$ with no commutativity hypothesis, is what carries the argument
through to the SA-Adam setting.

Throughout write $P = P_t$, $\rho = \rho_t$, $\eta = \eta_t$,
$\tau = \tau_t = \rho_t/\eta_t$, and let $\mu_1^{(t)}, \ldots, \mu_d^{(t)}$
be the eigenvalues of $P_t H$, equivalently of $S_t := P_t^{1/2} H
P_t^{1/2}$.

We use the following \emph{standing ellipticity and stabilization}
conditions on the preconditioner: there are $t$-independent constants
$0 < p_- \le p_+ < \infty$, $0 < \mu_- \le \mu_+ < \infty$, and
$C_P < \infty$ such that, almost surely for all $t$,
\begin{equation}
\label{eq:appA-standing}
  p_- I \preceq P_t \preceq p_+ I, \qquad
  \mu_- I \preceq P_t^{1/2} H P_t^{1/2} \preceq \mu_+ I, \qquad
  \|P_{t+1} - P_t\|_{\mathrm{op}} \le C_P\, t^{-1}.
\end{equation}
For the SA-Adam preconditioner $P_t = \mathrm{Diag}(\hat v_{t-1})^{-1/2}$
all three conditions in~\eqref{eq:appA-standing} are verified in
Proposition~\ref{prop:sa-adam-stab}(ii): bounded gradients keep
$\hat v_{t-1}$ in a fixed interval $[\epsilon, c_G]$ with $O(t^{-1})$
increments (the bias correction contributing only an $O(t^{-c_2-1})$
subdominant term), which the Lipschitz diagonal maps
$x \mapsto x^{\pm1/2}$ propagate to the two-sided bounds and the
$O(t^{-1})$ increment of $P_t$; the $\hat m_t$ vs $m_t$ distinction is
handled separately as an effective-step-size perturbation in
Section~\ref{sec:bias-correction}.
More generally, for any preconditioner satisfying
Assumption~\ref{ass:stab} with $\beta = 1$, the matrix-inverse identity
$P_{t+1} - P_t = -P_{t+1}\,H\,(M_{t+1} - M_t)\,P_t$ combined with the
upper bound $\|P_t\| \le p_+$ yields
$\|P_{t+1}-P_t\|_{\mathrm{op}} \le p_+^2 \|H\|\,\|M_{t+1} - M_t\|_{\mathrm{op}}
= O(t^{-1})$.  Note that
Assumption~\ref{ass:stab} by itself controls only $M_t = (P_t H)^{-1}$
(an upper bound on $M_t$, hence a \emph{lower} spectral bound on
$P_t H$); the uniform \emph{upper} bound on $P_t$
in~\eqref{eq:appA-standing} is additional input that the SA-Adam
construction supplies.  For a general preconditioner, predictability
($P_t$ being $\mathcal F_{t-1}$-measurable, as the one-step conditioning
below requires) and~\eqref{eq:appA-standing} should be assumed directly;
the increment rate $O(t^{-1})$ is stronger than the CLT-remainder
threshold $\beta > (\alpha+1)/2$ and is what the Lyapunov time-variation
step below requires.

\subsection{The Symmetrizer and Its Exact Identities}

We define the symmetrizer and establish the two exact identities and the
two-sided bounds on which the Lyapunov argument rests.

\begin{definition}[Symmetrizer]
\label{def:Qsym}
For each $t$ define the symmetric matrix
\begin{equation}
\label{eq:Qsym}
  Q_t := \begin{pmatrix}
    H & \tfrac12\bigl(\eta_t H P_t - I\bigr) \\[3pt]
    \tfrac12\bigl(\eta_t P_t H - I\bigr) & \dfrac{1 - \rho_t}{\tau_t}\,P_t
  \end{pmatrix}.
\end{equation}
\end{definition}

Symmetry holds because $(\eta H P - I)^\top = \eta P H - I$ and
$H, P_t$ are symmetric, and does \emph{not} require $H P_t = P_t H$.
Positive definiteness is not asserted at this point: the uniform
anisotropic lower bound---in particular a strictly positive,
$t$-independent lower eigenvalue bound $\lambda_{\min}(Q_t) \ge c_0 > 0$---is
established only for all sufficiently large $t$ in
Lemma~\ref{lem:Qsym-bounds}, and the proof of
Proposition~\ref{prop:augmented-mse} starts the Lyapunov recursion there,
the finitely many earlier indices only enlarging the bootstrap constant
$K$.

\begin{lemma}[Scalar reduction and two exact identities]
\label{lem:Qsym-id}
Let $T_t := \mathrm{diag}(P_t^{1/2}, P_t^{-1/2})$.  Then
$\bar L_t := T_t^{-1} L_t T_t$ and $\bar Q_t := T_t^\top Q_t T_t$ are
\[
  \bar L_t = \begin{pmatrix} \rho S_t & (1-\rho) I \\ -\tau S_t & \tau I
  \end{pmatrix},
  \qquad
  \bar Q_t = \begin{pmatrix} S_t & \tfrac12(\eta S_t - I) \\[2pt]
    \tfrac12(\eta S_t - I) & \tfrac{1-\rho}{\tau} I \end{pmatrix},
\]
both functions of $S_t$ alone; diagonalizing $S_t = U\,\mathrm{diag}
(\mu_i^{(t)})\,U^\top$ and applying $\mathrm{diag}(U, U)$ reduces them to
the $d$ scalar pairs
\[
  \mathrm L_i = \begin{pmatrix} \rho\mu_i & 1-\rho \\ -\tau\mu_i & \tau
  \end{pmatrix},
  \qquad
  \mathrm q_i = \begin{pmatrix} \mu_i & \tfrac{\eta\mu_i - 1}{2} \\[2pt]
    \tfrac{\eta\mu_i - 1}{2} & \tfrac{1-\rho}{\tau} \end{pmatrix},
  \qquad \eta = \tfrac{\rho}{\tau}.
\]
Each pair obeys the exact identities
\begin{equation}
\label{eq:Qsym-id}
  \mathrm L_i^\top \mathrm q_i + \mathrm q_i \mathrm L_i
    = (\tau + \rho\mu_i)\,\mathrm q_i,
  \qquad
  \mathrm L_i^\top \mathrm q_i \mathrm L_i = \tau\mu_i\,\mathrm q_i .
\end{equation}
Consequently, in the original coordinates, whenever $Q_t \succeq 0$---in
particular for all sufficiently large $t$, by
Lemma~\ref{lem:Qsym-bounds}---
\begin{equation}
\label{eq:Qsym-matrix}
  L_t^\top Q_t + Q_t L_t \succeq \tau_t\, Q_t,
  \qquad
  L_t^\top Q_t L_t \preceq \mu_+\,\tau_t\, Q_t .
\end{equation}
\end{lemma}

\begin{proof}
The two congruences are immediate and use no commutativity between
$P_t$ and $H$: $T_t^{-1} L_t T_t$ replaces the
blocks $\rho P H, (1-\rho) P, -\tau H$ by $\rho S_t, (1-\rho) I,
-\tau S_t$ (e.g.\ $P^{-1/2}(\rho P H) P^{1/2} = \rho P^{1/2} H P^{1/2}
= \rho S_t$, and it leaves $\tau I$), and $T_t^\top Q_t T_t$ replaces
$H, \eta H P, \tfrac{1-\rho}{\tau} P$ by $S_t, \eta S_t,
\tfrac{1-\rho}{\tau} I$.  Both depend on $P_t$ only through $S_t$, so
$\mathrm{diag}(U, U)$ block-diagonalizes them into the $\mathrm L_i,
\mathrm q_i$ above.  For~\eqref{eq:Qsym-id}, write $\mathrm L_i =
\bigl(\begin{smallmatrix} a & b \\ c & d \end{smallmatrix}\bigr)$ with
$a = \rho\mu_i$, $b = 1-\rho$, $c = -\tau\mu_i$, $d = \tau$, and
$\mathrm q_i = \bigl(\begin{smallmatrix} p & r \\ r & s
\end{smallmatrix}\bigr)$ with $p = \mu_i$, $r = \tfrac{\eta\mu_i - 1}{2}$,
$s = \tfrac{1-\rho}{\tau}$; throughout we use $\eta\tau = \rho$, which
gives $\mathrm{tr}\,\mathrm L_i = a + d = \tau + \rho\mu_i$ and
$\det\mathrm L_i = ad - bc = \rho\mu_i\tau + (1-\rho)\tau\mu_i = \tau\mu_i$.
Direct multiplication gives the symmetric matrix
\[
  \mathrm L_i^\top \mathrm q_i + \mathrm q_i \mathrm L_i
  = \begin{pmatrix}
      2(a p + c r) & (a+d) r + (c s + p b) \\[2pt]
      (a+d) r + (c s + p b) & 2(b r + d s)
    \end{pmatrix},
\]
and we check entrywise that it equals $(a+d)\,\mathrm q_i$:
\begin{itemize}
\item[(i)] off-diagonal: $c s + p b
  = -\tau\mu_i\cdot\tfrac{1-\rho}{\tau} + \mu_i(1-\rho) = 0$,
  leaving $(a+d) r$;
\item[(ii)] $(1,1)$ entry: $2(a p + c r)
  = 2\rho\mu_i^2 - \tau\mu_i(\eta\mu_i - 1)
  = 2\rho\mu_i^2 - \rho\mu_i^2 + \tau\mu_i
  = (\rho\mu_i + \tau)\,\mu_i = (a+d) p$ (using $\tau\eta = \rho$);
\item[(iii)] $(2,2)$ entry: $2(b r + d s)
  = (1-\rho)(\eta\mu_i - 1) + 2(1-\rho)
  = (1-\rho)(\eta\mu_i + 1)
  = (\tau + \rho\mu_i)\tfrac{1-\rho}{\tau} = (a+d) s$
  (using $\eta = \rho/\tau$).
\end{itemize}
Hence $\mathrm L_i^\top \mathrm q_i + \mathrm q_i \mathrm L_i
= (\tau + \rho\mu_i)\,\mathrm q_i$, the first identity.  The second is
then automatic from the Cayley--Hamilton relation for $2\times2$
matrices, $\mathrm{adj}\,\mathrm L_i = (\mathrm{tr}\,\mathrm L_i)
I - \mathrm L_i$: the first identity reads $\mathrm L_i^\top \mathrm q_i
= \mathrm q_i\bigl((\mathrm{tr}\,\mathrm L_i) I - \mathrm L_i\bigr)
= \mathrm q_i\,\mathrm{adj}\,\mathrm L_i$, and right-multiplying by
$\mathrm L_i$ with $\mathrm{adj}(\mathrm L_i)\,\mathrm L_i
= (\det\mathrm L_i) I = \tau\mu_i I$ gives $\mathrm L_i^\top \mathrm q_i
\mathrm L_i = \tau\mu_i\,\mathrm q_i$.  (Structurally, $\mathrm q_i$ is,
up to scale, the unique symmetric form rendering the trace-free part
$\mathrm L_i - \tfrac12(\mathrm{tr}\,\mathrm L_i) I$ skew-adjoint,
i.e.\ the invariant quadratic form of the complex-eigenvalue block of
Lemma~\ref{lem:hurwitz}; cf.\ the real rotation--scaling normal form
\citep{hornjohnson2013}.)  Summing~\eqref{eq:Qsym-id} over $i$ and using $\mathrm q_i \succeq 0$ for
all large $t$ (Lemma~\ref{lem:Qsym-bounds}) together with
$\tau + \rho\mu_i \ge \tau$ and $\tau\mu_i \le \mu_+ \tau$ gives
$\bar L_t^\top \bar Q_t + \bar Q_t \bar L_t \succeq \tau \bar Q_t$ and
$\bar L_t^\top \bar Q_t \bar L_t \preceq \mu_+ \tau \bar Q_t$.  The
congruence $L_t = T_t \bar L_t T_t^{-1}$, $Q_t = T_t^{-\top} \bar Q_t
T_t^{-1}$ preserves both Loewner inequalities (e.g.\ $L_t^\top Q_t +
Q_t L_t = T_t^{-\top}(\bar L_t^\top \bar Q_t + \bar Q_t \bar L_t)
T_t^{-1}$), yielding~\eqref{eq:Qsym-matrix}.
\end{proof}

The same scalar reduction yields the two-sided eigenvalue bounds on $Q_t$ used
to convert the Lyapunov estimate into a Euclidean mean-square bound.

\begin{lemma}[Uniform and anisotropic bounds]
\label{lem:Qsym-bounds}
There exist $t$-independent constants $0 < c_0 \le C_0 < \infty$ such
that, for all sufficiently large $t$,
\begin{equation}
\label{eq:Qsym-bounds}
  c_0\, \begin{pmatrix} I_d & 0 \\ 0 & \tau_t^{-1} I_d \end{pmatrix}
  \;\preceq\; Q_t \;\preceq\; C_0\,\tau_t^{-1}\, I_{2d}.
\end{equation}
The lower bound is \emph{anisotropic}---it weights the buffer block by
$\tau_t^{-1}$---and in particular $\lambda_{\min}(Q_t) \ge c_0$ is bounded
below independently of $t$.  Consequently, if $\mathbb E V_t = O(\eta_t)$
for $V_t = z^\top Q_t z$, then
\begin{equation}
\label{eq:two-scale-split}
  \mathbb E\|\Delta_t\|^2 = O(\eta_t) = O(t^{-\alpha}), \qquad
  \mathbb E\|m_{t-1}\|^2 = O(\tau_t\eta_t) = O(\rho_t) = O(t^{-\gamma}),
\end{equation}
the two-time-scale split anticipated earlier.
\end{lemma}

\begin{proof}
By the orthogonal reduction of Lemma~\ref{lem:Qsym-id}---the congruence
$\mathrm{diag}(U,U)$ block-diagonalizes $\bar Q_t = T_t^\top Q_t T_t$ into
the $2\times2$ blocks $\mathrm q_i$, so $\mathrm{spec}(\bar Q_t) =
\bigcup_{i} \mathrm{spec}(\mathrm q_i)$ and the Loewner relations
$\bar Q_t \succeq c\,\mathrm{diag}(I_d, \tau^{-1}I_d)$,
$\bar Q_t \preceq C\tau^{-1}I_{2d}$ hold iff the corresponding scalar
relations hold for every $\mathrm q_i$---it suffices to bound each
\[
  \mathrm q_i = \begin{pmatrix} \mu_i & r_i \\ r_i & s \end{pmatrix},
  \qquad r_i := \tfrac{\eta\mu_i - 1}{2}, \quad s := \tfrac{1-\rho}{\tau},
\]
and then transfer the bounds through the congruence $T_t =
\mathrm{diag}(P_t^{1/2}, P_t^{-1/2})$.  Since $\eta, \rho, \tau \to 0$ and
$\mu_- \le \mu_i \le \mu_+$, for all large $t$ we have $0 < \eta\mu_i < 1$,
hence $|r_i| = \tfrac12|1 - \eta\mu_i| \le \tfrac12$ and $r_i^2 \le
\tfrac14$, and $\tau^{-1} \ge \max\{\mu_+, 1\}$.

\emph{Upper bound.}  For a symmetric $2\times2$ block $\lambda_{\max} \le
\max(p,s) + |r|$, so $\lambda_{\max}(\mathrm q_i) \le \max(\mu_i, s) +
|r_i| \le 2\tau^{-1}$; the rate is sharp, $\lambda_{\max}(\mathrm q_i) =
s(1+o(1)) = \tfrac{1-\rho}{\tau}(1+o(1))$, so $\bar Q_t \preceq 2\tau^{-1}
I_{2d}$.

\emph{Anisotropic lower bound.}  Fix $c \in (0, \tfrac12\min\{\mu_-,1\})$.
The $(2,2)$ entry of $\mathrm q_i - c\,\mathrm{diag}(1,\tau^{-1})$ is
$(1-\rho-c)/\tau > 0$ for large $t$, and its Schur complement
\citep{hornjohnson2013} is $(\mu_i-c) - \tau r_i^2/(1-\rho-c) \ge
\tfrac14\mu_- > 0$ uniformly in $i$ (since $\mu_i - c \ge \tfrac12\mu_-$
and the subtracted term is $O(\tau)$).  Hence $\mathrm q_i \succeq
c\,\mathrm{diag}(1,\tau^{-1})$, i.e.\ $\bar Q_t \succeq c\,\mathrm{diag}(I_d,
\tau^{-1}I_d)$.

\emph{Transfer.}  Congruence by $T_t^{-1} = \mathrm{diag}(P_t^{-1/2},
P_t^{1/2})$ preserves the Loewner order; with $p_- I \preceq P_t \preceq
p_+ I$ from~\eqref{eq:appA-standing}, the lower bound transfers to $Q_t
\succeq c_0\,\mathrm{diag}(I_d, \tau_t^{-1}I_d)$ ($c_0 :=
c\min(p_+^{-1},p_-)$, whence $\lambda_{\min}(Q_t) \ge c_0$) and the upper
to $Q_t \preceq C_0\,\tau_t^{-1} I_{2d}$ ($C_0 := 2\max(p_-^{-1},p_+)$),
which is~\eqref{eq:Qsym-bounds}.

\emph{Two-time-scale split.}  With $z = (\Delta_t, m_{t-1})$, the
anisotropic lower bound gives $V_t = z^\top Q_t z \ge c_0(\|\Delta_t\|^2
+ \tau_t^{-1}\|m_{t-1}\|^2)$.  Hence $\mathbb E V_t = O(\eta_t)$ forces,
term by term,
$\mathbb E\|\Delta_t\|^2 \le c_0^{-1}\,\mathbb E V_t = O(\eta_t)
  = O(t^{-\alpha})$ and
$\mathbb E\|m_{t-1}\|^2 \le c_0^{-1}\tau_t\,\mathbb E V_t
  = O(\tau_t\eta_t) = O(\rho_t) = O(t^{-\gamma})$,
using $\tau_t\eta_t = \rho_t$; this is the split~\eqref{eq:two-scale-split}.
\end{proof}

\emph{Two alternative constructions.}
Two other choices satisfy a Lyapunov decrease but are worse scaled.  A
block-diagonal $Q_t^{\mathrm{bd}} = \mathrm{diag}\bigl(\tfrac{\tau}{\rho}
P^{-1}H^{-1}P^{-1},\, \tfrac{1-\rho}{\rho} H^{-1}P^{-1}H^{-1}\bigr)$ gives
an exact continuous-time Lyapunov drift (its off-diagonal terms cancel
without requiring $P_tH = HP_t$), but not the exact discrete one-step
identity~\eqref{eq:Qsym-onestep} of the preferred symmetrizer, and---as
$P_t, H$ have bounded spectra, so its eigenvalues track the scalar
prefactors
$\tfrac{\tau}{\rho} = \eta_t^{-1}$ and
$\tfrac{1-\rho}{\rho}\asymp\rho_t^{-1}$---has largest
eigenvalue of order $\rho_t^{-1} \gg \tau_t^{-1}$; the integral solution
$Q_t^{\mathrm{int}} = \int_0^\infty e^{-sL_t^\top}e^{-sL_t}\,\mathrm ds$ of $L_t^\top
Q_t^{\mathrm{int}} + Q_t^{\mathrm{int}} L_t = I$ exists by positive stability (Lemma~\ref{lem:hurwitz})
and is scaled as $\tau_t^{-2}$: $L_t$ is highly non-normal, with spectral
gap $\asymp\tau_t$ and an eigenvector matrix of condition number
$\asymp\tau_t^{-1/2}$, giving the upper bound $\|Q_t^{\mathrm{int}}\| \le \int_0^\infty
\|e^{-sL_t}\|^2\,\mathrm ds \lesssim \tau_t^{-1}/\tau_t = \tau_t^{-2}$---an
order already attained by a single scalar block, where the Lyapunov
solution has an entry $\asymp (\mathrm{tr}\,\mathrm L_i \cdot
\det\mathrm L_i)^{-1} \asymp \tau_t^{-2}$.  The
symmetrizer~\eqref{eq:Qsym} is preferred precisely for its anisotropic
bound~\eqref{eq:Qsym-bounds}---$O(1)$ lower and $\tau_t^{-1}$ upper
eigenvalue---which yields the $O(t^{-\alpha})$ MSE rate directly.

\subsection{One-Step Decrease and the Iterate Mean-Square Bound}

The exact identities of Lemma~\ref{lem:Qsym-id} now pay off twice---an exact
discrete one-step decrease, and, with control of the symmetrizer's time
variation, the augmented mean-square bound.  We establish the two ingredients
in turn, then assemble them.

\begin{lemma}[Exact one-step identity]
\label{lem:Qsym-decrease}
For every $t$,
\begin{equation}
\label{eq:Qsym-onestep}
  (I - \eta_t L_t)^\top Q_t\,(I - \eta_t L_t) = (1 - \rho_t)\, Q_t .
\end{equation}
\end{lemma}

\begin{proof}
By the congruence of Lemma~\ref{lem:Qsym-id} it suffices to prove the
identity for each scalar block; this uses no positivity of $\mathrm q_i$,
since~\eqref{eq:Qsym-id} are exact equalities, so the conclusion holds
for \emph{every} $t$---in contrast to the Loewner
inequalities~\eqref{eq:Qsym-matrix}, which need $Q_t \succeq 0$.  Writing
$\mathrm M_i := I - \eta\mathrm L_i$ and expanding the congruence,
$\mathrm M_i^\top \mathrm q_i\,\mathrm M_i = \mathrm q_i
- \eta\bigl(\mathrm L_i^\top \mathrm q_i + \mathrm q_i \mathrm L_i\bigr)
+ \eta^2\,\mathrm L_i^\top \mathrm q_i \mathrm L_i$.
Substituting \emph{both} exact identities~\eqref{eq:Qsym-id},
$\mathrm L_i^\top \mathrm q_i + \mathrm q_i \mathrm L_i =
(\tau + \rho\mu_i)\,\mathrm q_i$ and
$\mathrm L_i^\top \mathrm q_i \mathrm L_i = \tau\mu_i\,\mathrm q_i$, yields
$\mathrm M_i^\top \mathrm q_i\,\mathrm M_i
= \bigl(1 - \eta(\tau + \rho\mu_i) + \eta^2\tau\mu_i\bigr)\,\mathrm q_i$.
The scalar factor is exactly $\det(\mathrm M_i)$: the $2\times2$
expansion $\det(I - \eta\mathrm L_i) = 1 - \eta\,\mathrm{tr}\,\mathrm L_i
+ \eta^2\det\mathrm L_i$, with $\mathrm{tr}\,\mathrm L_i = \tau + \rho\mu_i$
and $\det\mathrm L_i = \tau\mu_i$ (Lemma~\ref{lem:Qsym-id}), reproduces
the bracket.  Hence $\mathrm M_i^\top \mathrm q_i \mathrm M_i =
\det(\mathrm M_i)\,\mathrm q_i$---the invariant-form relation of the
second identity in~\eqref{eq:Qsym-id}, carried over from $\mathrm L_i$ to
the one-step map $\mathrm M_i$.  Finally, since $\eta\tau = \rho$, the two
$\mu_i$-terms cancel,
$-\eta\rho\mu_i + \eta^2\tau\mu_i = \eta\mu_i(\eta\tau - \rho) = 0$,
so $\det(\mathrm M_i) = 1 - \eta\tau = 1 - \rho$ and
$\mathrm M_i^\top \mathrm q_i \mathrm M_i = (1 - \rho)\,\mathrm q_i$.  The
congruence $Q_t = T_t^{-\top}\bar Q_t T_t^{-1}$, $L_t = T_t\bar L_t
T_t^{-1}$---whence $I - \eta_t L_t = T_t(I - \eta_t\bar L_t)T_t^{-1}$---then
gives~\eqref{eq:Qsym-onestep} for every $t$.
\end{proof}

Because $Q_t$ is time-varying, the Lyapunov recursion also feels its
step-to-step increment $Q_{t+1} - Q_t$; the next lemma shows this feedback is
dominated by the contraction under $\gamma < 1$.

\begin{lemma}[Blockwise time variation]
\label{lem:Qsym-timevar}
Let $D_t := Q_{t+1} - Q_t$.  Under the standing
conditions~\eqref{eq:appA-standing} and the momentum schedule
$\rho_t = c_1/t^\gamma$ with $\gamma\in(\alpha,1)$, $D_t$ has the block form
\[
  D_t = \begin{pmatrix} 0 & D_{12,t} \\ D_{12,t}^\top & D_{22,t}
  \end{pmatrix},\qquad
  \|D_{12,t}\|_{\mathrm{op}} = O(t^{-\alpha-1}),\quad
  \|D_{22,t}\|_{\mathrm{op}} = O(t^{\gamma-\alpha-1}),
\]
its iterate block vanishing because the $(1,1)$ block of $Q_t$
in~\eqref{eq:Qsym} is the $t$-independent matrix $H$.  Consequently, under the further
assumptions of Proposition~\ref{prop:augmented-mse} and with the threshold
constant
\[
  C_\sharp^\ast = C_\sharp^\ast(\gamma, p_\pm, \mu_\pm, C_P)
  := [(\gamma-\alpha)p_+ + C_P]/c_0,
\]
for every fixed $C_\sharp > C_\sharp^\ast$ there is a constant $C \ge 0$,
with both $C_\sharp$ and $C$ \emph{independent of $t$ and of any bootstrap
constant}, such that
\begin{equation}
\label{eq:timevar-bound}
  \bigl|\mathbb E[z_{t+1}^\top D_t z_{t+1}]\bigr|
  \;\le\; \tfrac{C_\sharp}{t}\,\mathbb E V_t \;+\; C\,\eta_t\rho_t
  \qquad\text{for all sufficiently large } t .
\end{equation}
The threshold $C_\sharp^\ast$ is independent of $c_1$, and the feedback
coefficient $C_\sharp/t$ is $o(\rho_t)$ under $\gamma < 1$, so the
contraction $\rho_t$ dominates the feedback without requiring a
$c_1$-threshold (cf.~\eqref{eq:c1-threshold}).
\end{lemma}

\begin{proof}
The $(1,1)$ block of $Q_t$ in~\eqref{eq:Qsym} is the $t$-independent
matrix $H$, so $D_{11,t} = H - H = 0$.  For the off-diagonal block,
$D_{12,t} = \tfrac12(\eta_{t+1} H P_{t+1} - \eta_t H P_t)$; telescoping the
product across the two indices,
$\eta_{t+1} H P_{t+1} - \eta_t H P_t
  = (\eta_{t+1} - \eta_t)\,H P_t + \eta_{t+1}\,H\,(P_{t+1} - P_t)$,
so submultiplicativity and $\|P_t\|_{\mathrm{op}} \le p_+$ give
$\|D_{12,t}\|_{\mathrm{op}}
  \le \tfrac12\|H\|\bigl(|\eta_{t+1} - \eta_t|\,p_+
      + \eta_{t+1}\,\|P_{t+1} - P_t\|_{\mathrm{op}}\bigr)$.
With $\eta_t = \eta_0 t^{-\alpha}$ the mean value theorem yields
$|\eta_{t+1} - \eta_t| = \alpha\eta_0\,\zeta^{-\alpha-1}
\le \alpha\eta_0\,t^{-\alpha-1}$ for some $\zeta \in (t, t+1)$, while
$\eta_{t+1} = O(t^{-\alpha})$ and $\|P_{t+1} - P_t\| = O(t^{-1})$
by~\eqref{eq:appA-standing}; both terms are therefore $O(t^{-\alpha-1})$
and $\|D_{12,t}\| = O(t^{-\alpha-1})$.

For the buffer block, the $(2,2)$ entry of $Q_t$ is $b_t P_t$ with
$b_t = (1-\rho_t)\tau_t^{-1} = \tau_t^{-1} - \eta_t = \tau_t^{-1}(1+o(1))$.
Applying the mean value theorem to the regularly varying $b(x) =
(\eta_0/c_1)x^{\gamma-\alpha} - \eta_0 x^{-\alpha}$ gives $b_{t+1} - b_t =
(\gamma-\alpha)\tau_t^{-1}/t\,(1+o(1))$ (the $\eta_x/x$ term negligible
since $\eta_x\tau_x = \rho_x \to 0$), so telescoping as above, with
$\|P_{t+1}-P_t\| \le C_P/t$, yields $\|D_{22,t}\| \le
[(\gamma-\alpha)p_+ + C_P + o(1)]\,\tau_t^{-1}/t = O(t^{\gamma-\alpha-1})$.
Writing $z_{t+1} = (\Delta_{t+1}, m_t)$ and using $D_{11,t} = 0$,
$\bigl|\mathbb E[z_{t+1}^\top D_t z_{t+1}]\bigr|
  \le 2\|D_{12,t}\|\,\mathbb E[\|\Delta_{t+1}\|\,\|m_t\|]
    + \|D_{22,t}\|\,\mathbb E\|m_t\|^2$.
The anisotropic bound~\eqref{eq:Qsym-bounds} expresses both second
moments through $V_t$: $\mathbb E\|\Delta_t\|^2 \le c_0^{-1}\mathbb E V_t$
and $\mathbb E\|m_{t-1}\|^2 \le c_0^{-1}\tau_t\mathbb E V_t$.  We
propagate these one step explicitly.

\emph{One-step propagation.}  The buffer recursion $m_t =
(1-\rho_t)m_{t-1} + \rho_t(H\Delta_t + u_t) + \rho_t\xi_t$, with the
martingale increment conditionally orthogonal to the predictable part,
Jensen's inequality on the convex combination (weights $\rho_t \le 1$ for
$t \ge t_0$, the finitely many earlier indices absorbed below), $\|u_t\|
\le L_R\|\Delta_t\|^2$, and $\mathbb E\|\Delta_t\|^4 = O(\eta_t^2)$
(Assumption~\ref{ass:fourth}), gives
\begin{equation}
\label{eq:mbuffer-prop}
  \mathbb E\|m_t\|^2 \le c_0^{-1}\tau_t\,\mathbb E V_t\,(1+o(1)) + C'\eta_t\rho_t.
\end{equation}
The iterate recursion $\Delta_{t+1} = \Delta_t - \eta_t P_t m_t$ with
Young's inequality (weight $\rho_t$; the cross factor
$(1+\rho_t^{-1})\eta_t^2 = \eta_t/\tau_t\,(1+o(1)) \to 0$ since
$\gamma < 2\alpha$) and~\eqref{eq:mbuffer-prop} then give $\mathbb
E\|\Delta_{t+1}\|^2 \le c_0^{-1}\mathbb E V_t\,(1+o(1)) + C'\eta_t\rho_t$,
where $C'$ depends only on the structural constants $(p_\pm, \|H\|, L_R,
c_0, \mathrm{tr}\,\overline S, C_4)$, not on any bootstrap constant.
Hence the buffer term carries the feedback, $\|D_{22,t}\|\,\mathbb
E\|m_t\|^2 \le C_\sharp^\ast(1+o(1))\,\mathbb E V_t/t +
O(\eta_t\rho_t)$ with $C_\sharp^\ast := [(\gamma-\alpha)p_+ + C_P]/c_0$,
at most $(C_\sharp/t)\,\mathbb E V_t + O(\eta_t\rho_t)$ for any fixed
$C_\sharp > C_\sharp^\ast$ and large $t$.

The cross term is lower order: by Cauchy--Schwarz and the two one-step
bounds, $\mathbb E[\|\Delta_{t+1}\|\,\|m_t\|] \le
c_0^{-1}\tau_t^{1/2}\,\mathbb E V_t\,(1+o(1)) + C''(\mathbb E
V_t\,\eta_t\rho_t)^{1/2} + C''\eta_t\rho_t$; multiplying by
$2\|D_{12,t}\| = O(\eta_t/t)$ and applying Young's inequality to the
middle term gives $2\|D_{12,t}\|\,\mathbb E[\|\Delta_{t+1}\|\,\|m_t\|]
\le o(t^{-1})\,\mathbb E V_t + O(\eta_t\rho_t)$.
Adding the two contributions gives~\eqref{eq:timevar-bound} for all
sufficiently large $t$: the $\mathbb E V_t$-proportional part is at most
$(C_\sharp/t)\,\mathbb E V_t$ for any fixed $C_\sharp > C_\sharp^\ast$
(both $C_\sharp$ and $C$ independent of any bootstrap constant), and $C$
collects the remaining $K$-independent $O(\eta_t\rho_t)$ terms.
\end{proof}

We now assemble the one-step decrease, the time-variation bound, and the noise
and Taylor terms into the Lyapunov recursion for $\mathbb E V_t$.

\begin{proof}[Completion of the proof of Proposition~\ref{prop:augmented-mse}]
Set $V_t(z) := z^\top Q_t z$, so $c_0\|z\|^2 \le V_t(z) \le C_0\tau_t^{-1}
\|z\|^2$ for all sufficiently large $t$ by~\eqref{eq:Qsym-bounds}, and
write $G_t := I - \eta_t L_t$.

\emph{Predictable one-step.}  Since $Q_t, G_t, B_t$ and $u_t = r(x_t)$
are $\mathcal F_{t-1}$-measurable (whereas $Q_{t+1}$ is only
$\mathcal F_t$-measurable, $P_{t+1}$ depending on $g_t$), we condition
with the \emph{predictable} $Q_t$, not $Q_{t+1}$.  Using
$z_{t+1} = G_t z_t + B_t(u_t + \xi_t)$, $\mathbb E[\xi_t \mid
\mathcal F_{t-1}] = 0$, and $\Sigma_t := \mathbb E[\xi_t\xi_t^\top \mid
\mathcal F_{t-1}] \preceq \overline S$,
\begin{equation}
\label{eq:mse-onestep}
  \mathbb E[z_{t+1}^\top Q_t z_{t+1} \mid \mathcal F_{t-1}]
  = z_t^\top G_t^\top Q_t G_t\, z_t
  + 2\langle G_t z_t,\, Q_t B_t u_t\rangle
  + (B_t u_t)^\top Q_t (B_t u_t)
  + \mathrm{tr}(Q_t B_t \Sigma_t B_t^\top),
\end{equation}
and the exact identity~\eqref{eq:Qsym-onestep} gives $G_t^\top Q_t G_t =
(1-\rho_t) Q_t$.

\emph{Noise.}  As $\|Q_t\|_{\mathrm{op}} = O(\tau_t^{-1})$ and the
dominant block of $B_t \Sigma_t B_t^\top$ is $\rho_t^2\Sigma_t$,
$\mathrm{tr}(Q_t B_t \Sigma_t B_t^\top) = O(\tau_t^{-1}\rho_t^2)
  = O(\eta_t\rho_t)$, using $\rho_t/\tau_t = \eta_t$.

\emph{Taylor terms.}  The remainder enters only through $B_t u_t$, with
$\|u_t\| \le L_R\|\Delta_t\|^2$ (Assumption~\ref{ass:quadratic}); the
quadratic term is $(B_t u_t)^\top Q_t(B_t u_t) =
O(\eta_t\rho_t\|\Delta_t\|^4)$ and, after Young's inequality on the cross
term, the non-absorbed part is $O(\varepsilon^{-1}\eta_t\|\Delta_t\|^4)$.
By Assumption~\ref{ass:fourth} ($\mathbb E\|\Delta_t\|^4 =
O(t^{-2\alpha})$), both are $O(t^{-3\alpha}) = o(\eta_t\rho_t)$.

\emph{Closing the recursion.}  Taking expectations
in~\eqref{eq:mse-onestep}, using $G_t^\top Q_t G_t = (1-\rho_t)Q_t$ and
the noise and Taylor bounds, and adding the correction
$\mathbb E[z_{t+1}^\top(Q_{t+1} - Q_t) z_{t+1}]$ that converts
$\mathbb E[z_{t+1}^\top Q_t z_{t+1}]$ into $\mathbb E V_{t+1}$, we obtain,
for any fixed $\varepsilon \in (0,1)$,
\begin{equation}
\label{eq:mse-recursion}
  \mathbb E V_{t+1} \le \bigl(1 - (1-\varepsilon)\rho_t\bigr)\mathbb E V_t
  + C_\varepsilon\,\eta_t\rho_t
  + \bigl|\mathbb E[z_{t+1}^\top(Q_{t+1} - Q_t) z_{t+1}]\bigr|.
\end{equation}
The last term is the time variation.  For all sufficiently large $t$,
Lemma~\ref{lem:Qsym-timevar} supplies the absorption
bound~\eqref{eq:timevar-bound}, $|\mathbb E
[z_{t+1}^\top(Q_{t+1}-Q_t)z_{t+1}]| \le (C_\sharp/t)\mathbb E V_t +
C\eta_t\rho_t$, where $C_\sharp$ is any fixed constant exceeding
$C_\sharp^\ast = [(\gamma-\alpha)p_+ + C_P]/c_0$, and $C_\sharp, C$ are
both independent of any bootstrap constant.  Substituting
into~\eqref{eq:mse-recursion} and using $\rho_t = c_1/t^\gamma$,
\begin{equation}
\label{eq:mse-recursion-final}
  \mathbb E V_{t+1} \le \bigl(1 - (1-\varepsilon)\rho_t + C_\sharp/t\bigr)
  \mathbb E V_t + C'\,\eta_t\rho_t .
\end{equation}
Since $\gamma < 1$, the feedback-to-contraction ratio $(C_\sharp/t)/\rho_t
= (C_\sharp/c_1)\,t^{\gamma-1} \to 0$, so the feedback is negligible
against $\rho_t$ for \emph{any} $c_1 > 0$:
\begin{equation}
\label{eq:c1-threshold}
  c_1 > 0
  \qquad (\text{the threshold } c_1^\ast(\gamma) \text{ is trivial for } \gamma<1).
\end{equation}

\emph{Bootstrap induction.}  Fix $\varepsilon = \tfrac12$ and $t_0$ large
enough that $\rho_t \le 1$, \eqref{eq:timevar-bound} holds, and the
$\rho_t$-dominance
\begin{equation}
\label{eq:rho-dominates}
  (1-\varepsilon)\rho_t - \frac{C_\sharp + \alpha}{t} + O(t^{-2})
  \;\ge\; \tfrac12(1-\varepsilon)\,\rho_t
\end{equation}
holds for $t \ge t_0$ (as $\rho_t = c_1 t^{-\gamma}$, $\gamma < 1$,
dominates $t^{-1}$).  With $K := \max\{2C'/(1-\varepsilon),
\max_{t \le t_0}\eta_t^{-1}\mathbb E V_t\} < \infty$ (finite since bounded
gradients make each $\mathbb E V_t < \infty$), induction
on~\eqref{eq:mse-recursion-final} via $\eta_{t+1} = \eta_t(1 - \alpha/t +
O(t^{-2}))$ and~\eqref{eq:rho-dominates} gives $\mathbb E V_t \le K\eta_t$
for all $t \ge 1$.

Hence $\mathbb E V_t = O(\eta_t) = O(t^{-\alpha})$, and
by~\eqref{eq:two-scale-split} $\mathbb E\|z_t\|^2 =
\mathbb E\|\Delta_t\|^2 + \mathbb E\|m_{t-1}\|^2 = O(t^{-\alpha})$ for all
$t \ge 1$, verifying Assumption~\ref{ass:iterate}.
\end{proof}

\paragraph{What is proved and what is assumed.}
The construction~\eqref{eq:Qsym}, its
identities~\eqref{eq:Qsym-id}--\eqref{eq:Qsym-matrix}, the anisotropic
bounds~\eqref{eq:Qsym-bounds}, the exact one-step
identity~\eqref{eq:Qsym-onestep}, and the time-variation bound
(Lemma~\ref{lem:Qsym-timevar}) hold for any positive-definite $P_t$
obeying~\eqref{eq:appA-standing}, with no commutativity hypothesis: the
exact identities~\eqref{eq:Qsym-id} and~\eqref{eq:Qsym-onestep} for every
$t$, and the positivity-dependent bounds~\eqref{eq:Qsym-matrix},
\eqref{eq:Qsym-bounds}, and Lemma~\ref{lem:Qsym-timevar} for all
sufficiently large $t$, the expectation bound of the latter additionally
requiring the stochastic assumptions of Proposition~\ref{prop:augmented-mse}.
Given these, the completion is an elementary predictable one-step
recursion closed by a bootstrap in which, under $\gamma\in(\alpha,1)$,
the contraction $\rho_t$ dominates the time-variation feedback for any
$c_1 > 0$, the constant $C_\sharp$ affecting only the leading
$\eta_t$-coefficient and not the rate.  Beyond the standing
stochastic-approximation assumptions and schedule in force throughout
this section---martingale-difference noise, the conditional covariance
bound, bounded gradients, and $\gamma\in(\alpha,1)$---the argument
requires only two additional inputs: the preconditioner
conditions~\eqref{eq:appA-standing} (supplied by SA-Adam,
Proposition~\ref{prop:sa-adam-stab}) and the fourth-moment stability of
Assumption~\ref{ass:fourth} (used only for the Taylor remainder).

\section{Detailed Proof of Theorem~\ref{thm:augmented-clt}}
\label{app:clt-proof}

We complete the steps of the proof sketch in Section~\ref{sec:clt}.
The proof proceeds entirely in the raw coordinates of $z_t$: the
exact Polyak--Ruppert identity~\eqref{eq:AtBt} below extracts the
$H^{-1}SH^{-1}$ sandwich for the iterate marginal directly, with no
rescaling required.  The buffer marginal is degenerate at the
$\sqrt n$-scale in these coordinates---see
Remark~\ref{rem:raw-vs-rescaled}---which is harmless for
Theorem~\ref{thm:main}, since only the iterate $(1,1)$ block of
$\Sigma_z$ enters the main result.

\paragraph{Provenance.}  The architecture is the classical
Polyak--Ruppert route---an Abel-summation decomposition of $\overline
z_n$ closed by a martingale CLT \citep{polyak1992acceleration,
halldavid1980martingale}---carried out for \emph{constant}-drift linear
SA by \citet[Thm.~1]{mou2020linear} and, for the \emph{time-varying} preconditioner
setting, by the remainder analysis of \citet{anhuo2026}.  New here
are the two ingredients that close the momentum-augmented route: the
exact identity $A_tB_t = (-H^{-1},0)^\top$~\eqref{eq:AtBt} and the
vanishing of the $\sqrt n$-scaled remainder under
$\gamma\in(\alpha,1)$ (Lemma~\ref{lem:Rnz-bound}).

\subsection{Polyak--Ruppert Decomposition}

For $z_{t+1} = (I - \eta_t L_t) z_t + B_t(u_t + \xi_t)$,
left-multiply by $A_t := L_t^{-1}/\eta_t$ and rearrange:
$z_t = A_t (z_t - z_{t+1}) + A_t B_t (u_t + \xi_t)$.
Summing over $t = 1, \ldots, n$ and applying Abel summation:
$\sum_{t=1}^n z_t = A_1 z_1 - A_n z_{n+1} + \sum_{t=2}^n (A_t -
  A_{t-1}) z_t + \sum_{t=1}^n A_t B_t (u_t + \xi_t)$.
Dividing by $n$:
\begin{equation}
\label{eq:pr-decomposition}
  \overline z_n = R_n^z + \overline{A_t B_t \xi_t} + \overline{A_t B_t u_t},
\end{equation}
where $R_n^z := n^{-1}[A_1 z_1 - A_n z_{n+1} + \sum_{t=2}^n (A_t -
A_{t-1}) z_t]$ and the bars denote sample means.

The point of the identity is that $A_t B_t = L_t^{-1} B_t/\eta_t$
extracts the canonical Polyak--Ruppert sandwich \emph{exactly}.  Using
$L_t^{-1}$ from Lemma~\ref{lem:Linv}, $B_t = (-\eta_t \rho_t P_t, \rho_t
I)^\top$, the identity $M_t P_t = (P_t H)^{-1} P_t = H^{-1}$, and
$\tau_t^{-1} = \eta_t/\rho_t$:
\begin{equation}
\label{eq:AtBt}
\begin{aligned}
  A_t B_t &= \frac{L_t^{-1} B_t}{\eta_t}
  = \frac{1}{\eta_t}\begin{pmatrix}
      -\eta_t \rho_t\, M_t P_t - (1-\rho_t)\rho_t \tau_t^{-1} H^{-1} \\
      -\eta_t \rho_t\, P_t^{-1} P_t + \rho_t^2 \tau_t^{-1} I
    \end{pmatrix} \\
  &= \frac{1}{\eta_t}\begin{pmatrix}
      -\eta_t \rho_t H^{-1} - (1-\rho_t)\eta_t H^{-1} \\
      -\eta_t \rho_t I + \eta_t \rho_t I
    \end{pmatrix}
  = \begin{pmatrix} -H^{-1} \\ 0 \end{pmatrix}.
\end{aligned}
\end{equation}
The identity is exact (no $o(1)$ corrections) and holds for every $t$
and every positive-definite $P_t$, using only $\tau_t = \rho_t/\eta_t >
0$; no bound $\tau_t \le 1$ is needed.  The iterate block is exactly
$-H^{-1}$ and the buffer block exactly $0$, so the buffer average carries
no leading-order term---the source of the $\sqrt n$-scale buffer
degeneracy (Remark~\ref{rem:raw-vs-rescaled}).  In particular
$\bigl[\overline{A_t B_t \xi_t}\bigr]_{1} = -H^{-1}\,\overline{\xi}_n$.

\subsection{Bound on the Remainder}

The remainder $R_n^z$ collects the boundary and increment terms of the Abel
summation~\eqref{eq:pr-decomposition}; we show it is negligible at the
$\sqrt n$ scale.

\begin{lemma}[Time-varying remainder bound]
\label{lem:Rnz-bound}
Assume the schedule $\rho_t = c_1/t^\gamma$ with $c_1>0$,
$\gamma \in (\alpha, 1)$, $\alpha \in (1/2, 1)$ (under the convex-weight
convention of Section~\ref{sec:sa-adam-defn}, i.e.\ the shifted schedule
$\rho_t = c_1/(t + t_0)^\gamma$ when $c_1 \ge 1$); the augmented
mean-square bounds $\|\Delta_t\|_2 = O(t^{-\alpha/2})$ and
$\|m_t\|_2 = O(t^{-\gamma/2})$ (i.e.~\eqref{eq:clt-stability}); and the
standing preconditioner conditions~\eqref{eq:appA-standing}
(in particular $\|P_t^{-1}\|_{\mathrm{op}} = O(1)$ and
$\|P_t - P_{t-1}\|_{\mathrm{op}} = O(t^{-1})$).  Then
\begin{equation}
\label{eq:Rnz-target}
  \|R_n^z\|_{L^2} = o(n^{-1/2}).
\end{equation}
\end{lemma}

\begin{proof}
We prove the lemma via two concrete identities,
$\sqrt n\,\overline m_{n-1} \to 0$ and $\sqrt n\,\overline g_n \to 0$ in
$L^2$, which together cover both blocks of $R_n^z$.

\emph{Setup.}  Bars denote \emph{averages}: $\overline g_n :=
\frac1n\sum_{t=1}^n g_t$ and $\overline m_{n-1} := \frac1n\sum_{t=0}^{n-1}
m_t$ (distinct from the endpoint $m_{n-1}$); write $\|X\|_2 := (\mathbb
E\|X\|^2)^{1/2}$.  The standing conditions~\eqref{eq:appA-standing} give
$\|P_t^{-1}\|_{\mathrm{op}} = O(1)$ and $\|P_t^{-1} -
P_{t-1}^{-1}\|_{\mathrm{op}} = O(t^{-1})$; all partial-sum bounds below
combine Minkowski's inequality with $\sum_{t=1}^n t^{\beta-1} =
O(n^\beta)$ ($\beta > 0$).

\emph{Averaged buffer.}  The update $\Delta_{t+1} = \Delta_t - \eta_t P_t
m_t$ inverts to $m_t = C_t(\Delta_t - \Delta_{t+1})$ with $C_t :=
\eta_t^{-1}P_t^{-1}$, so $\|C_t\| = O(t^\alpha)$ and $\|C_t - C_{t-1}\| =
O(t^{\alpha-1})$.  Summation by parts and $\|\Delta_t\|_2 =
O(t^{-\alpha/2})$ give $\bigl\|\sum_{t=1}^{n-1} m_t\bigr\|_2 \le O(1) +
O(n^\alpha)O(n^{-\alpha/2}) +
\sum_{t=2}^{n-1}O(t^{\alpha-1})O(t^{-\alpha/2}) = O(n^{\alpha/2})$, so
dividing by $n$ and scaling by $\sqrt n$,
\begin{equation}
\label{eq:mbar-vanish}
  \bigl\|\sqrt n\,\overline m_{n-1}\bigr\|_2
  \;=\; \Bigl\|\sqrt n\cdot \tfrac{1}{n}\sum_{t=0}^{n-1}m_t\Bigr\|_2
  \;=\; O(n^{(\alpha-1)/2}) \;\to\; 0 \quad\text{for }\alpha<1.
\end{equation}

\emph{Averaged gradient.}  Write $a_t := \rho_t^{-1}$, which satisfies
$a_t = O(t^\gamma)$ and $|a_{t+1} - a_t| = O(t^{\gamma-1})$ for both the
plain and the shifted (convex-weight) schedules of
Section~\ref{sec:sa-adam-defn}.  From the buffer recursion $g_t = a_t(m_t
- m_{t-1}) + m_{t-1}$, summation by parts gives
$\sum_{t=1}^n g_t
  = a_n m_n - a_1 m_0 + \sum_{t=1}^{n-1}(a_t - a_{t+1})\,m_t
    + \sum_{t=1}^n m_{t-1}
  = a_n m_n + \sum_{t=1}^{n-1}(a_t - a_{t+1})\,m_t
    + \sum_{t=1}^n m_{t-1}$,
the boundary term $a_1 m_0$ vanishing by the initialization $m_0 = 0$.
Taking $\|\cdot\|_2$,
$\Bigl\|\sum_{t=1}^n g_t\Bigr\|_2
  \le O(n^\gamma)\cdot O(n^{-\gamma/2})
    + \sum_{t=1}^{n-1}O(t^{\gamma-1})\cdot O(t^{-\gamma/2})
    + O(n^{\alpha/2})
  = O(n^{\gamma/2}) + O(n^{\alpha/2})$.
Dividing by $n$ and multiplying by $\sqrt n$,
\begin{equation}
\label{eq:gbar-vanish}
  \bigl\|\sqrt n\,\overline g_n\bigr\|_2
  \;=\; O(n^{(\gamma-1)/2}) + O(n^{(\alpha-1)/2}) \;\to\; 0
  \quad\text{for }\alpha<\gamma<1.
\end{equation}

\emph{Combining.}  It remains to identify the two blocks of $R_n^z$ with
the averages just bounded.  Write $z_t = (\Delta_t, m_{t-1})$, so that
$[\overline z_n]_1 = \overline\Delta_n$ and $[\overline z_n]_2 =
n^{-1}\sum_{t=1}^n m_{t-1} = \overline m_{n-1}$ (using $m_0 = 0$).

\emph{Iterate block.}  Averaging the gradient identity $g_t = H\Delta_t
+ u_t + \xi_t$ over $t = 1,\dots,n$ gives $\overline g_n = H\overline
\Delta_n + \overline u_n + \overline\xi_n$, that is,
\begin{equation}
\label{eq:Delta-grad}
  \overline\Delta_n = H^{-1}\overline g_n - H^{-1}\overline u_n
    - H^{-1}\overline\xi_n .
\end{equation}
On the other hand, the first block of the
decomposition~\eqref{eq:pr-decomposition}, together with the exact
identity~\eqref{eq:AtBt}---whose iterate row is the \emph{constant}
$-H^{-1}$, so $[\overline{A_tB_t\xi_t}]_1 = -H^{-1}\overline\xi_n$ and
$[\overline{A_tB_t u_t}]_1 = -H^{-1}\overline u_n$---reads
$\overline\Delta_n = R_{n,1}^z - H^{-1}\overline\xi_n
    - H^{-1}\overline u_n$.
Subtracting this from~\eqref{eq:Delta-grad}, the noise and Taylor
averages cancel \emph{exactly}---each enters both expressions with the
same coefficient $-H^{-1}$---leaving
$R_{n,1}^z = H^{-1}\,\overline g_n$,
so $\sqrt n\,R_{n,1}^z = H^{-1}\,\sqrt n\,\overline g_n \to 0$ in $L^2$
by~\eqref{eq:gbar-vanish}.

\emph{Buffer block.}  The second row of $A_tB_t$ is exactly $0$
by~\eqref{eq:AtBt}, so $[\overline{A_tB_t\xi_t}]_2 =
[\overline{A_tB_t u_t}]_2 = 0$ and the second block
of~\eqref{eq:pr-decomposition} collapses to $R_{n,2}^z = [\overline z_n]_2
= \overline m_{n-1}$ exactly; hence $\sqrt n\,R_{n,2}^z \to 0$ in $L^2$
by~\eqref{eq:mbar-vanish}.  Combining the two blocks yields
$\|\sqrt n\, R_n^z\|_2 \to 0$, equivalently~\eqref{eq:Rnz-target}.
\end{proof}

\begin{remark}[Why $\gamma < 1$ is necessary]
\label{rem:Rnz-gamma-necessary}
The bound $\|R_n^z\|_{L^2} = o(n^{-1/2})$ is sharp at $\gamma = 1$: at
the boundary, the endpoint-term scaling $\sqrt n\cdot
n^{\gamma-1}\|m_n\|_{L^2}/c_1$ produces $\sqrt n\cdot
n^{0}\cdot O(n^{-1/2}) = O(1)$, not $o(1)$.  Concretely, take the scalar
model with $H = 1$ and noise variance $\sigma^2$.  For $c_1 >
(1+\alpha)/2$ the momentum second moment $v_t := \mathbb E[m_t^2]$ obeys
the balance
$v_{t+1} = \Bigl(1 - \tfrac{2c_1 - \alpha}{t} + o(t^{-1})\Bigr) v_t
    + \tfrac{c_1^2\sigma^2}{t^2} + o(t^{-2})$,
so $n\,v_n \to c_1^2\sigma^2/(2c_1 - 1 - \alpha)$ and the coupled
averaged-iterate balance gives $n\,\mathbb E[\overline\Delta_n^2] \to
\sigma^2\bigl(1 + 1/(2c_1 - 1 - \alpha)\bigr)$, \emph{strictly larger}
than the sandwich baseline $\sigma^2 = H^{-1}SH^{-1}$; at $c_1 =
(1+\alpha)/2$ the limit grows logarithmically and below it the boundary
term is larger still, so the sandwich fails throughout $\gamma = 1$.  The
schedule
$\rho_t = c_1/t^\gamma$ with $\gamma\in(\alpha, 1)$ is the right
parametric choice: $\gamma > \alpha$ preserves two-time-scale
separation (Appendix~\ref{app:lyapunov}), $\gamma < 1$ delivers the
sandwich limit.  The canonical Polyak--Ruppert choice $\gamma = 1$ is
just outside the admissible range.
\end{remark}

\subsection{Martingale CLT for the Leading Term}

The leading term of the decomposition~\eqref{eq:pr-decomposition} is a
martingale sum with the constant coefficient~\eqref{eq:AtBt}; we now establish
its central limit theorem.

\begin{lemma}[Martingale CLT]
\label{lem:mart-clt}
Under Assumptions~\ref{ass:mg}--\ref{ass:convex}, bounded gradients, the
conditional-covariance continuity at $x^*$ of
Theorem~\ref{thm:augmented-clt}(a), and the iterate bound
$\mathbb E\|\Delta_t\|^2 = O(t^{-\alpha})$ of~\eqref{eq:clt-stability},
\[
  \sqrt n\, \overline{A_t B_t \xi_t} \;\xrightarrow{d}\;
  \mathcal N\bigl(0,\, \Sigma_z\bigr),
  \qquad
  \Sigma_z = \begin{pmatrix} H^{-1} S H^{-1} & 0 \\ 0 & 0 \end{pmatrix},
\]
which coincides with the closed form $\Sigma_z = L^{-1}\Sigma_w L^{-\top}$
of Theorem~\ref{thm:projection} (in the raw coordinates of $z_t$).
\end{lemma}

\begin{proof}
By the exact identity~\eqref{eq:AtBt}, $A_t B_t = (-H^{-1},0)^\top =: C$ is
\emph{constant}, so $\{n^{-1/2}C\xi_t\}_{t=1}^n$ is a martingale-difference
array ($\mathbb E[\xi_t \mid \mathcal F_{t-1}] = 0$, with $C$ deterministic).
We apply the martingale CLT of
\citet[Cor.~3.1]{halldavid1980martingale} to each scalar projection
$v^\top C\xi_t$ and pass to $\mathbb R^{2d}$ by the Cram\'er--Wold device.
\emph{Conditional Lindeberg:} bounded gradients give
$\|C\xi_t\| \le 2G\|H^{-1}\|_{\mathrm{op}} =: K$ a.s., so the event
$\{\|C\xi_t\| > \epsilon\sqrt n\}$ is empty once $n > K^2/\epsilon^2$ and the
conditional Lindeberg sum vanishes.  \emph{Conditional variance:}
$\mathbb E\|\Delta_t\|^2 = O(t^{-\alpha}) \to 0$ gives $x_t \to x^*$ in
probability, so by continuity of $S(\cdot)$ at $x^*$
(Theorem~\ref{thm:augmented-clt}(a)) and the uniform bound
$\|S(x_t)\|_{\mathrm{op}} \le 4G^2$, bounded convergence and Ces\`aro averaging
yield $\frac1n\sum_{t=1}^n \mathbb E[\xi_t\xi_t^\top \mid \mathcal F_{t-1}]
\xrightarrow{p} S$; hence the array's conditional covariance converges to
$C S C^\top = (-H^{-1},0)^\top S (-H^{-1},0) =
\mathrm{diag}(H^{-1}SH^{-1}, 0) = \Sigma_z$.  The two conditions give
$\sqrt n\,\overline{A_t B_t \xi_t} = n^{-1/2}\sum_{t=1}^n C\xi_t
\xrightarrow{d} \mathcal N(0, \Sigma_z)$.
\end{proof}

\begin{remark}[Degenerate buffer marginal in raw coordinates]
\label{rem:raw-vs-rescaled}
The buffer block of $\Sigma_z$ is degenerate: in raw coordinates the
averaged buffer $\overline m_{n-1}$ has no $\sqrt n$-scale Gaussian
fluctuation, because the buffer-block coefficient in~\eqref{eq:AtBt} is
exactly $0$.  The buffer does fluctuate, but only under a coarser,
buffer-magnifying normalization (reflecting the anisotropic $\tau_t^{-1}$
weighting of~\eqref{eq:Qsym-bounds}); this rescaled buffer marginal is
not used in Theorem~\ref{thm:main}, where only the iterate $(1,1)$ block
of $\Sigma_z$ enters.
\end{remark}

\subsection{Taylor Remainder}

It remains to control the second-order Taylor term, which is asymptotically
negligible at the $\sqrt n$ scale.

\begin{lemma}[Taylor-remainder negligibility]
\label{lem:taylor}
Under the conditions of Theorem~\ref{thm:augmented-clt},
$\sqrt n\, \overline{A_t B_t u_t} = o_p(1)$.
\end{lemma}

\begin{proof}
This is the standard Polyak--Ruppert Taylor-remainder estimate
\citep{polyak1992acceleration}, in the time-varying preconditioner setting of
\citet{anhuo2026}; the only chain-specific input is the uniform bound
$\|A_t B_t\|_{\mathrm{op}} = \|H^{-1}\|_{\mathrm{op}} = O(1)$
from~\eqref{eq:AtBt}.  By Assumption~\ref{ass:quadratic}
$\|u_t\| \le L_R \|\Delta_t\|^2$, so $\mathbb E\|u_t\| = O(t^{-\alpha})$
by~\eqref{eq:clt-stability}, whence
$\sqrt n\,\mathbb E\bigl\|\overline{A_t B_t u_t}\bigr\| \le
  n^{-1/2}\sum_{t=1}^n \mathbb E\|A_t B_t u_t\| = O(n^{1/2-\alpha}) \to 0$
for $\alpha > 1/2$; $L^1$-convergence gives $\sqrt n\,\overline{A_t B_t u_t}
= o_p(1)$.
\end{proof}

\subsection{Combining the Pieces}

We now assemble the three preceding estimates---the martingale CLT, the
negligible Taylor term, and the vanishing remainder---into the joint CLT.

\begin{proof}[Completion of the proof of Theorem~\ref{thm:augmented-clt}]
Combining the decomposition~\eqref{eq:pr-decomposition} with
Lemmas~\ref{lem:Rnz-bound}, \ref{lem:mart-clt}, and~\ref{lem:taylor}: the
martingale term converges to $\mathcal N(0, \Sigma_z)$ with
$\Sigma_z = \mathrm{diag}(H^{-1}SH^{-1}, 0)$, while
$\sqrt n\,\overline{A_t B_t u_t} = o_p(1)$ and
$\sqrt n\, R_n^z = o_{L^2}(1) = o_p(1)$, so Slutsky's lemma
\citep[Lemma~2.8]{vaart1998asymptotic} gives $\sqrt n\,\overline z_n
\xrightarrow{d} \mathcal N(0, \Sigma_z)$.  This $\Sigma_z$ is, for
\emph{every} $t$, the frozen-chain Polyak--Ruppert covariance of
Theorem~\ref{thm:projection}: with $\Sigma_w(t) = B_t S B_t^\top/\eta_t^2$
and $A_t B_t = L_t^{-1}B_t/\eta_t$, the exact identity~\eqref{eq:AtBt} gives
$L_t^{-1}\Sigma_w(t)L_t^{-\top} = (A_t B_t)\,S\,(A_t B_t)^\top =
\mathrm{diag}(H^{-1}SH^{-1}, 0)$---so the limiting and frozen-chain
covariances are literally the same $t$-independent matrix, with no
reconciliation needed.  Taking the iterate marginal,
$\sqrt n\,\overline\Delta_n \xrightarrow{d} \mathcal N(0, H^{-1}SH^{-1})$,
as Theorem~\ref{thm:augmented-clt} asserts; the degenerate buffer block is discussed
in Remark~\ref{rem:raw-vs-rescaled}.
\end{proof}

\end{document}